\documentclass{article}

\usepackage{graphicx} % Required for inserting images
\usepackage{mathtools}
\usepackage{bm}
\usepackage{amssymb}
\usepackage{amsthm}
\usepackage{amsmath}
\usepackage{todonotes}
\usepackage[left=25mm,right=25mm,top=30mm,bottom=30mm]{geometry}
\usepackage{cases}
\usepackage{stmaryrd}
\usepackage[svgnames]{xcolor}
\usepackage{url}

\newtheorem{theorem}{Theorem}[section]
\newtheorem{lemma}[theorem]{Lemma}
\newtheorem{proposition}[theorem]{Proposition}
\newtheorem{corollary}[theorem]{Corollary}
\theoremstyle{remark}
\newtheorem{remark}[theorem]{Remark}

\theoremstyle{definition}
\newtheorem{definition}[theorem]{Definition}

\newcommand{\supp}{\operatorname{supp}}

\newcommand{\abs}[1]{\left\lvert #1 \right\rvert}
\newcommand{\relmiddle}[1]{\mathrel{}\middle#1\mathrel{}}

\numberwithin{equation}{section}

\allowdisplaybreaks

\title{Well-posedness of the Langmuir film problem}
\makeatletter
\renewcommand\@date{{%
  \vspace{-\baselineskip}%
  \large\centering
  \begin{tabular}{@{}c@{}}
    Yoichiro Mori\textsuperscript{1} \\
    \normalsize y1mori@sas.upenn.edu
  \end{tabular}%
  \quad\quad
  \begin{tabular}{@{}c@{}}
    Shinya Okabe\textsuperscript{2} \\
    \normalsize shinya.okabe@tohoku.ac.jp
  \end{tabular}%
  \quad\quad
  \begin{tabular}{@{}c@{}}
    Koya Sakakibara\textsuperscript{3,4} \\
    \normalsize ksakaki@se.kanazawa-u.ac.jp
  \end{tabular}

  \bigskip

  \textsuperscript{1}Department of Mathematics, University of Pennsylvania\par
  \textsuperscript{2}Mathematical Institute, Tohoku University\par
  \textsuperscript{3}Faculty of Mathematics and Physics, Institute of Science and Engineering, Kanazawa University\par
  \textsuperscript{4}RIKEN iTHEMS
}}

\begin{document}

\maketitle

\begin{abstract}
    We analyze the inviscid Langmuir layer–Stokesian subfluid (ILLSS) model for two-phase Langmuir monolayers coupled to a Stokes flow in the underlying subfluid. Eliminating the bulk variables, we reformulate the coupled three-dimensional system as an evolution on the film involving the Dirichlet-to-Neumann (DtN) operator. We identify the Fourier symbol of the DtN operator and show it coincides with that of the fractional Laplacian, which yields an explicit Fourier-multiplier representation and allows construction of the corresponding fundamental solution. Using this representation we express the surface velocity as a convolution of the fundamental solution with the interfacial curvature forcing and analyze its normal limit to derive a boundary integral equation for the moving curve. Independently, exploiting the DtN representation we establish a curve-shortening identity: the interfacial perimeter decreases monotonically and its time derivative is controlled by $\dot{H}^{1/2}(\mathbb{R}^2)$-norm of the surface velocity. Building on the boundary integral equation, we prove local well-posedness via maximal $L^2$-regularity for quasilinear parabolic systems, employing a DeTurck-type reparametrization, and show equivalence with the original ILLSS system. Finally, we introduce a linearly implicit parametric finite-element scheme which captures experimentally observed relaxation dynamics.
\end{abstract}

\section{Introduction}
\label{sec:intro}

Langmuir monolayers are molecularly thin films formed at the air--water interface; they consist of amphiphilic molecules which self-assemble into ordered structures. Since the pioneering works of Langmuir~\cite{langmuir1917constitution} and Blodgett~\cite{blodgett1935films}, these systems have provided a canonical setting for the study of interfacial phenomena in soft condensed matter. Their relevance encompasses both biophysical problems—such as lipid phase behavior, membrane protein adsorption, and drug--membrane interactions~\cite{dynarowicz2024advantages,pereira2021recent,rojewska2021langmuir}—and technological applications including nanolithography, biosensing, and the fabrication of functional coatings via Langmuir--Blodgett transfer techniques~\cite{lin2024controllable}.

From the viewpoint of continuum mechanics, Langmuir monolayers exhibit nontrivial mesoscale dynamics. In two-phase regimes, domains of nearly uniform surface density are separated by sharp interfaces endowed with line tension, and the resulting interfacial motion is driven by an interplay of curvature forces, long-range dipolar interactions, and viscous dissipation in the subfluid. Experimental techniques such as Brewster-angle microscopy and fluorescence imaging permit the visualization of domain morphology and kinetics at micron length scales~\cite{henon1991microscope,olafsen2010experimental}, revealing morphologies that include circular droplets, stripes, labyrinthine patterns, and elongated (bola-shaped) domains.

To capture these phenomena, Alexander and collaborators proposed the inviscid Langmuir layer--Stokesian subfluid (ILLSS) model~\cite{alexander2007domain}. The model couples a sharp-interface evolution on the film to a three-dimensional Stokes flow in the subphase: the monolayer is modeled as an incompressible, inviscid two-dimensional medium while dissipation is produced exclusively by the Newtonian subfluid. The resulting mathematical formulation is a nonlocal free-boundary problem with hierarchical coupling across spatial dimensions (3D bulk, 2D surface, 1D interface).

Throughout this paper we denote the subfluid by
\[
    B=\mathbb{R}^2\times(-\infty,0),
\]
with planar boundary $\partial B \simeq \mathbb{R}^2$. The evolving interface is a smooth closed curve $\Gamma(t)\subset\partial B$. The governing equations read
\begin{subnumcases}{}
    -\triangle v + \nabla q = 0,\quad \nabla \cdot v = 0, & in $B$, \label{eq:Stokes_subfluid} \\
    v_3 = 0,\quad v_\parallel|_{\partial B} = u, & on $\partial B\setminus\Gamma$, \label{eq:Langmuir_incompressibility} \\
    \nabla p = -\partial_{x_3} v_\parallel|_{\partial B},\quad \nabla \cdot u = 0, & on  $\partial B \setminus \Gamma$, \label{eq:velocity_surface} \\
    \llbracket u\cdot\nu \rrbracket = 0,\quad \llbracket p \rrbracket = \kappa,\quad U = u \cdot \nu, & on $\Gamma$, \label{eq:jump_condition}
\end{subnumcases}
where $v=(v_1,v_2,v_3)^\top$ and $q$ denote the velocity and pressure in the subfluid, $u=(u_1,u_2)^\top$ and $p$ denote the surface velocity and surface pressure on $\partial B$, and $v_\parallel=(v_1,v_2)^\top$ is the tangential trace of $v$. The unit normal to $\Gamma$ pointing outward from the enclosed domain is denoted by $\nu$, the curvature of $\Gamma$ by $\kappa$, and the normal velocity of $\Gamma$ by $U$.
For a scalar field $f$ defined in a neighborhood of $\Gamma$ we adopt the convention
\[
    \llbracket f\rrbracket(x)\coloneqq \lim_{\varepsilon\searrow0}\bigl(f(x-\varepsilon\nu(x)) - f(x+\varepsilon\nu(x))\bigr),
    \qquad x\in\Gamma,
\]
so that $\llbracket f\rrbracket=f_--f_+$, where $f_\pm$ denote the traces from the interior/exterior sides of $\Gamma$. See Figure~\ref{fig:ILLSS} for the schematic setup of the ILLSS model.

\begin{figure}
    \centering
    \includegraphics[width=0.5\linewidth]{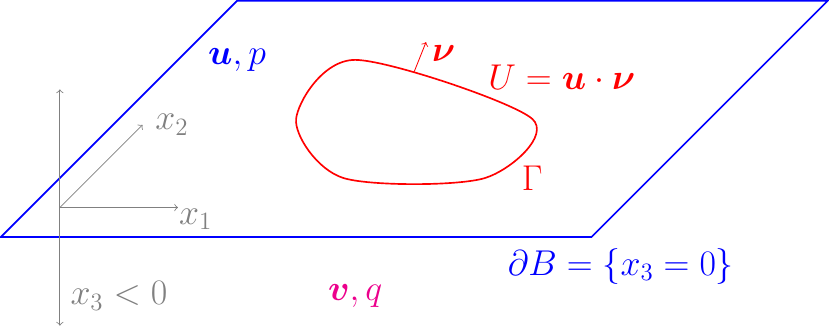}
    \caption{A schematic diagram of the ILLSS model.}
    \label{fig:ILLSS}
\end{figure}

We emphasize two significant differences between our formulation \eqref{eq:Stokes_subfluid}--\eqref{eq:jump_condition} and the original model proposed by Alexander et al.~\cite{alexander2007domain}.
First, regarding the boundary condition \eqref{eq:Langmuir_incompressibility}, the original work implicitly assumed the equality $v_\parallel = u$ on the entire boundary $\partial B$. In contrast, we explicitly restrict this condition to $\partial B \setminus \Gamma$, acknowledging the potential singularity of the velocity field at the interface $\Gamma$.
Second, and more importantly, regarding the jump condition \eqref{eq:jump_condition}, the original model imposed the continuity of the full velocity vector, i.e., $\llbracket u \rrbracket = 0$. However, since the Langmuir film is modeled as an inviscid fluid, the tangential velocity should satisfy a slip condition rather than continuity. Therefore, we adopt the condition $\llbracket u \cdot \nu \rrbracket = 0$, which allows for a discontinuity in the tangential component (i.e., a vortex sheet) across the interface. This modification is crucial for the physical consistency of the inviscid model.

Although the ILLSS model is concise in its formulation, it has proven effective in reproducing experimentally observed relaxation dynamics. Nevertheless, rigorous mathematical analysis of the ILLSS system---and, more generally, of problems featuring hierarchical coupling between bulk, surface, and interfacial dynamics---remains largely undeveloped. Establishing a firm analytical foundation for such models is therefore both mathematically interesting and relevant for a systematic understanding of interfacial hydrodynamics.

To clarify the mathematical positioning of the ILLSS model, it is instructive to compare it with two well-studied free boundary problems governing the motion of fluid interfaces: the classical Hele-Shaw problem and the quasistationary Stokes problem. In what follows, let $\Omega_-$ denote the bounded interior domain enclosed by the moving interface $\Gamma$, and let $\Omega_+$ denote the unbounded exterior domain. These models can be categorized based on the differential order of the equations governing the velocity field and the resulting geometric evolution of the interface.

First, the classical one-phase Hele-Shaw problem is governed by Darcy's law:
\begin{equation*}
    \begin{dcases*}
        u + \nabla p = 0,\quad \nabla\cdot u = 0,&in $\Omega_-$,\\
        p = \kappa&on $\Gamma$.
    \end{dcases*}
\end{equation*}
Here, the velocity equation $u = -\nabla p$ involves no derivatives of the velocity field (a 0-th order problem). The normal velocity of the interface is given by $U = u \cdot \nu = -\partial_\nu p$. Using the Dirichlet-to-Neumann operator $\Lambda_{\mathrm{DN}}$ associated with the harmonic pressure, this can be expressed as $U = -\Lambda_{\mathrm{DN}} \kappa$. Since $\Lambda_{\mathrm{DN}}$ is a first-order operator and the curvature $\kappa$ involves second-order spatial derivatives, the resulting geometric evolution is of the third order, known as the Mullins--Sekerka flow \cite{chen1993hele, constantin1993global, escher1997classical}.

On the other hand, the quasistationary Stokes problem (often referred to as viscous sintering) is governed by:
\begin{equation*}
    \begin{dcases*}
        -\triangle u + \nabla p = 0, \quad \nabla \cdot u = 0,&in $\Omega_-$,\\
        (\nabla u + (\nabla u)^\top - pI)\nu = \kappa\nu&on $\Gamma$.
    \end{dcases*}
\end{equation*}
Here, the momentum equation involves the Laplacian of the velocity (a 2nd-order problem). The interface velocity $U$ is determined by the balance of viscous stress and surface tension. In terms of mapping properties, the solution operator maps the curvature force to velocity roughly as the inverse of the Laplacian (modulo boundary traces), i.e., $U \approx \Lambda_{\mathrm{DN}}^{-1} \kappa$. Consequently, the interface evolution behaves like a first-order parabolic equation \cite{hopper1990plane, prokert1995existence}.

The ILLSS model discussed in this paper lies mathematically between these two regimes. As we will see in our reformulation (specifically \eqref{eq:Langmuir-DtN}), the system is governed essentially by:
\begin{equation*}
    \begin{dcases*}
        \Lambda_{\mathrm{DN}} u + \nabla p = 0, \quad \nabla \cdot u = 0,&in $\Omega_\pm$,\\
        \llbracket u \cdot \nu \rrbracket = 0,\quad\llbracket p \rrbracket = \kappa&on $\Gamma$.
    \end{dcases*}
\end{equation*}
The momentum equation involves the DtN operator (or fractional Laplacian $(-\triangle)^{1/2}$ \cite{caffarelli2007extension}), making it a 1st-order problem for the velocity. The resulting normal velocity takes the form $U = c\,\kappa + \text{nonlocal terms}$. Since the leading term is the curvature itself, the geometric evolution is of the second order. Structurally, this resembles the area-preserving curve shortening flow (or volume-preserving mean curvature flow) \cite{gage1986area, huisken1987volume, escher1998volume}, where the motion is driven by curvature but constrained by a nonlocal term arising from global conservation laws.

Thus, the ILLSS model fills the gap between the Hele-Shaw problem and the Stokes problem, offering a rich mathematical structure that exhibits properties of both local curvature flow and nonlocal hydrodynamic interactions.

The principal contributions of the present work are as follows.
\begin{enumerate}
    \item \textbf{Derivation of the boundary integral equation:}
    We rigorously reduce the coupled 3D-2D-1D system to a 1D geometric evolution equation on the interface $\Gamma$. This is achieved by identifying the Dirichlet-to-Neumann (DtN) operator for the subfluid Stokes flow with the fractional Laplacian $(-\Delta)^{1/2}$. We construct the explicit fundamental solution $E$ associated with this operator and represent the surface velocity as a convolution of $E$ with the interfacial curvature forcing. Analyzing the normal limits of this representation, we derive a closed boundary integral equation governing the motion of the curve.

    \item \textbf{Curve-shortening property:}
    Exploiting the Fourier multiplier representation of the DtN operator, we establish a curve-shortening identity. We prove that the rate of change of the interfacial perimeter is not positive and is precisely controlled by the $\dot H^{1/2}(\mathbb{R}^2)$-seminorm of the surface velocity. This provides a quantitative link between geometric relaxation and fractional hydrodynamic dissipation.

    \item \textbf{Local well-posedness via maximal regularity:}
    We prove the local well-posedness of the derived boundary integral equation in the class of smooth curves. To overcome the tangential degeneracy inherent in geometric evolution equations, we employ a DeTurck-type reparametrization, transforming the system into a quasilinear parabolic equation. We then establish maximal $L^2$-regularity for the linearized operator with variable coefficients and apply the contraction mapping principle in the intersection space $\mathbb{E}(0,T) = H^1(0,T;H^1(\mathbb{S}^1)) \cap L^2(0,T;H^3(\mathbb{S}^1))$.

    \item \textbf{Instantaneous smoothing and equivalence:}
    We prove that the solution $\gamma(t)$ becomes instantly $C^\infty$ in both space and time for $t > 0$. The proof relies on a ``parameter trick'' (time scaling) combined with the implicit function theorem and a bootstrap argument for elliptic regularity. Based on this regularity and the uniqueness of the classical solution to the elliptic problem, we rigorously reconstruct the full solution $(v, q, u, p)$ of the original ILLSS model from the solution of the boundary integral equation, thereby establishing the equivalence between the two formulations.

    \item \textbf{Parametric finite-element scheme:}
    We introduce a linearly implicit parametric finite-element scheme for the boundary integral equation. We demonstrate that the numerical results robustly capture the relaxation dynamics observed experimentally, including the evolution of complex bola-shaped domains, and are consistent with the geometric properties derived in our analysis.
\end{enumerate}

The logical structure of this paper establishes the equivalence between the ILLSS model and the derived boundary integral equation as follows.
In Section~\ref{sec:uniqueness}, we establish the uniqueness of the classical solution to the stationary elliptic problem governing the bulk and surface flow. This ensures that if a solution exists, it is unique.
In Section~\ref{sec:ILLSS-to-BIE}, we explicitly construct a candidate solution using the fundamental solution of the DtN operator and derive the boundary integral equation.
Section~\ref{sec:preliminaries} provides the necessary analytical preliminaries.
Section~\ref{sec:well-posedness} establishes the local well-posedness and, crucially, the $C^\infty$-regularity of the solution to the boundary integral equation.
Finally, in Subsection~\ref{subsec:equivalence}, we close the logical loop: the smoothness of the boundary integral equation solution guarantees that our constructed candidate satisfies the regularity requirements of the classical solution defined in Section~\ref{sec:uniqueness}, thereby establishing the strict equivalence between the boundary integral equation and the original ILLSS model via the uniqueness theorem.
Section~\ref{sec:numerics} presents numerical experiments.

\section{Uniqueness of the solution to elliptic problems}
\label{sec:uniqueness}

Given a Jordan curve $\Gamma \subset \partial B$ of class $C^2$, we consider the following elliptic problem:
\begin{subnumcases}{}
    -\triangle v + \nabla q = 0,\quad \nabla \cdot v = 0,& in $B$,\label{eq:Stokes}\\
    v_3 = 0,\quad v_\parallel = u,& in $\partial B \setminus \Gamma$,\label{eq:DBC}\\
    \nabla p = -\partial_{x_3}v_\parallel,\quad \nabla \cdot u = 0,& in $\partial B \setminus \Gamma$,\label{eq:Langmuir}\\
    \llbracket u \cdot \nu \rrbracket = 0,\quad \llbracket p \rrbracket = g,& on $\Gamma$,\label{eq:jump}
\end{subnumcases}
where $g$ is a given pressure jump along $\Gamma$. In this section, we establish the uniqueness of classical solutions to the system \eqref{eq:Stokes}--\eqref{eq:jump}.

\subsection{Definition of classical solution}

We first define the notion of a classical solution. We require sufficient regularity for the equations to be satisfied pointwise within each subdomain. To accommodate the jump conditions on the interface, in particular the slip condition, we allow the surface fields $u$ and $p$ to be discontinuous across $\Gamma$, requiring only that they be continuously differentiable up to the boundary $\Gamma$ from each side ($\Omega_-$ and $\Omega_+$). Furthermore, specific decay conditions at infinity are imposed to ensure the validity of integration by parts.

\begin{definition}
    \label{def:classical}
    Let $g \in C^0(\Gamma)$ be a given scalar function representing the magnitude of the normal stress jump. We assume that $g$ satisfies the zero net force condition:
    \begin{align*}
        \int_\Gamma g(y)\nu(y) \,\mathrm{d}s(y) = 0.
    \end{align*}
    A quadruplet $(v,q,u,p)$ is called a \textit{classical solution} to the problem~\eqref{eq:Stokes}--\eqref{eq:jump} if it satisfies the following conditions:
    \begin{enumerate}
        \item \textbf{Regularity.} The functions possess the following smoothness properties:
        \begin{alignat*}{2}
            v &\in C^2(B;\mathbb{R}^3) \cap C^1(\overline{B}\setminus\Gamma; \mathbb{R}^3), &\qquad
            q &\in C^1(B)\cap C^0(\overline{B}\setminus\Gamma), \\
            u &\in C^1(\overline{\Omega_-};\mathbb{R}^2) \cap C^1(\overline{\Omega_+};\mathbb{R}^2), &\qquad
            p &\in C^1(\overline{\Omega_-}) \cap C^1(\overline{\Omega_+}).
        \end{alignat*}
        
        \item \textbf{Decay.} The functions satisfy the following decay rates as $|x|\to\infty$:
        \begin{align*}
            &|v(x)|=\mathcal{O}(|x|^{-2}),\quad
            |\nabla v(x)|=\mathcal{O}(|x|^{-3}),\quad
            |q(x)|=\mathcal{O}(|x|^{-3}),\quad
            |\nabla q(x)|=\mathcal{O}(|x|^{-4}),\\
            &|u(x)|=\mathcal{O}(|x|^{-2}),\quad
            |p(x)|=\mathcal{O}(|x|^{-2}),
        \end{align*}
        uniformly for $x\in B$ (bulk fields) or $x\in \partial B$ (surface fields).
        
        \item \textbf{Finite Energy.} The bulk velocity and pressure possess finite Dirichlet integrals:
        \begin{align*}
            \int_B|\nabla v(x)|^2\,\mathrm{d}x<\infty,\qquad
            \int_B|\nabla q(x)|^2\,\mathrm{d}x<\infty.
        \end{align*}
        
        \item \textbf{Pointwise Satisfaction.} The equations~\eqref{eq:Stokes}--\eqref{eq:Langmuir} are satisfied pointwise in $B$ and $\partial B\setminus\Gamma$, respectively.
        The jump conditions~\eqref{eq:jump} on $\Gamma$ are satisfied in the sense of the one-sided limits guaranteed by (i).
    \end{enumerate}
\end{definition}

\subsection{Distributional formulation of surface equations}

In this subsection, we establish that the piecewise classical solution satisfies the governing equations in the sense of distributions on the entire boundary $\partial B$. Specifically, we show that the jump conditions across the interface $\Gamma$ are consistently captured by singular source terms involving the Dirac measure. This global distributional formulation provides the rigorous basis for the integration by parts arguments (generalized Green's formulas) employed in the subsequent uniqueness proof.

\begin{lemma}
    \label{lem:distributional}
    Let $(v,q,u,p)$ be a classical solution to \eqref{eq:Stokes}--\eqref{eq:jump}. Then, the surface equations hold in the sense of distributions on $\partial B$:
    \begin{align*}
        \nabla\cdot u = 0,\qquad
        \nabla p+\partial_{x_3}v_\parallel = -g\nu\delta_\Gamma,
    \end{align*}
    where $\delta_\Gamma$ is the Dirac measure supported on $\Gamma$.
\end{lemma}

\begin{proof}
    Let $\phi \in C_c^\infty(\partial B)$ be a scalar test function. We compute the distributional divergence:
    \begin{align*}
        \langle\nabla\cdot u,\phi\rangle \coloneqq -\int_{\partial B}u \cdot \nabla\phi\,\mathrm{d}x'.
    \end{align*}
    Since $u$ is continuously differentiable on the closure of each subdomain $\overline{\Omega_\pm}$, we can split the integral and apply the classical Gauss--Green theorem on each subdomain.
    Recalling that the unit normal vector on $\Gamma$ points from $\Omega_-$ to $\Omega_+$, we have
    \begin{align*}
        -\int_{\Omega_-}u\cdot\nabla\phi\,\mathrm{d}x'
        &= \int_{\Omega_-}(\nabla\cdot u)\phi\,\mathrm{d}x' - \int_\Gamma(u_-\cdot\nu)\phi\,\mathrm{d}s, \\
        -\int_{\Omega_+}u\cdot\nabla\phi\,\mathrm{d}x' 
        &= \int_{\Omega_+}(\nabla\cdot u)\phi\,\mathrm{d}x' + \int_\Gamma(u_+\cdot\nu)\phi\,\mathrm{d}s.
    \end{align*}
    Using the pointwise condition $\nabla\cdot u = 0$ in $\Omega_\pm$ and summing these contributions, we obtain
    \begin{align*}
        \langle\nabla\cdot u,\phi\rangle
        = -\int_\Gamma\left(u_-\cdot\nu - u_+\cdot\nu\right)\phi\,\mathrm{d}s
        = -\int_\Gamma\llbracket u\cdot\nu \rrbracket\phi\,\mathrm{d}s.
    \end{align*}
    The jump condition $\llbracket u\cdot\nu \rrbracket = 0$ implies $\langle\nabla\cdot u,\phi\rangle = 0$.

    For the momentum equation, let $\psi \in C_c^\infty(\partial B;\mathbb{R}^2)$ be a vector test function.
    Applying integration by parts on each subdomain and using the relations $\nabla p = -\partial_{x_3}v_\parallel$ in $\Omega_\pm$ and $\llbracket p \rrbracket = g$ on $\Gamma$, we obtain
    \begin{align*}
        \langle\nabla p,\psi\rangle
        &\coloneqq -\int_{\partial B}p(\nabla\cdot\psi)\,\mathrm{d}x' \\
        &= \int_{\Omega_-\cup\Omega_+}\nabla p\cdot\psi\,\mathrm{d}x' - \int_\Gamma \llbracket p \rrbracket (\nu \cdot \psi)\,\mathrm{d}s \\
        &= \int_{\Omega_-\cup\Omega_+} (-\partial_{x_3}v_\parallel) \cdot \psi\,\mathrm{d}x' - \int_\Gamma g(\nu \cdot \psi)\,\mathrm{d}s.
    \end{align*}

    We now clarify the meaning of the first term on the right-hand side.
    First, we verify $v \in H^1(B)$. 
    The decay condition $|v(x)| = \mathcal{O}(|x|^{-2})$ ensures $v \in L^2(B \cap \{|x|>R\})$ for large $R$.
    In the bounded region, the condition $\nabla v \in L^2(B)$ combined with the Sobolev embedding implies $v \in L^2(B \cap \{|x|\le R\})$. Thus $v \in L^2(B)$, and consequently $v \in H^1(B)$.
    
    Since $-\triangle v = -\nabla q \in L^2(B)$, the velocity $v$ belongs to the space $D_\triangle(B) \coloneqq \{w \in H^1(B) \mid \triangle w \in L^2(B)\}$.
    By the trace theory of Lions and Magenes~\cite[Theorem~6.5 in Chapter~2, Section~6.5]{lions1972non}, the generalized normal derivative $\partial_{x_3}v$ is uniquely defined as an element of the dual space $H^{-1/2}(\partial B)$.
    The term involving $\partial_{x_3}v_\parallel$ in the distributional formula should be understood as the duality pairing:
    \begin{align*}
        \int_{\Omega_-\cup\Omega_+} (-\partial_{x_3}v_\parallel) \cdot \psi\,\mathrm{d}x' = \langle -\partial_{x_3}v_\parallel, \psi \rangle_{H^{-1/2}, H^{1/2}}.
    \end{align*}
    Crucially, the Dirac measure $\delta_\Gamma$ supported on the curve $\Gamma$ belongs to $H^{-1}(\partial B)$ but not to $H^{-1/2}(\partial B)$ (as the trace of an $H^1$-function onto a line requires regularity $s > 1/2$).
    Since $\partial_{x_3}v_\parallel \in H^{-1/2}(\partial B)$, it cannot contain a singular measure supported on $\Gamma$.
    Therefore, the singular term in the gradient of pressure arises solely from the pressure jump:
    \begin{align*}
        \langle \nabla p, \psi \rangle = \langle -\partial_{x_3}v_\parallel, \psi \rangle_{H^{-1/2}, H^{1/2}} - \langle g\nu\delta_\Gamma, \psi \rangle.
    \end{align*}
    Thus, $\nabla p = -\partial_{x_3}v_\parallel - g\nu\delta_\Gamma$ in $\mathcal{D}'(\partial B)$.
\end{proof}

\subsection{Uniqueness theorem}

With the distributional formulation established in Lemma~\ref{lem:distributional}, we are now ready to prove the uniqueness theorem.

\begin{theorem}
    \label{thm:uniqueness}
    The classical solution to the problem \eqref{eq:Stokes}--\eqref{eq:jump} is unique.
\end{theorem}

\begin{proof}
    Let $(V,Q,U,P)$ be the difference of two classical solutions. Since the problem is linear, this difference satisfies the homogeneous system (i.e., with $g = 0$). We proceed to show that the energy norm of $V$ vanishes.

    \textit{Step 1: Function spaces and generalized Green's formula.}
    First, we verify the functional setting. By the decay and finite energy conditions in Definition~\ref{def:classical}, the same arguments as in Lemma~\ref{lem:distributional} ensure that the bulk fields satisfy $V \in H^1(B; \mathbb{R}^3)$ and $Q \in H^1(B)$, while the surface fields satisfy $U \in L^2(\partial B; \mathbb{R}^2)$ and $P \in L^2(\partial B)$.

    Since $-\triangle V = -\nabla Q \in L^2(B; \mathbb{R}^3)$, $V$ belongs to the space $D_\triangle(B)$. According to Lions and Magenes~\cite{lions1972non}, the generalized normal derivative $\partial_{n}V = \partial_{x_3}V|_{\partial B}$ is uniquely defined in the dual space $H^{-1/2}(\partial B; \mathbb{R}^3)$, and the following generalized Green's formula holds:
    \begin{align}\label{eq:green_id}
        \int_B|\nabla V|^2\,\mathrm{d}x = -\int_B\triangle V\cdot V\,\mathrm{d}x + \langle\partial_{x_3}V, V\rangle_{H^{-1/2},H^{1/2}}.
    \end{align}
    Note that for the difference solution, the jump conditions imply that no singular measures (like $\delta_\Gamma$) appear in the derivative; thus $\partial_{x_3}V$ is a regular distribution in the sense discussed in Lemma~\ref{lem:distributional}.

    \textit{Step 2: Vanishing of the bulk term.}
    We analyze the term $-\int_B\triangle V\cdot V\,\mathrm{d}x$. Substituting $\triangle V = \nabla Q$, we consider $-\int_B\nabla Q \cdot V\,\mathrm{d}x$.
    Since $Q, V \in H^1(B)$, their traces belong to $H^{1/2}(\partial B) \subset L^2(\partial B)$. Thus, we can apply the standard integration by parts formula:
    \begin{align*}
        -\int_B\nabla Q\cdot V\,\mathrm{d}x
        = \int_BQ(\nabla\cdot V)\,\mathrm{d}x - \int_{\partial B}Q(V\cdot n)\,\mathrm{d}x'.
    \end{align*}
    The first term vanishes because $\nabla \cdot V = 0$. For the boundary integral, the condition $V_3 = 0$ on $\partial B\setminus\Gamma$ implies $V \cdot n = V_3 = 0$ almost everywhere on $\partial B$ (since the curve $\Gamma$ has 2D Lebesgue measure zero). Consequently, the boundary integral vanishes, and thus the entire bulk term is zero.

    \textit{Step 3: Vanishing of the boundary term.}
    We evaluate the boundary pairing in \eqref{eq:green_id}.
    From Lemma~\ref{lem:distributional} (with $g=0$), we have the distributional relations $\nabla P + \partial_{x_3}V_\parallel = 0$ and $\nabla \cdot U = 0$ on $\partial B$.
    Since $V|_{\partial B} = (U, 0)$ almost everywhere, the duality pairing reduces to:
    \begin{align*}
        \langle \partial_{x_3}V, V \rangle_{H^{-1/2}, H^{1/2}} 
        = \langle \partial_{x_3}V_\parallel, U \rangle_{H^{-1/2}, H^{1/2}}.
    \end{align*}
    Substituting $\partial_{x_3}V_\parallel = -\nabla P$, we obtain
    \begin{align*}
        \langle \partial_{x_3}V_\parallel, U \rangle_{H^{-1/2}, H^{1/2}} = \langle -\nabla P, U \rangle_{H^{-1/2}, H^{1/2}}.
    \end{align*}
    Note that although $P \in L^2$ generally implies only $\nabla P \in H^{-1}$, the relation $\nabla P = -\partial_{x_3}V_\parallel$ combined with $V \in D_\triangle(B)$ improves the regularity to $\nabla P \in H^{-1/2}$. This ensures that the duality pairing is well-defined.
    By the definition of the distributional derivative (or duality),
    \begin{align*}
        \langle -\nabla P, U \rangle = \langle P, \nabla \cdot U \rangle = 0,
    \end{align*}
    since $\nabla \cdot U = 0$ globally on $\partial B$.

    \textit{Conclusion.}
    The energy identity \eqref{eq:green_id} reduces to $\int_B|\nabla V|^2\,\mathrm{d}x = 0$.
    Since $V$ vanishes at infinity, Poincare-type inequalities imply $V \equiv 0$ in $B$.
    Consequently, $U = V|_{\partial B} \equiv 0$, $\nabla Q = \Delta V \equiv 0$, and $\nabla P = -\partial_{x_3}V_\parallel \equiv 0$. The decay conditions then imply $Q \equiv 0$ and $P \equiv 0$.
\end{proof}

% 数理モデル説明パート
\section{Derivation of the Boundary Integral Equation}
\label{sec:ILLSS-to-BIE}

Having established the uniqueness of the classical solution in the previous section, we now turn our attention to the derivation of an explicit evolution equation for the interface $\Gamma$. To this end, we employ the Dirichlet-to-Neumann (DtN) operator and its fundamental solution to reformulate the bulk problem into a boundary integral equation.

Throughout this section, we assume that the solution and the interface possess sufficient regularity to justify the requisite computations, including integral representations and trace operations. This allows us to formally derive the governing equation for the moving curve. The rigorous mathematical analysis of the derived model, including its local well-posedness and the optimal regularity of solutions, will be provided in Section~\ref{sec:well-posedness}.

\subsection{Dirichlet-to-Neumann operator}

Let us consider the Stokes equation~\eqref{eq:Stokes_subfluid} in the subfluid region with the Dirichlet boundary condition~\eqref{eq:Langmuir_incompressibility}; that is,
\begin{equation}
    \begin{dcases*}
        -\triangle v+\nabla q=0,\quad\nabla\cdot v=0, & in $B$,\\
        v_3=0,\quad v_\parallel=u, & on $\partial B\setminus\Gamma$,
    \end{dcases*}
    \label{eq:Stokes_BVP}
\end{equation}
where the velocity field $u$ on the Langmuir film is given.

\begin{proposition}
\label{prop:Laplace_solves_Stokes}
Let $\Gamma\subset\partial B$ be a $C^2$ Jordan curve. 
Suppose that $u \in C^1(\overline{\Omega_-};\mathbb{R}^2) \cap C^1(\overline{\Omega_+};\mathbb{R}^2)$ satisfies the incompressibility condition $\nabla\cdot u=0$ in $\partial B\setminus\Gamma$.
Assume that $u$ satisfies the jump condition
\[
    \llbracket u\cdot\nu \rrbracket = 0\quad\text{on }\Gamma.
\]
Assume further that $u$ decays at infinity as $|u(x')| = \mathcal{O}(|x'|^{-2})$ as $|x'|\to\infty$.
Let $v\colon\overline{B}\rightarrow\mathbb{R}^3$ be the solution to the Dirichlet boundary value problem for the Laplace equation
\begin{equation}
    \begin{dcases*}
        \triangle v=0 & in $B$,\\
        v_3=0,\quad v_\parallel=u,&on $\partial B$.
    \end{dcases*}
    \label{eq:DP}
\end{equation}
Then, the vector field $v$ is solenoidal in $B$, and thus $(v, q)$ with some constant pressure $q$ solves the Stokes equation~\eqref{eq:Stokes_BVP}.
\end{proposition}

\begin{proof}
    Set $\tilde{u}\coloneqq(u_1,u_2,0)^\top$.
    Since $u$ is piecewise continuous and bounded on $\partial B$ (boundedness follows from the continuity up to $\Gamma$ and the decay at infinity), the solution $v$ to the Laplace equation~\eqref{eq:DP} is given by the Poisson integral formula for the half-space:
    \begin{equation*}
        v(x)=-\frac{x_3}{2\pi}\int_{\partial B}\frac{\tilde{u}(y)}{|x-y|^3}\,\mathrm{d}y.
    \end{equation*}
    (See, for instance, \cite[Theorem 14 (p.~37)]{evans2010partial}.)
    Although $u$ may have discontinuities across $\Gamma$, it is well-known that the Poisson integral of a bounded function yields a harmonic function in $B$ that attains the boundary values at every point of continuity (i.e., on $\partial B \setminus \Gamma$).

    We proceed to verify that this $v$ satisfies the incompressibility condition $\nabla\cdot v=0$ in $B$.
    Since $x \in B$ is away from the boundary, the kernel is smooth, and we can differentiate under the integral sign.
    Using the identity $\partial_{x_i}|x-y|^{-3}=-\partial_{y_i}|x-y|^{-3}$ for $i=1,2$, and the surface divergence condition $\nabla\cdot u=0$ in $\Omega_\pm$, we compute
    \begin{align*}
        \nabla\cdot v(x)
        &= -\frac{x_3}{2\pi}\int_{\partial B}\left(u_1(y)\partial_{x_1}|x-y|^{-3}+u_2(y)\partial_{x_2}|x-y|^{-3}\right)\,\mathrm{d}y\\
        &= \frac{x_3}{2\pi}\int_{\partial B}u(y)\cdot\nabla_{y}'|x-y|^{-3}\,\mathrm{d}y\\
        &= \frac{x_3}{2\pi}\left(\int_{\Omega_-}\nabla_y'\cdot\left(\frac{u(y)}{|x-y|^3}\right)\,\mathrm{d}y+\int_{\Omega_+}\nabla_y'\cdot\left(\frac{u(y)}{|x-y|^3}\right)\,\mathrm{d}y\right),
    \end{align*}
    where $\nabla_y'$ denotes the surface gradient with respect to $y$.
    Applying the Gauss divergence theorem to the integral over the bounded domain $\Omega_-$, we obtain
    \begin{align*}
        \int_{\Omega_-}\nabla_y'\cdot\left(\frac{u(y)}{|x-y|^3}\right)\,\mathrm{d}y = \int_{\Gamma}\frac{(u\cdot\nu)_-}{|x-y|^3}\,\mathrm{d}s(y),
    \end{align*}
    where $\nu$ is the unit normal to $\Gamma$ pointing outward from $\Omega_-$.

    For the exterior domain $\Omega_+$, we consider the truncated domain $\Omega_+^R \coloneqq \Omega_+\cap \{y \in \partial B \mid |y|<R\}$ with sufficiently large $R$.
    Applying the divergence theorem yields
    \begin{align*}
        \int_{\Omega_+^R}\nabla_y'\cdot\left(\frac{u(y)}{|x-y|^3}\right)\,\mathrm{d}y &= \int_{\partial D_R}\frac{u(y)\cdot n_R(y)}{|x-y|^3}\,\mathrm{d}s(y) - \int_\Gamma\frac{(u\cdot\nu)_+}{|x-y|^3}\,\mathrm{d}s(y),
    \end{align*}
    where $n_R(y) = y/|y|$ and the negative sign on $\Gamma$ arises because $\nu$ points into $\Omega_+$.
    We estimate the contribution from the outer boundary $\partial D_R$. For fixed $x \in B$ and sufficiently large $R$ (such that $|y|=R > 2|x|$), we have $|x-y| \ge R/2$.
    Using the decay condition $|u(y)| \le C R^{-2}$, we find
    \begin{align*}
        \left|\int_{\partial D_R}\frac{u(y)\cdot n_R(y)}{|x-y|^3}\,\mathrm{d}s(y)\right|
        \le \int_{\partial D_R} \frac{C R^{-2}}{(R/2)^3} \,\mathrm{d}s(y)
        = \frac{8C}{R^5} \cdot 2\pi R = \mathcal{O}(R^{-4}).
    \end{align*}
    Thus, this term vanishes as $R\to\infty$. Consequently,
    \begin{align*}
        \int_{\Omega_+}\nabla_y'\cdot\left(\frac{u(y)}{|x-y|^3}\right)\,\mathrm{d}y = -\int_\Gamma\frac{(u\cdot\nu)_+}{|x-y|^3}\,\mathrm{d}s(y).
    \end{align*}
    Summing the contributions from $\Omega_-$ and $\Omega_+$, we arrive at
    \begin{align*}
        \nabla\cdot v(x) 
        = \frac{x_3}{2\pi}\int_\Gamma \frac{(u\cdot\nu)_- - (u\cdot\nu)_+}{|x-y|^3}\,\mathrm{d}s(y)
        = \frac{x_3}{2\pi}\int_\Gamma \frac{\llbracket u\cdot\nu \rrbracket}{|x-y|^3}\,\mathrm{d}s(y).
    \end{align*}
    Since the jump $\llbracket u\cdot\nu \rrbracket$ vanishes on $\Gamma$ by assumption, the integral is zero, implying $\nabla\cdot v(x) = 0$ for all $x \in B$.
    Finally, since $v$ is harmonic ($\triangle v = 0$) and divergence-free ($\nabla \cdot v = 0$), the Stokes equation $-\triangle v + \nabla q = 0$ is satisfied with a constant pressure $q$ (since $\nabla q = \triangle v = 0$).
\end{proof}

By Proposition~\ref{prop:Laplace_solves_Stokes}, the velocity field $v$ in the subfluid is uniquely determined by its boundary value $u$ as the solution to the Laplace equation. This allows us to reformulate the problem solely in terms of surface quantities.
We introduce the Dirichlet-to-Neumann (DtN) operator $\Lambda_{\mathrm{DN}}\colon H^{1}(\partial B) \to H^{-1}(\partial B)$ associated with the harmonic extension in the lower half-space $B$:
\begin{align*}
    \Lambda_{\mathrm{DN}} u \coloneqq \partial_{x_3} v_\parallel|_{\partial B},
\end{align*}
where $v$ is the solution to \eqref{eq:DP}.
Using this operator, the term $-\partial_{x_3}v_\parallel|_{\partial B}$ in the momentum equation \eqref{eq:velocity_surface} is replaced by $-\Lambda_{\mathrm{DN}}u$. Consequently, the ILLSS model is recast as the following system on the interface and the boundary plane:
\begin{subnumcases}{\label{eq:Langmuir-DtN}}
    \nabla p = -\Lambda_{\mathrm{DN}}u, \quad \nabla\cdot u = 0 & in $\partial B\setminus\Gamma$, \label{eq:Langmuir-DtN-1}\\
    \llbracket u\cdot\nu\rrbracket=0,\quad\llbracket p\rrbracket=\kappa,\quad U=u\cdot\nu & on $\Gamma$. \label{eq:Langmuir-DtN-2}
\end{subnumcases}

\begin{remark}
\label{rem:DtN}
The DtN operator $\Lambda_{\mathrm{DN}}$ is identified with the fractional Laplace operator of order one-half, denoted by $(-\triangle)^{1/2}$ (see, e.g., \cite[Section~1]{caffarelli2007extension}).
In terms of the Fourier multiplier, this relation implies
\begin{equation}
    \mathcal{F}\Lambda_{\mathrm{DN}}u = |k|\mathcal{F}u,
    \label{eq:DtN-Fourier}
\end{equation}
where $k$ denotes the frequency variable in $\mathbb{R}^2$ \cite[Eq.~(7)]{lischke2020what}. Here, the two-dimensional Fourier transform $\mathcal{F}$ and its inverse $\mathcal{F}^{-1}$ are defined as
\begin{equation}
    \hat{u}(k) \coloneqq \int_{\mathbb{R}^2}u(x)\mathrm{e}^{-\mathrm{i}k\cdot x}\,\mathrm{d}x,\quad
    u(x) \coloneqq \frac{1}{(2\pi)^2}\int_{\mathbb{R}^2}\hat{u}(k)\mathrm{e}^{\mathrm{i}k\cdot x}\,\mathrm{d}k,
    \label{eq:Fourier}
\end{equation}
respectively. This spectral characterization \eqref{eq:DtN-Fourier} plays an essential role in the subsequent analysis, particularly in deriving the fundamental solution and proving the curve-shortening property.
\end{remark}

\subsection{The fundamental solution tensor to the DtN operator with incompressibility constraint}

To derive the explicit form of the solution $u$ to the system \eqref{eq:Langmuir-DtN-1}--\eqref{eq:Langmuir-DtN-2}, we construct the fundamental solution tensor for the DtN operator $\Lambda_{\mathrm{DN}}$ subject to the divergence-free constraint.

\begin{lemma}
    \label{lem:fund_sol}
    The fundamental solution tensor $E$ for the surface evolution equation is given by
    \begin{equation}
        E(x)=\frac{x\otimes x}{2\pi|x|^3},
        \label{eq:fundamental_explicit}
    \end{equation}
    where $a\otimes b\coloneqq(a_ib_j)_{i,j=1,2}\in\mathbb{R}^{2\times2}$ denotes the tensor product of vectors $a, b \in \mathbb{R}^2$.
\end{lemma}

\begin{proof}
    We seek the fundamental solution for the following system on the whole plane $\partial B \simeq \mathbb{R}^2$:
    \begin{equation*}
        \nabla p=-\Lambda_{\mathrm{DN}}u+f,\quad\nabla\cdot u=0.
    \end{equation*}
    As established in Lemma~\ref{lem:distributional}, the interface condition $\llbracket u \cdot \nu \rrbracket = 0$ ensures that the divergence-free constraint $\nabla \cdot u = 0$ holds globally on $\partial B$ in the sense of distributions. This justifies the application of the Fourier transform over the entire plane.

    Applying the Fourier transform and using the multiplier representation \eqref{eq:DtN-Fourier}, we obtain the algebraic system:
    \begin{equation}
        \mathrm{i}\hat{p}(k)k=-|k|\hat{u}(k)+\hat{f}(k),\quad
        \hat{u}(k)\cdot k=0,\qquad
        \text{for}\ k\in\mathbb{R}^2.
    \label{eq:Fourier-transform_Langmuir-DtN}
    \end{equation}
    Taking the inner product of the first equation with $k$ and using the orthogonality condition $\hat{u}\cdot k = 0$, we find
    \begin{equation*}
        \mathrm{i}\hat{p}(k)|k|^2=\hat{f}(k)\cdot k,\quad\text{which implies}\quad\hat{p}(k)=\frac{\hat{f}(k)\cdot k}{\mathrm{i}|k|^2}.
    \end{equation*}
    Substituting this expression for $\hat{p}$ back into the first equation of \eqref{eq:Fourier-transform_Langmuir-DtN}, we solve for $\hat{u}(k)$:
    \begin{equation*}
        |k|\hat{u}(k) = \hat{f}(k) - \mathrm{i}\hat{p}(k)k = \hat{f}(k) - \frac{\hat{f}(k)\cdot k}{|k|^2}k = \left(I-\frac{k\otimes k}{|k|^2}\right)\hat{f}(k).
    \end{equation*}
    Thus, the symbol of the fundamental solution tensor is
    \begin{equation*}
        \hat{E}(k)=\frac{1}{|k|}I-\frac{k\otimes k}{|k|^3}\qquad\text{for}\ k\in\mathbb{R}^2.
    \end{equation*}
    
    The fundamental solution $E$ is obtained by the inverse Fourier transform in the sense of tempered distributions.
    Recall that the Fourier transform of the Riesz potential kernel $|x|^{-1}$ in $\mathbb{R}^2$ is
    \begin{align*}
        \mathcal{F}\left[\frac{1}{|x|}\right](k) = \frac{2\pi}{|k|}.
    \end{align*}
    (See, e.g., Stein~\cite[Chapter~V, Section~1, Lemma~1]{stein1970singular}). By the inversion formula,
    \begin{align}
        \mathcal{F}^{-1}\left[\frac{1}{|k|}I\right](x) = \frac{1}{2\pi|x|}I.
    \label{eq:inv_Fourier_1}
    \end{align}

    Next, we compute the inverse Fourier transform of the term $k\otimes k/|k|^3$.
    Using the distributional identity $\triangle|x| = 1/|x|$ in $\mathbb{R}^2$, we have
    \begin{align*}
        \mathcal{F}[|x|](k) = -|k|^{-2} \mathcal{F}[\triangle|x|](k) = -|k|^{-2} \frac{2\pi}{|k|} = -\frac{2\pi}{|k|^3}.
    \end{align*}
    Utilizing the property $\mathcal{F}^{-1}[k_i k_j \hat{g}](x) = -\partial_{x_i}\partial_{x_j} g(x)$, we obtain
    \begin{align*}
        \mathcal{F}^{-1}\left[\frac{k_i k_j}{|k|^3}\right](x)
        &= \mathcal{F}^{-1}\left[ \frac{1}{-2\pi} k_i k_j \mathcal{F}[|x|](k) \right](x)
        = \frac{1}{2\pi} \partial_{x_i}\partial_{x_j}|x|.
    \end{align*}
    Direct calculation of the Hessian yields
    \begin{align*}
        \partial_{x_i}\partial_{x_j}|x| = \partial_{x_i}\left(\frac{x_j}{|x|}\right) = \frac{\delta_{ij}}{|x|} - \frac{x_i x_j}{|x|^3}.
    \end{align*}
    Thus, in tensor notation,
    \begin{align}
        \mathcal{F}^{-1}\left[\frac{k\otimes k}{|k|^3}\right](x) = \frac{1}{2\pi}\left(\frac{1}{|x|}I - \frac{x\otimes x}{|x|^3}\right).
    \label{eq:inv_Fourier_2}
    \end{align}

    Subtracting \eqref{eq:inv_Fourier_2} from \eqref{eq:inv_Fourier_1}, we arrive at the desired expression:
    \begin{align*}
        E(x) &= \frac{1}{2\pi|x|}I - \frac{1}{2\pi}\left(\frac{1}{|x|}I - \frac{x\otimes x}{|x|^3}\right)
        = \frac{x\otimes x}{2\pi|x|^3}.
    \end{align*}
\end{proof}

\subsection{Explicit representation of the normal velocity}

We now apply the general theory established in the previous sections to the specific setting of the ILLSS model.
Recall from Lemma~\ref{lem:distributional} that the momentum equation satisfies the following distributional relation on the entire boundary $\partial B$:
\begin{equation*}
    \nabla p = -\partial_{x_3}v_\parallel - g\nu\delta_\Gamma.
\end{equation*}
Substituting the definition of the DtN operator $\Lambda_{\mathrm{DN}}u = \partial_{x_3}v_\parallel$ and the jump condition $g=\kappa$, this equation becomes
\begin{equation*}
    \nabla p + \Lambda_{\mathrm{DN}} u = - \kappa \nu \delta_\Gamma.
\end{equation*}
This implies that the velocity field $u$ is given by the convolution of the fundamental solution tensor $E$ (derived in Lemma~\ref{lem:fund_sol}) and the singular force term $-\kappa \nu \delta_\Gamma$:
\begin{equation}
    u(x) = -\int_\Gamma E(x - y) \nu(y) \kappa(y) \,\mathrm{d}s(y), \quad x \in \partial B \setminus \Gamma.
    \label{eq:velocity_field_conv}
\end{equation}

We first examine the decay behavior of the velocity field $u$.  
Let $\Gamma$ be a $C^2$ Jordan curve.
Since the kernel $E(z)$ is homogeneous of degree $-1$, we have the expansion $E(x-y) = E(x) + \mathcal{O}(|x|^{-2})$ as $|x|\to\infty$. Substituting this into \eqref{eq:velocity_field_conv}, we obtain
\begin{align*}
    u(x)
    = -E(x) \int_{\Gamma} \kappa(y)\nu(y) \, \mathrm{d}s(y) + \mathcal{O}\left(|x|^{-2}\right) \qquad \text{as } |x| \to \infty.
\end{align*}
Since $\Gamma$ is a closed curve, the Frenet formulas imply the relation $\kappa\nu = -\partial_s \tau$, and thus
\begin{align*}
    \int_\Gamma \kappa(y)\nu(y) \, \mathrm{d}s(y)
    = -\int_\Gamma \partial_s \tau(y) \, \mathrm{d}s(y)
    = 0,
\end{align*}
where $\tau$ denotes the unit tangent vector to $\Gamma$.
This implies that the net force exerted by the surface tension on the fluid vanishes, consistent with the zero net force assumption in Definition~\ref{def:classical}.
Consequently, the monopole term of order $\mathcal{O}(|x|^{-1})$ vanishes, and we obtain the faster dipole decay:
\begin{align}
    |u(x)| = \mathcal{O}(|x|^{-2}) \qquad \text{as } |x| \to \infty.
    \label{eq:u_decay}
\end{align}
This decay rate verifies the condition $|u|=\mathcal{O}(|x|^{-2})$ required for the uniqueness theorem (Theorem~\ref{thm:uniqueness}).

To derive the boundary integral equation model, we analyze the behavior of the normal component of the velocity field $u$ defined by \eqref{eq:velocity_field_conv} as $x$ approaches the interface $\Gamma$.
Consistent with Definition~\ref{def:classical}, we evaluate the limit along the normal direction.

\begin{proposition}
\label{prop:velocity-limit}
Let $\Gamma$ be a $C^2$ Jordan curve. Then, for any $z \in \Gamma$, the limits of $u$ from the interior and exterior along the normal direction exist and coincide. Specifically,
\begin{equation}
    \lim_{\ell \to 0} u(z + \ell \nu(z)) = -\frac{1}{\pi} \kappa(z) \nu(z) - \int_\Gamma E(z - y) \nu(y) \kappa(y) \,\mathrm{d}s(y).
    \label{eq:velocity_limit}
\end{equation}
\end{proposition}

Note that the integral on the right-hand side of \eqref{eq:velocity_limit} is well-defined as a proper integral due to the geometric cancellation of the singularity.

The following geometric lemma plays a key role in the proof.

\begin{lemma}
\label{lem:ratio_bound}
Let $\delta > 0$ and let $g \in C^2([-\delta, \delta])$ satisfy $g(0) = g'(0) = 0$.
\begin{enumerate}
    \item[{\rm (i)}] For any $\xi \in [-\delta, \delta]$, the following estimates hold:
    \begin{equation*}
        |g(\xi)| \le \frac{1}{2} \| g'' \|_{L^\infty} \xi^2, \qquad
        |g'(\xi)| \le \| g'' \|_{L^\infty} |\xi|.
    \end{equation*}
    \item[{\rm (ii)}] Consider the graph $\Gamma_\delta = \{ G(\xi) \coloneqq (\xi, g(\xi))^\top \mid \xi \in [-\delta, \delta] \}$.
    Let $x(\ell) \coloneqq (0, -\ell)^\top$. Suppose $\ell$ is chosen such that the origin is the nearest point on $\Gamma_\delta$ to $x(\ell)$ and $|x(\ell)| \le 1$. Then, there exist positive constants $c_1, c_2$ depending only on $\delta$ and $\| g'' \|_{L^\infty}$ such that
    \begin{equation*}
        c_1 \le \frac{|x(\ell) - G(\xi)|}{|x(\ell) - \xi'|} \le c_2 \quad \text{for all } \xi \in [-\delta, \delta],
    \end{equation*}
    where $\xi' \coloneqq (\xi, 0)^\top$.
\end{enumerate}
\end{lemma}
The proof of the above lemma will be omitted as it is a standard geometric estimate; see e.g. \cite[p.~135]{friedman1964partial}.

\begin{proof}[Proof of Proposition~\ref{prop:velocity-limit}]
    By translation and rotation, we assume $z = 0$ and $\nu(z) = (0, -1)^\top$. Thus $x(\ell) = z + \ell \nu(z) = (0, -\ell)^\top$.

    \textit{Step 1: Well-definedness of the integral.}
    Since $\Gamma \in C^2$, there exists $c > 0$ such that $|(z- y) \cdot \nu(y)| \le c |z - y|^2$ for $y \in \Gamma$ \cite[Lemma~3.15, p.~124]{folland1995introduction}.
    Using the explicit form \eqref{eq:fundamental_explicit} of $E$, we estimate the kernel at $x=z$:
    \begin{equation*}
        |E(z - y) \nu(y)| = \frac{1}{2\pi} \frac{|(z - y) \left((z - y) \cdot \nu(y)\right)|}{|z - y|^3} 
        \le \frac{1}{2\pi} \frac{|z-y| \cdot c|z-y|^2}{|z-y|^3} = \frac{c}{2\pi}.
    \end{equation*}
    Thus, the integrand is bounded on $\Gamma$, and the integral in \eqref{eq:velocity_limit} is well-defined as a proper Lebesgue integral.

    \textit{Step 2: Local decomposition.}
    Locally near $z=0$, $\Gamma$ is represented as a graph $\Gamma_\delta$ of $g(\xi)$.
    We decompose the integral into a local part $I_\delta(x)$ on $\Gamma_\delta$ and a remote part $J_\delta(x)$ on $\Gamma \setminus \Gamma_\delta$.
    Since $x(\ell)$ stays away from $\Gamma \setminus \Gamma_\delta$ for small $\ell$, the limit $\lim_{\ell \to 0} J_\delta(x(\ell)) = J_\delta(z)$ is trivial.
    The main task is to analyze $I_\delta(x(\ell))$. We approximate $\Gamma_\delta$ by the tangent line (flat segment) and define the approximation term:
    \begin{equation*}
        I_\delta'(x) \coloneqq -\int_{-\delta}^\delta E(x - \xi') \nu(z) \kappa(z) \, \mathrm{d}\xi,
    \end{equation*}
    where $\xi'=(\xi, 0)^\top$ and we have replaced $\nu(\xi')$ and $\kappa(\xi')$ with their values at the origin, $\nu(z) = (0, -1)^\top$ and $\kappa(z)$.

    \textit{Step 3: Limit of the flat approximation $I_\delta'$.}
    Explicit calculation yields:
    \begin{equation*}
        I_\delta'(x(\ell))
        = -\frac{\kappa(z)}{2\pi} \int_{-\delta}^\delta \frac{(x(\ell) - \xi') (x(\ell) - \xi')^\top \nu(z)}{|x(\ell) - \xi'|^3} \, \mathrm{d}\xi.
    \end{equation*}
    Substituting $x(\ell)=(0,-\ell)^\top$ and $\nu(z)=(0,-1)^\top$, we have $(x(\ell) - \xi') \cdot \nu(z) = \ell$. Thus,
    \begin{equation*}
        I_\delta'(x(\ell)) = -\frac{\kappa(z)}{2\pi} \int_{-\delta}^\delta \frac{1}{(\xi^2 + \ell^2)^{3/2}} \begin{pmatrix} -\xi \ell \\ -\ell^2 \end{pmatrix} \, \mathrm{d}\xi.
    \end{equation*}
    The first component vanishes due to odd symmetry. The second component is:
    \begin{equation*}
        (I_\delta'(x(\ell)))_2 = \frac{\kappa(z)}{2\pi} \int_{-\delta}^\delta \frac{\ell^2}{(\xi^2 + \ell^2)^{3/2}} \, \mathrm{d}\xi.
    \end{equation*}
    By the substitution $\xi = |\ell| \zeta$, the integral becomes:
    \begin{equation*}
        \int_{-\delta/|\ell|}^{\delta/|\ell|} \frac{\ell^2}{|\ell|^3 (\zeta^2 + 1)^{3/2}} |\ell| \, \mathrm{d}\zeta = \int_{-\delta/|\ell|}^{\delta/|\ell|} \frac{1}{(\zeta^2 + 1)^{3/2}} \, \mathrm{d}\zeta.
    \end{equation*}
    As $\ell \to 0$, the integration limits tend to $\pm \infty$. Since $\int_{-\infty}^\infty (\zeta^2+1)^{-3/2} d\zeta = 2$, we have:
    \begin{equation*}
        \lim_{\ell \to 0} (I_\delta'(x(\ell)))_2 = \frac{\kappa(z)}{2\pi} \cdot 2 = \frac{\kappa(z)}{\pi}.
    \end{equation*}
    Recovering the vector form with $\nu(z)=(0,-1)^\top$ (note that the second component corresponds to the $-\nu(z)$ direction), we find:
    \begin{equation}
        \lim_{\ell \to 0} I_\delta'(x(\ell)) = \frac{\kappa(z)}{\pi} \begin{pmatrix} 0 \\ 1 \end{pmatrix} = -\frac{\kappa(z)}{\pi} \nu(z).
        \label{eq:limit_Idp}
    \end{equation}
    This limit is independent of the sign of $\ell$, which implies the continuity of the velocity field across the interface (for the flat approximation part).

    \textit{Step 4: Convergence of the remainder $I_\delta - I_\delta'$.}
    It remains to show that the error term $R(\ell) \coloneqq I_\delta(x(\ell)) - I_\delta'(x(\ell))$ converges to $I_\delta(z)$.
    First, observe that $I_\delta'(z) = 0$. Indeed, at $\ell=0$, we have $x(0)=z=0$ and $\xi'=(\xi, 0)^\top$, so $(z-\xi')\cdot\nu(z) = (-\xi, 0)^\top \cdot (0, -1)^\top = 0$. Thus, the integrand of $I_\delta'(z)$ vanishes identically.
    Consequently, proving $R(\ell) \to I_\delta(z)$ is equivalent to proving $I_\delta(x(\ell)) - I_\delta'(x(\ell)) \to I_\delta(z) - I_\delta'(z)$.
    The difference of the kernels possesses a weaker singularity (in fact, it is bounded) due to the $C^2$ regularity of $\Gamma$ and the continuity of $\kappa$.
    By applying the estimates in Lemma~\ref{lem:ratio_bound}, one can verify that the difference is continuous up to the boundary (details omitted; see e.g., \cite[Chapter~5]{friedman1964partial}).
    
    \textit{Conclusion.}
    Combining the results, we obtain
    \begin{align*}
        \lim_{\ell \to 0} u(x(\ell)) 
        &= \lim_{\ell \to 0} J_\delta(x(\ell)) + \lim_{\ell \to 0} I_\delta'(x(\ell)) + \lim_{\ell \to 0} (I_\delta(x(\ell)) - I_\delta'(x(\ell))) \\
        &= J_\delta(z) - \frac{\kappa(z)}{\pi} \nu(z) + (I_\delta(z) - 0) \\
        &= -\frac{\kappa(z)}{\pi} \nu(z) - \int_\Gamma E(z - y) \nu(y) \kappa(y) \, \mathrm{d}s(y).
    \end{align*}
\end{proof}

Using Proposition~\ref{prop:velocity-limit}, the normal velocity $U(z) = u(z) \cdot \nu(z)$ is given by:
\begin{align}
    U(z) &= -\frac{1}{\pi} \kappa(z) - \nu(z) \cdot \int_\Gamma E(z - y) \nu(y) \kappa(y) \, \mathrm{d}s(y) \notag \\
    &= -\frac{1}{\pi} \kappa(z) - \frac{1}{2\pi} \int_\Gamma \frac{(\nu(z) \cdot (z-y)) (\nu(y) \cdot (z-y))}{|z-y|^3} \kappa(y) \, \mathrm{d}s(y).
    \label{eq:normal_velocity_explicit}
\end{align}

\subsection{Boundary Integral Equation}

We now translate the explicit representation of the normal velocity $U$ derived in Proposition~\ref{prop:velocity-limit} into an evolution equation for the parameterization of the interface $\Gamma$.

Let $\gamma \in C^2(\mathbb{S}^1; \mathbb{R}^2)$ be a counter-clockwise parameterization of the closed curve $\Gamma$, where $\mathbb{S}^1 \coloneqq \mathbb{R}/\mathbb{Z}$.
We introduce the following geometric quantities associated with $\gamma$.
The unit tangent vector $\tau(\gamma)$ and the unit outward normal vector $\nu(\gamma)$ are defined by
\begin{equation}
    \tau(\gamma)(x) \coloneqq \frac{\partial_x \gamma(x)}{|\partial_x \gamma(x)|}, \quad
    \nu(\gamma)(x) \coloneqq \begin{pmatrix}
        0&1\\
        -1&0
    \end{pmatrix} \tau(\gamma)(x), \quad x \in \mathbb{S}^1.
\end{equation}
The \textit{curvature density} $\kappa(\gamma)$ is defined by
\begin{equation}
    \kappa(\gamma)(x) \coloneqq -\frac{\nu(\gamma)(x) \cdot \partial_x^2 \gamma(x)}{|\partial_x \gamma(x)|}.
\end{equation}
Note that this quantity relates to the standard geometric curvature $\kappa_{\text{geom}}$ by $\kappa(\gamma)(x) = \kappa_{\text{geom}}(\gamma(x)) |\partial_x \gamma(x)|$.

Next, we define the integral kernel $K(\gamma)$ on $\mathbb{S}^1 \times \mathbb{S}^1$ corresponding to the fundamental solution $E$. Using the notation $\Delta\gamma(x,y) \coloneqq \gamma(y) - \gamma(x)$, we set
\begin{equation}
    K(\gamma)(x, y) \coloneqq \frac{(\nu(\gamma)(x) \cdot \Delta\gamma(x,y)) \, (\nu(\gamma)(y) \cdot \Delta\gamma(x,y))}{|\Delta\gamma(x,y)|^3}, \qquad x \neq y.
\end{equation}

Using these notations, we rewrite the normal velocity $U$ at a point $z = \gamma(x)$ given by \eqref{eq:velocity_limit} in terms of the parameter $x$. 
Substituting the line element $\mathrm{d}s(y) = |\partial_y \gamma(y)|\,\mathrm{d}y$ and the relation $\kappa_{\text{geom}}(y)\,\mathrm{d}s(y) = \kappa(\gamma)(y)\,\mathrm{d}y$, the integral term transforms as follows:
\begin{align*}
    \int_{\Gamma} \frac{(\nu(z) \cdot (z-y)) (\nu(y) \cdot (z-y))}{|z-y|^3} \kappa_{\text{geom}}(y) \, \mathrm{d}s(y) 
    &= \int_{\mathbb{S}^1} \frac{(\nu(\gamma)(x) \cdot \Delta\gamma(x,y)) (\nu(\gamma)(y) \cdot \Delta\gamma(x,y))}{|\Delta\gamma(x,y)|^3} \kappa(\gamma)(y) \, \mathrm{d}y \\
    &= \int_{\mathbb{S}^1} K(\gamma)(x, y) \kappa(\gamma)(y) \, \mathrm{d}y.
\end{align*}
Note that the sign change in the difference vector is cancelled out by the quadratic form in the numerator.
Similarly, the local term $- \kappa_{\text{geom}}(z)/\pi$ transforms to $-\kappa(\gamma)(x)/(\pi |\partial_x \gamma(x)|)$.

Consequently, the normal velocity $U(\gamma(x))$ is represented as
\begin{equation}
    U(\gamma(x)) = -\frac{1}{\pi |\partial_x \gamma(x)|} \kappa(\gamma)(x) + F(\gamma)(x),
    \label{eq:normal_velocity_param}
\end{equation}
where the nonlocal forcing term $F(\gamma)$ is defined by
\begin{equation}
    F(\gamma)(x) \coloneqq -\frac{1}{2\pi} \int_{\mathbb{S}^1} K(\gamma)(x, y) \kappa(\gamma)(y) \, \mathrm{d}y.
\end{equation}

The motion of the interface is determined by the condition $\partial_t \gamma \cdot \nu = U$. Assuming that the tangential velocity is zero, we obtain the boundary integral equation:
\begin{equation}
    \begin{dcases*}
        \partial_t \gamma = \left( -\frac{1}{\pi |\partial_x \gamma|} \kappa(\gamma) + F(\gamma) \right) \nu(\gamma) & in $\mathbb{S}^1 \times (0, T]$, \\
        \gamma(\cdot, 0) = \gamma_0(\cdot) & in $\mathbb{S}^1$.
    \end{dcases*}
    \label{eq:BIE}
\end{equation}

Summarizing the above discussions, we obtain the following proposition, which derives the boundary integral equation \eqref{eq:BIE} from the original ILLSS model \eqref{eq:Stokes_subfluid}--\eqref{eq:jump_condition}.
Crucially, the uniqueness of the classical solution established in Theorem~\ref{thm:uniqueness} ensures that the interface velocity is uniquely determined by the interface configuration via the boundary integral representation \eqref{eq:normal_velocity_param}, thereby justifying the reduction of the bulk problem to the boundary integral equation.

\begin{proposition}
\label{prop:BIE_derivation}
Let $(\Gamma(t))_{t \in [0, T]} \subset \partial B$ be a family of $C^2$ Jordan curves evolving according to the ILLSS model.
Then, the normal velocity of $\Gamma(t)$ is uniquely determined by \eqref{eq:normal_velocity_param}.
Consequently, the evolution of the parameterization $\gamma$ satisfies the boundary integral equation~\eqref{eq:BIE} under the condition of vanishing tangential velocity.
\end{proposition}

\subsection{Curve-Shortening Property}
\label{subsubsec:CS}

The formulation based on the DtN operator allows us to establish the curve-shortening property in a transparent manner. 
We demonstrate that the length of the interface decreases monotonically, with the dissipation rate precisely controlled by the fractional Sobolev norm of the surface velocity.

First, we recall the definition of the homogeneous Sobolev space $\dot{H}^s(\mathbb{R}^2)$ for $s \in \mathbb{R}$.
Based on the Fourier transform definition given in \eqref{eq:Fourier}, the space is defined as the completion of smooth compactly supported functions:
\begin{equation}
    \dot{H}^s(\mathbb{R}^2) \coloneqq \overline{C_c^\infty(\mathbb{R}^2)}^{\|\cdot\|_{\dot{H}^s(\mathbb{R}^2)}}, \qquad
    \|f\|_{\dot{H}^s(\mathbb{R}^2)} \coloneqq \left\| |k|^s \hat{f}(k) \right\|_{L^2(\mathbb{R}^2)}.
\end{equation}

\begin{proposition}
\label{prop:CS}
Let $(\Gamma(t))_{t \in [0,T)} \subset \partial B$ be a family of $C^2$ Jordan curves evolving according to the ILLSS model~\eqref{eq:Langmuir-DtN-1}--\eqref{eq:Langmuir-DtN-2}.
Let $\mathcal{L}(t)$ denote the length of $\Gamma(t)$.
Then, the rate of change of the perimeter satisfies:
\begin{equation}
    \frac{\mathrm{d}}{\mathrm{d}t} \mathcal{L}(t) = -\frac{1}{(2\pi)^2} \|u\|_{\dot{H}^{1/2}(\mathbb{R}^2)}^2.
\end{equation}
In particular, the length $\mathcal{L}(t)$ is strictly decreasing unless the fluid is at rest ($u \equiv 0$).
\end{proposition}

\begin{proof}
    First, recall from \eqref{eq:u_decay} that the velocity field satisfies the decay estimate $|u(x)| = \mathcal{O}(|x|^{-2})$ as $|x| \to \infty$. 
    This decay rate is sufficient to justify the integration by parts on the unbounded domain $\mathbb{R}^2$.

    \textit{Step 1: Geometric evolution.}
    According to the Reynolds transport theorem, the time derivative of the length $\mathcal{L}(t)$ of a curve evolving with normal velocity $U = u \cdot \nu$ is given by
    \begin{equation}
        \frac{\mathrm{d}}{\mathrm{d}t} \mathcal{L}(t) = \int_{\Gamma} \kappa (u \cdot \nu) \, \mathrm{d}s.
        \label{eq:geo_rate}
    \end{equation}

    \textit{Step 2: Energy dissipation.}
    We compute the energy dissipation rate defined by the DtN operator.
    Using the bulk relation $-\Lambda_{\mathrm{DN}} u = \nabla p$ (valid in $\Omega_\pm$), we split the integral over $\mathbb{R}^2$ into the interior and exterior domains:
    \begin{align*}
        -\int_{\mathbb{R}^2} (\Lambda_{\mathrm{DN}} u) \cdot u \, \mathrm{d}x
        &= \int_{\Omega_-} \nabla p \cdot u \, \mathrm{d}x + \int_{\Omega_+} \nabla p \cdot u \, \mathrm{d}x.
    \end{align*}
    Applying the divergence theorem to each domain and using the incompressibility $\nabla \cdot u = 0$, we have
    \begin{align*}
        \int_{\Omega_-} \nabla p \cdot u \, \mathrm{d}x 
        &= \int_{\partial \Omega_-} p (u \cdot \nu_-) \, \mathrm{d}s = \int_{\Gamma} p_- (u \cdot \nu) \, \mathrm{d}s, \\
        \int_{\Omega_+} \nabla p \cdot u \, \mathrm{d}x 
        &= \int_{\partial \Omega_+} p (u \cdot \nu_+) \, \mathrm{d}s = -\int_{\Gamma} p_+ (u \cdot \nu) \, \mathrm{d}s,
    \end{align*}
    where $\nu$ is the unit normal vector pointing from $\Omega_-$ to $\Omega_+$.
    Summing these contributions and using the pressure jump condition $\llbracket p \rrbracket = \kappa$, we obtain
    \begin{equation}
        -\int_{\mathbb{R}^2} (\Lambda_{\mathrm{DN}} u) \cdot u \, \mathrm{d}x
        = \int_{\Gamma} (p_- - p_+) (u \cdot \nu) \, \mathrm{d}s
        = \int_{\Gamma} \kappa (u \cdot \nu) \, \mathrm{d}s.
        \label{eq:dissipation_rate}
    \end{equation}

    \textit{Conclusion.}
    Comparing \eqref{eq:geo_rate} and \eqref{eq:dissipation_rate}, we obtain the identity:
    \begin{equation*}
        \frac{\mathrm{d}}{\mathrm{d}t} \mathcal{L}(t) = -\int_{\mathbb{R}^2} (\Lambda_{\mathrm{DN}} u) \cdot u \, \mathrm{d}x.
    \end{equation*}
    Finally, using Plancherel's theorem and the symbol of the DtN operator ($|k|$), we arrive at the desired result:
    \begin{equation*}
        \frac{\mathrm{d}}{\mathrm{d}t} \mathcal{L}(t) 
        = -\frac{1}{(2\pi)^2} \int_{\mathbb{R}^2} \overline{\hat{u}(k)} \cdot (|k| \hat{u}(k)) \, \mathrm{d}k
        = -\frac{1}{(2\pi)^2} \|u\|_{\dot{H}^{1/2}(\mathbb{R}^2)}^2.
    \end{equation*}
\end{proof}

%%%%%%%%%%%%%%%%%%%%%%%%%%%%%%%%%%%%%%%%%%%%%%%%%%%
%%%%%%%%%%%%%%%%%%%%%%%%%%%%%%%%%%%%%%%%%%%%%%%%%%%
%%%%%%%%%%%%%%%%%%%%%%%%%%%%%%%%%%%%%%%%%%%%%%%%%%%
\section{Preliminaries}
\label{sec:preliminaries}

Hereafter, we suppress the target space $\mathbb{R}^2$ in the notation of function spaces and norms.
Specifically, we write $C^k$, $L^2$, and $H^k$ to denote $C^k(\mathbb{S}^1;\mathbb{R}^2)$, $L^2(\mathbb{S}^1;\mathbb{R}^2)$, and $H^k(\mathbb{S}^1;\mathbb{R}^2)$, respectively.
Accordingly, we abbreviate their norms to $\|\cdot\|_{C^k}$, $\|\cdot\|_{L^2}$, and $\|\cdot\|_{H^k}$.
Furthermore, for a reference curve $\bar{\gamma}\in H^2$ and $\varepsilon>0$, we define the open ball centered at $\bar{\gamma}$ by
\begin{align*}
    B_\varepsilon(\bar{\gamma})\coloneqq\{\eta\in H^2\mid \|\eta-\bar{\gamma}\|_{H^2}<\varepsilon\}.
\end{align*}

\subsection{DeTurck trick}

Since the boundary integral equation \eqref{eq:BIE} describes a geometric evolution, it is invariant under reparametrization of the curve. This geometric invariance leads to the degeneracy of the diffusion operator in the tangential direction, precluding the direct application of standard parabolic theory. To establish local well-posedness, we employ the so-called DeTurck trick to fix the gauge and transform the system into a strictly parabolic quasilinear equation.

Recall the Frenet--Serret formula for a curve $\gamma$:
\begin{align*}
    -\frac{1}{|\partial_x\gamma|}\kappa(\gamma)\nu(\gamma)
    = \partial_s\tau(\gamma)
    = \frac{1}{|\partial_x\gamma|}\partial_x\left(\frac{\partial_x\gamma}{|\partial_x\gamma|}\right)
    = \frac{\partial_x^2\gamma}{|\partial_x\gamma|^2} - \left(\frac{\partial_x^2\gamma}{|\partial_x\gamma|^2}\cdot\tau(\gamma)\right)\tau(\gamma).
\end{align*}
The first term on the right-hand side corresponds to the principal elliptic part (essentially the Laplacian), whereas the second term represents the tangential component responsible for the degeneracy.
Following the idea of DeTurck, we consider the following \textit{modified problem} where we artificially add a tangential velocity term to cancel the degeneracy, effectively replacing the curvature term with the Laplacian-like term:
\begin{equation}
    \begin{dcases*}
        \partial_t\gamma = \frac{1}{\pi}\frac{\partial_x^2\gamma}{|\partial_x\gamma|^2} + F(\gamma)\nu(\gamma) & in $\mathbb{S}^1\times(0,T]$,\\
        \gamma(\cdot,0) = \gamma_0(\cdot) & in $\mathbb{S}^1$.
    \end{dcases*}
    \label{eq:Langmuir_rewrite}
\end{equation}
This system is strictly parabolic. We now verify that solving this modified problem is equivalent to solving the original geometric problem, modulo a time-dependent reparametrization.

Suppose we have a sufficiently smooth solution $\tilde{\gamma}$ to the modified problem \eqref{eq:Langmuir_rewrite} that satisfies the non-degeneracy condition $|\partial_x \tilde{\gamma}| > 0$. We construct a solution to the original problem \eqref{eq:BIE} by reparametrizing $\tilde{\gamma}$.
Let $G(\tilde{\gamma})$ be the scalar tangential velocity component defined by:
\[
    G(\tilde{\gamma}) \coloneqq \frac{1}{\pi |\partial_x \tilde{\gamma}|} \left( \frac{\partial_x^2 \tilde{\gamma}}{|\partial_x \tilde{\gamma}|^2} \cdot \tau(\tilde{\gamma}) \right).
\]
Consider the following ODE for the reparametrization map $\varphi(x,t)$:
\begin{equation} \label{eq:change}
    \begin{dcases*}
    \frac{\partial\varphi}{\partial t}(x,t) = -G(\tilde{\gamma})(\varphi(x,t),t) & for $t\in(0,T]$,\\
    \varphi(x,0) = x.
    \end{dcases*}
\end{equation}
Since $\mathbb{S}^1$ is compact, for a smooth $\tilde{\gamma}$, this ODE admits a unique solution $\varphi(\cdot,t)$ which is a diffeomorphism of $\mathbb{S}^1$ for each $t$.
Define $\gamma(x,t) \coloneqq \tilde{\gamma}(\varphi(x,t), t)$.
By the chain rule, the time derivative of $\gamma$ is given by
\begin{align*}
    \partial_t \gamma(x,t) 
    &= (\partial_t \tilde{\gamma})(\varphi, t) + (\partial_x \tilde{\gamma})(\varphi, t) \frac{\partial\varphi}{\partial t} \\
    &= \left( \frac{1}{\pi} \frac{\partial_x^2 \tilde{\gamma}}{|\partial_x \tilde{\gamma}|^2} + F(\tilde{\gamma})\nu(\tilde{\gamma}) \right)(\varphi, t) - \left( G(\tilde{\gamma}) \partial_x \tilde{\gamma} \right)(\varphi, t).
\end{align*}
Using the definition of $G(\tilde{\gamma})$ and the relation $\partial_x \tilde{\gamma} = |\partial_x \tilde{\gamma}| \tau(\tilde{\gamma})$, the second term becomes:
\[
    G(\tilde{\gamma}) \partial_x \tilde{\gamma} 
    = \left[ \frac{1}{\pi |\partial_x \tilde{\gamma}|} \left( \frac{\partial_x^2 \tilde{\gamma}}{|\partial_x \tilde{\gamma}|^2} \cdot \tau(\tilde{\gamma}) \right) \right] |\partial_x \tilde{\gamma}| \tau(\tilde{\gamma})
    = \frac{1}{\pi} \left( \frac{\partial_x^2 \tilde{\gamma}}{|\partial_x \tilde{\gamma}|^2} \cdot \tau(\tilde{\gamma}) \right) \tau(\tilde{\gamma}).
\]
Subtracting this tangential component recovers the curvature vector via the Frenet--Serret formula:
\begin{align*}
    \partial_t \gamma
    &= \frac{1}{\pi} \left[ \frac{\partial_x^2 \tilde{\gamma}}{|\partial_x \tilde{\gamma}|^2} - \left( \frac{\partial_x^2 \tilde{\gamma}}{|\partial_x \tilde{\gamma}|^2} \cdot \tau(\tilde{\gamma}) \right) \tau(\tilde{\gamma}) \right] + F(\tilde{\gamma})\nu(\tilde{\gamma}) \\
    &= \frac{1}{\pi} \left( -\frac{1}{|\partial_x \tilde{\gamma}|} \kappa(\tilde{\gamma})\nu(\tilde{\gamma}) \right) + F(\tilde{\gamma})\nu(\tilde{\gamma}).
\end{align*}
Thus, $\gamma$ satisfies the original motion law \eqref{eq:BIE} concerning the normal velocity.

For the subsequent analysis, we define the differential operator $A(\gamma)$ and the lower-order term $f(\gamma)$ for the modified problem \eqref{eq:Langmuir_rewrite} as
\begin{align*}
    A(\gamma) \coloneqq \frac{1}{\pi |\partial_x\gamma|^2} \partial_x^2, \qquad
    f(\gamma) \coloneqq F(\gamma)\nu(\gamma).
\end{align*}
Then, \eqref{eq:Langmuir_rewrite} is written as an abstract quasilinear evolution equation:
\begin{equation}
    \begin{dcases*}
        \partial_t \gamma = A(\gamma)\gamma + f(\gamma) & for $t\in(0,T]$,\\
        \gamma(0) = \gamma_0.
    \end{dcases*}
    \label{eq:abstract_evolution}
\end{equation}
In view of the equivalence established above, the well-posedness of the original geometric evolution is reduced to that of this strictly parabolic system. Therefore, in the remainder of this paper, we shall focus our analysis on the modified problem~\eqref{eq:abstract_evolution}.

\subsection{Admissible set and function spaces}

In this subsection, we define the admissible set of curves and introduce the necessary function spaces and analytical tools.

\subsubsection{Sobolev embeddings}

We recall the standard Sobolev embeddings on the circle $\mathbb{S}^1$. For each $k \in \mathbb{Z}_{\ge 0}$ and $\alpha \in [0, 1/2)$, we have the continuous embedding $H^{k+1} \hookrightarrow C^{k,\alpha}$.
We denote the embedding constant by $S_{k,\alpha}$:
\begin{align}
    \|\gamma\|_{C^{k,\alpha}} \le S_{k,\alpha} \|\gamma\|_{H^{k+1}}.
    \label{eq:Sobolev}
\end{align}
We abbreviate $S_{k,0}$ as $S_k$.

\subsubsection{Admissible set}

For $\gamma \in C^1(\mathbb{S}^1; \mathbb{R}^2)$, we define the \textit{chord-arc constant} $|\gamma|_\ast$ by
\begin{align*}
    |\gamma|_\ast \coloneqq \inf_{\substack{x,y\in \mathbb{S}^1\\x\neq y}} \frac{|\gamma(x)-\gamma(y)|}{d_{\mathbb{S}^1}(x,y)},
\end{align*}
where $d_{\mathbb{S}^1}(x,y) \coloneqq \min_{k\in\mathbb{Z}}|x-y-k|$ denotes the geodesic distance on $\mathbb{S}^1$.
The condition $|\gamma|_\ast > 0$ ensures that the curve has no self-intersections and that the parameterization is non-degenerate ($|\partial_x \gamma| > 0$).
Specifically, this definition yields the following lower bounds:
\begin{alignat}{2}
    |\Delta\gamma(y,x)| &\ge |\gamma|_\ast d_{\mathbb{S}^1}(x,y), &\qquad& x,y\in\mathbb{S}^1, \label{eq:diff_LB}\\
    |\partial_x \gamma(x)| &\ge |\gamma|_\ast, &\qquad& x\in\mathbb{S}^1. \label{eq:gamma'_LB}
\end{alignat}
We define the open set of admissible parameterizations $V$ by
\begin{align*}
    V \coloneqq \left\{ \gamma\in H^2(\mathbb{S}^1; \mathbb{R}^2) \relmiddle{|} |\gamma|_\ast > 0 \right\}.
\end{align*}
Note that the continuous embedding $H^2 \hookrightarrow C^1$ ensures that $|\gamma|_\ast$ is well-defined for $\gamma \in V$.
Since $H^1(\mathbb{S}^1)$ is a Banach algebra, the differential operator $A$ and the nonlinear term $f$ act as well-defined mappings on this domain:
\[
    A \colon V \to \mathcal{B}(H^3, H^1), \qquad f \colon V \to H^1.
\]

\subsubsection{Hardy--Littlewood maximal function}

To handle the singular integral kernel appearing in the nonlocal term $F(\gamma)$, we utilize the Hardy--Littlewood maximal function.

\begin{definition}
    \label{def:HL}
    Let $h$ be a locally integrable function on $\mathbb{S}^1$. The \textit{Hardy--Littlewood maximal function} $\mathcal{M}(h)$ is defined by
    \begin{align*}
        \mathcal{M}(h)(x) \coloneqq \sup_{\substack{I \subset \mathbb{S}^1 \text{: open arc},\\ I \ni x}} \frac{1}{|I|} \int_I |h(y)| \, \mathrm{d}y,
    \end{align*}
    where $|I|$ denotes the arc length of $I$.
\end{definition}

\begin{proposition}[{See e.g., \cite[Theorem 1 in Chapter~I, Section~1]{stein1970singular}}]
    \label{prop:HL}
    The maximal operator is bounded on $L^2$. That is, there exists a constant $C_{\mathrm{HL}}>0$ such that
    \begin{align*}
        \|\mathcal{M}(h)\|_{L^2} \le C_{\mathrm{HL}}\|h\|_{L^2}
    \end{align*}
    holds for all $h\in L^2(\mathbb{S}^1)$.
\end{proposition}

\subsection{Variations}
\label{subsubsec:Variations}

In this subsection, we study the variations of the operators $A$ and $f$ defined in \eqref{eq:abstract_evolution} and establish that both are of class $C^\infty$ on the admissible set $V$.

We begin with the following elementary estimate for the chord length.

\begin{lemma}
    \label{lem:delta}
    Let $\gamma\in H^2$. Then we have
    \begin{align*}
        |\Delta\gamma(x,y)|\le S_1\|\gamma\|_{H^2}d_{\mathbb{S}^1}(x,y)
    \end{align*}
    for any $x,y\in\mathbb{S}^1$, where $\Delta\gamma(x,y) \coloneqq \gamma(y) - \gamma(x)$ and $S_1$ is the embedding constant in \eqref{eq:Sobolev}.
\end{lemma}

\begin{proof}
    Let $\Gamma_{x,y} \subset \mathbb{S}^1$ be the shorter arc connecting $x$ and $y$ (i.e., the geodesic segment), whose length corresponds to $d_{\mathbb{S}^1}(x,y)$.
    By the fundamental theorem of calculus along the curve, we have
    \begin{align*}
        |\gamma(x) - \gamma(y)|
        &= \left| \int_{\Gamma_{x,y}} \partial_z \gamma(z) \, \mathrm{d}z \right|
        \le \sup_{z \in \mathbb{S}^1} |\partial_z \gamma(z)| \int_{\Gamma_{x,y}} \mathrm{d}z
        = \|\gamma'\|_{C^0} d_{\mathbb{S}^1}(x,y).
    \end{align*}
    Using the Sobolev embedding estimate $\|\gamma'\|_{C^0} \le \|\gamma\|_{C^1} \le S_1 \|\gamma\|_{H^2}$, we obtain the desired inequality.
\end{proof}

The following lemma ensures that the admissible set $V$ is open.

\begin{lemma}
    \label{lem:V_open}
    The set $V$ is open in $H^2$.
\end{lemma}

\begin{proof}
    Let $\gamma \in V$. Then, there exists a constant $\delta > 0$ such that $|\gamma|_\ast \ge \delta$.
    Set $\varepsilon \coloneqq \delta/(2S_1)$.
    Choose any $\eta \in B_\varepsilon(\gamma)$, and let $x, y \in \mathbb{S}^1$ with $x \neq y$.
    Then, the lower bound \eqref{eq:diff_LB} and Lemma~\ref{lem:delta} imply that
    \begin{align*}
        |\Delta\eta(x,y)|
        &\ge |\Delta\gamma(x,y)| - |\Delta(\eta - \gamma)(x,y)| \\
        &\ge \left(|\gamma|_\ast - S_1 \|\eta - \gamma\|_{H^2} \right) d_{\mathbb{S}^1}(x,y)
        > \frac{\delta}{2} d_{\mathbb{S}^1}(x,y),
    \end{align*}
    which implies that $|\eta|_\ast \ge \delta / 2$, i.e., $\eta \in V$.
    Thus, $V$ is open in $H^2$.
\end{proof}

To handle the variations, we rely on the following algebraic inequalities.

\begin{lemma}
    \label{lem:fundamental_estimate}
    Let $n, m \in \mathbb{N}$, $a, b \in \mathbb{R}$, and $z, w, r_i, s_i \in \mathbb{R}^n \setminus \{0\}$ ($i = 1, \ldots, 2m$).
    Then, the following inequalities hold:
    \begin{align*}
        \left|\frac{1}{a^m}-\frac{1}{b^m}\right|&\le\left(\sum_{i=1}^m|b|^{-i}|a|^{-m+i-1}\right)|a-b|,\\
        \left|\prod_{i=1}^m(z\cdot r_i)-\prod_{i=1}^m(w\cdot r_i)\right|&\le\left(\sum_{i=1}^m|w|^{i-1}|z|^{m-i}\right)|z-w|\prod_{i=1}^m|r_i|,\\
        \left|\prod_{i=1}^m(z\cdot r_i)-\prod_{i=1}^m(z\cdot s_i)\right|&\le|z|^m\sum_{i=1}^m\left(\prod_{j=1}^{i-1}|s_j|\prod_{j=i+1}^m|r_j|\right)|r_i-s_i|,\\
        \left|\prod_{i=1}^m(r_{2i-1}\cdot r_{2i})-\prod_{i=1}^m(s_{2i-1}\cdot s_{2i})\right|&\le\sum_{i=1}^m\left(\prod_{j=1}^{2(i-1)}|s_j|\prod_{j=2i+1}^{2m}|r_j|\right)\left(|r_{2i}||r_{2i-1}-s_{2i-1}|+|s_{2i-1}||r_{2i}-s_{2i}|\right).
    \end{align*}
\end{lemma}

\begin{proof}
    The first inequality follows from the factorization $a^{-m} - b^{-m} = (b-a) \sum_{k=0}^{m-1} a^{-(m-k)} b^{-(k+1)}$.
    The other inequalities follow from the telescoping sum argument and the triangle inequality.
    For instance, the second inequality is derived as follows:
    \begin{align*}
        \left|\prod_{i=1}^m(z\cdot r_i)-\prod_{i=1}^m(w\cdot r_i)\right|
        &= \left|\sum_{k=1}^m \left( \prod_{i=1}^{k-1} (w \cdot r_i) \right) ((z-w) \cdot r_k) \left( \prod_{i=k+1}^m (z \cdot r_i) \right)\right| \\
        &\le \sum_{k=1}^m |w|^{k-1} |z|^{m-k} |z-w| \prod_{i=1}^m |r_i|.
    \end{align*}
    The remaining inequalities are proven similarly.
\end{proof}

Here and throughout this paper, the product $\prod_{i=n}^m$ is understood to be $1$ whenever $m < n$.
Combining these algebraic estimates with Sobolev embeddings, we obtain the following corollary.

\begin{corollary}
    \label{cor:fundamental_inequality}
    Suppose that $\gamma \in H^2$ satisfies $|\gamma|_\ast \ge \delta$ for some positive constant $\delta$.
    Then, there exist positive constants $C_1 = C_1(m, S_1, \delta)$, $C_2 = C_2(m, S_1, \delta, \|\gamma\|_{H^2})$, $C_3 = C_3(m, S_1, \|\gamma\|_{H^2})$, and $C_4 = C_4(m, S_1)$ such that, for any $\eta \in B_{\delta/(2S_1)}(\gamma)$, $\phi_1, \ldots, \phi_m, \varphi_1, \ldots, \varphi_m \in H^2$, and $x, y \in \mathbb{S}^1$, the following inequalities hold:
    \begin{align*}
        &\left| \frac{1}{|\gamma'|^m} - \frac{1}{|\eta'|^m} \right| \le C_1 \|\gamma - \eta\|_{H^2}, \\
        &\left| \prod_{i=1}^m (\gamma' \cdot \phi_i') - \prod_{i=1}^m (\eta' \cdot \phi_i') \right|
        \le C_2 \|\gamma - \eta\|_{H^2} \prod_{i=1}^m \|\phi_i\|_{H^2}, \\
        &\left| \prod_{i=1}^m (\gamma' \cdot \phi_i') - \prod_{i=1}^m (\gamma' \cdot \varphi_i') \right|
        \le C_3 \sum_{i=1}^m \left( \prod_{j=1}^{i-1} \|\varphi_j\|_{H^2} \prod_{j=i+1}^m \|\phi_j\|_{H^2} \right) \|\phi_i - \varphi_i\|_{H^2}, \\
        &\left| \prod_{i=1}^m (\phi_{2i-1}' \cdot \phi_{2i}') - \prod_{i=1}^m (\varphi_{2i-1}' \cdot \varphi_{2i}') \right| \\
        &\le C_4 \sum_{i=1}^m \left( \prod_{j=1}^{2(i-1)} \|\varphi_j\|_{H^2} \prod_{j=2i+1}^{2m} \|\phi_j\|_{H^2} \right) 
        \left( \|\phi_{2i}\|_{H^2} \|\phi_{2i-1} - \varphi_{2i-1}\|_{H^2} + \text{sym.} \right), \\
        &\left| \frac{1}{|\Delta\gamma|^m} - \frac{1}{|\Delta\eta|^m} \right| \le C_1 \|\gamma - \eta\|_{H^2} d_{\mathbb{S}^1}(x,y)^{-m}, \\
        &\left| \prod_{i=1}^m (\Delta\gamma \cdot \Delta\phi_i) - \prod_{i=1}^m (\Delta\eta \cdot \Delta\phi_i) \right|
        \le C_2 \|\gamma - \eta\|_{H^2} \left( \prod_{i=1}^m \|\phi_i\|_{H^2} \right) d_{\mathbb{S}^1}(x,y)^{2m}, \\
        &\left| \prod_{i=1}^m (\Delta\gamma \cdot \Delta\phi_i) - \prod_{i=1}^m (\Delta\gamma \cdot \Delta\varphi_i) \right| \le C_3 \sum_{i=1}^m \left( \prod_{j=1}^{i-1} \|\varphi_j\|_{H^2} \prod_{j=i+1}^m \|\phi_j\|_{H^2} \right) \|\phi_i - \varphi_i\|_{H^2} d_{\mathbb{S}^1}(x,y)^{2m}, \\
        &\left| \prod_{i=1}^m (\Delta\phi_{2i-1} \cdot \Delta\phi_{2i}) - \prod_{i=1}^m (\Delta\varphi_{2i-1} \cdot \Delta\varphi_{2i}) \right| \\
        &\le C_4 \sum_{i=1}^m \left( \prod_{j=1}^{2(i-1)} \|\varphi_j\|_{H^2} \prod_{j=2i+1}^{2m} \|\phi_j\|_{H^2} \right) \left( \|\phi_{2i}\|_{H^2} \|\phi_{2i-1} - \varphi_{2i-1}\|_{H^2} + \text{sym.} \right) d_{\mathbb{S}^1}(x,y)^{2m}.
    \end{align*}
\end{corollary}

For a general map $\mathcal{H}\colon V\to W$, its first variation $\mathcal{D}\mathcal{H}(\gamma)[\phi]$ at $\gamma\in V$ in the direction $\phi\in V$ is defined by
\begin{align*}
    \mathcal{D}\mathcal{H}(\gamma)[\phi] \coloneqq \left.\frac{\mathrm{d}}{\mathrm{d}\varepsilon}\mathcal{H}(\gamma+\varepsilon\phi)\right|_{\varepsilon=0}.
\end{align*}
Furthermore, for $k\in\mathbb{N}$, the $k$-th variation $\mathcal{D}^k\mathcal{H}(\gamma)[\phi_1,\ldots,\phi_k]$ of $\mathcal{H}$ with $\phi_1,\ldots,\phi_k\in V$ is defined recursively by
\begin{align*}
    \mathcal{D}^k\mathcal{H}(\gamma)[\phi_1,\ldots,\phi_k] \coloneqq \left.\frac{\mathrm{d}}{\mathrm{d}\varepsilon}\mathcal{D}^{k-1}\mathcal{H}(\gamma+\varepsilon\phi_k)[\phi_1,\ldots,\phi_{k-1}]\right|_{\varepsilon=0}.
\end{align*}
Note that $\mathcal{D}^k\mathcal{H}(\gamma)[\phi_1,\ldots,\phi_k]$ is linear in each variable $\phi_i$ ($i=1,\ldots,k$) and symmetric with respect to $\phi_1,\ldots,\phi_k$.

\begin{proposition}
    \label{prop:A_C-inf}
    $A\in C^\infty(V;\mathcal{B}(H^3;H^1))$.
\end{proposition}

\begin{proof}
    Let $\gamma\in V$ and set $B(\gamma)\coloneqq|\gamma'|^{-2}$.
    Since $\gamma\in V$, there exists a constant $\delta>0$ such that $|\gamma|_\ast\ge\delta$, which implies $|\gamma'|\ge\delta$ by \eqref{eq:gamma'_LB}.
    By induction on $k\in\mathbb{Z}_{\ge0}$, we see that the $k$-th variation $\mathcal{D}^kB(\gamma)[\phi_1,\ldots,\phi_k]$ with $\phi_1,\ldots,\phi_k\in H^2$ exists and is explicitly given by
    \begin{align*}
        \mathcal{D}^kB(\gamma)[\phi_1,\ldots,\phi_k]=\frac{1}{|\gamma'|^{2(k+1)}}\sum_{\ell=0}^{\lfloor\frac{k}{2}\rfloor}a_{k,\ell}|\gamma'|^{2\ell}P_\ell(\gamma',\phi_1',\ldots,\phi_k')
    \end{align*}
    with some constants $a_{k,\ell}$.
    Here, $P_\ell(z,r_1,\dots,r_k)$ is a homogeneous polynomial of degree $2(k-\ell)$ determined by the monomial $\prod_{i=1}^{k-2\ell}(z\cdot r_i)\,\prod_{j=1}^{\ell}(r_{k-2\ell+2j-1}\cdot r_{k-2\ell+2j})$.
    
    Fix $\eta\in B_{\delta/(2S_1)}(\gamma)$ and let $\varphi_1,\ldots,\varphi_k\in H^2$.
    Applying Corollary~\ref{cor:fundamental_inequality} to the explicit form of $\mathcal{D}^k B$, we obtain polynomials $D_{j,i}$ such that:
    \begin{align*}
        &\|\mathcal{D}^kB(\gamma)[\phi_1,\ldots,\phi_k]-\mathcal{D}^kB(\eta)[\varphi_1,\ldots,\varphi_k]\|_{C^0}
        \le D_{1,0}\|\gamma-\eta\|_{H^2}+\sum_{i=1}^kD_{1,i}\|\phi_i-\varphi_i\|_{H^2},\\
        &\|\partial_x\left(\mathcal{D}^kB(\gamma)[\phi_1,\ldots,\phi_k]-\mathcal{D}^kB(\eta)[\varphi_1,\ldots,\varphi_k]\right)\|_{L^2}
        \le D_{2,0}\|\gamma-\eta\|_{H^2}+\sum_{i=1}^kD_{2,i}\|\phi_i-\varphi_i\|_{H^2}.
    \end{align*}
    Since $A(\gamma)=B(\gamma)\partial_x^2$, the operator norm of the difference is bounded by
    \begin{align*}
        &\|\mathcal{D}^kA(\gamma)[\phi_1,\ldots,\phi_k]-\mathcal{D}^kA(\eta)[\varphi_1,\ldots,\varphi_k]\|_{\mathcal{B}(H^3;H^1)}\\
        &\le\sqrt{2}\left(\left(D_{1,0}+D_{2,0}\right)\|\gamma-\eta\|_{H^2}+\sum_{i=1}^k\left(D_{1,i}+D_{2,i}\right)\|\phi_i-\varphi_i\|_{H^2}\right).
    \end{align*}
    This implies the continuity of the $k$-th variation. Thus $A \in C^\infty$.
\end{proof}

Next, we investigate the regularity of the nonlinear term $f$.

\begin{lemma}
    \label{lem:tau-nu_C-inf}
    $\tau,\nu\in C^\infty(V; H^1)$.
\end{lemma}

\begin{proof}
    Let $\tilde{B}(\gamma)\coloneqq|\gamma'|^{-1}$.
    Similarly to the proof of Proposition~\ref{prop:A_C-inf}, the $k$-th variation of $\tilde{B}$ is given by a linear combination of terms involving $|\gamma'|^{-(2k+1)}$ and polynomials $P_\ell$.
    Using Corollary~\ref{cor:fundamental_inequality}, we verify the continuity of $\mathcal{D}^k \tilde{B}$.
    Since $\tau(\gamma)=\tilde{B}(\gamma)\gamma'$, the product rule for variations implies that $\tau \in C^\infty(V; H^1)$.
    The smoothness of $\nu = \mathcal{R}\tau$ follows immediately.
\end{proof}

\begin{corollary}
    \label{cor:kappa}
    Suppose that $\gamma\in H^2$ satisfies $|\gamma|_\ast\ge\delta$.
    Then, for any $k\in\mathbb{Z}_{\ge0}$, the variation of the curvature satisfies:
    \begin{align*}
        \left|\mathcal{D}^k\kappa(\gamma)[\phi_1,\ldots,\phi_k]-\mathcal{D}^k\kappa(\eta)[\varphi_1,\ldots,\varphi_k]\right|
        \le D_{1,0}\|\gamma-\eta\|_{H^2}+\sum_{i=1}^kD_{1,i}|(\phi_i-\varphi_i)''(y)|,
    \end{align*}
    where $D_{1,i}$ are polynomials in norms and pointwise second derivatives.
\end{corollary}

To handle the singular kernel $K(\gamma)$, we introduce auxiliary functions.
Let $h(\gamma)(x,y) \coloneqq |\Delta\gamma(x,y)|^{-3}$.

\begin{lemma}
    \label{lem:h}
    Suppose that $\gamma\in H^2$ satisfies $|\gamma|_\ast\ge\delta$ with a positive constant $\delta$.
    Let $\phi_1,\ldots,\phi_k,\varphi_1,\ldots,\varphi_k\in H^2$ and $\eta\in B_{\delta/(2S_1)}(\gamma)$.
    Then, for any $k\in\mathbb{Z}_{\ge0}$, there exist polynomials $D_{j,i}$ $(i=0,\ldots, k, \, j=1, 2)$ in $\|\gamma\|_{H^2}$, $\|\phi_\ell\|_{H^2}$, and $\|\varphi_\ell\|_{H^2}$ $(\ell=1, \ldots, k)$ such that the following estimates hold for all $x,y\in\mathbb{S}^1$ with $x\neq y${\rm :}
    \begin{align*}
        \left|\mathcal{D}^kh(\gamma)[\phi_1,\ldots,\phi_k]-\mathcal{D}^kh(\eta)[\varphi_1,\ldots,\varphi_k]\right|
        &\le\left(D_{1,0}\|\gamma-\eta\|_{H^2}+\sum_{i=1}^kD_{1,i}\|\phi_i-\varphi_i\|_{H^2}\right)d_{\mathbb{S}^1}(x,y)^{-3},\\
        \left|\partial_x\left(\mathcal{D}^kh(\gamma)[\phi_1,\ldots,\phi_k]-\mathcal{D}^kh(\eta)[\varphi_1,\ldots,\varphi_k]\right)\right|
        &\le\left(D_{2,0}\|\gamma-\eta\|_{H^2}+\sum_{i=1}^kD_{2,i}\|\phi_i-\varphi_i\|_{H^2}\right)d_{\mathbb{S}^1}(x,y)^{-4}.
    \end{align*}
\end{lemma}

\begin{proof}
    The $k$-th variation of $h$ is given by
    \begin{align*}
        \mathcal{D}^kh(\gamma)[\phi_1,\ldots,\phi_k]=\frac{1}{|\Delta\gamma|^{2k+3}}\sum_{\ell=0}^{\lfloor\frac{k}{2}\rfloor}c_{k,\ell}|\Delta\gamma|^{2\ell}P_\ell(\Delta\gamma,\Delta\phi_1,\ldots,\Delta\phi_k).
    \end{align*}
    Then, Corollary~\ref{cor:fundamental_inequality} implies the desired estimates.
\end{proof}

We also define $\mathcal{N}_1, \mathcal{N}_2 \colon C^1 \to C^0(\mathbb{S}^1\times\mathbb{S}^1)$ by
\begin{align*}
    \mathcal{N}_1(\gamma)(x,y) &\coloneqq \Delta\gamma(x,y) - \gamma'(x)(y - x),\qquad
    \mathcal{N}_2(\gamma)(x,y) \coloneqq \Delta\gamma(x,y) - \gamma'(y)(y - x).
\end{align*}

\begin{lemma}
    \label{lem:N}
    Let $\gamma \in H^2$. Then we have
    \begin{alignat*}{2}
        |\mathcal{N}_1(\gamma)(x,y)|
        &\le \frac{1}{2} \mathcal{M}(\gamma'')(x)d_{\mathbb{S}^1}(x,y)^2, &\qquad
        |\mathcal{N}_2(\gamma)(x,y)|
        &\le \frac{1}{2} \mathcal{M}(\gamma'')(y)d_{\mathbb{S}^1}(x,y)^2, \\
        |\partial_x \mathcal{N}_1(\gamma)(x,y)|
        &\le |\gamma''(x)|d_{\mathbb{S}^1}(x,y), &\qquad
        |\partial_x \mathcal{N}_2(\gamma)(x,y)|
        &\le \mathcal{M}(\gamma'')(x)d_{\mathbb{S}^1}(x,y),
    \end{alignat*}
    for any $x, y \in \mathbb{S}^1$.
\end{lemma}

\begin{proof}
    Suppose that $x,y$ are close enough so that the geodesic segment $\Gamma_{x,y}$ is well-defined.
    We parametrize $\Gamma_{x,y}$ by arc length $s \in [0, d_{\mathbb{S}^1}(x,y)]$.
    Then, using the Taylor expansion with integral remainder:
    \begin{align*}
        \mathcal{N}_1(\gamma)(x,y)
        &= (\gamma(y) - \gamma(x)) - \gamma'(x)(y-x)
        = \int_0^{d_{\mathbb{S}^1}(x,y)} \gamma''(s') (d(x,y) - s') \, \mathrm{d}s'.
    \end{align*}
    The term involving the second derivative is estimated by the maximal function as in Definition~\ref{def:HL}:
    \begin{align*}
        \left|\int_0^{d_{\mathbb{S}^1}(x,y)} \gamma''(s') (d(x,y) - s') \, \mathrm{d}s' \right|
        \le \mathcal{M}(\gamma'')(x) \int_0^{d_{\mathbb{S}^1}(x,y)} (d(x,y) - s') \, \mathrm{d}s'
        = \frac{1}{2}\mathcal{M}(\gamma'')(x)d_{\mathbb{S}^1}(x,y)^2.
    \end{align*}
    Since $\partial_x\mathcal{N}_1(\gamma)(x,y)=-\gamma''(x)(y-x)$, we obtain the desired estimate for $\partial_x\mathcal{N}_1(\gamma)$.
    Estimates for $\mathcal{N}_2(\gamma)$ and its derivative are derived similarly.
\end{proof}

Define $g_1(\gamma)(x,y) \coloneqq \nu(\gamma)(x) \cdot \Delta\gamma(x,y)$ and $g_2(\gamma)(x,y) \coloneqq \nu(\gamma)(y) \cdot \Delta\gamma(x,y)$.
Since $\nu(\gamma)(x) \cdot \gamma'(x) = 0$, we have $g_1(\gamma) = \nu(\gamma)(x) \cdot \mathcal{N}_1(\gamma)$.
The following lemma establishes the crucial estimates for $g_1$ and $g_2$.

\begin{lemma}
    \label{lem:g1_g2}
    Suppose that $\gamma \in H^2$ satisfies $|\gamma|_\ast \ge \delta$ for some positive constant $\delta$.
    Let $\phi_1,\ldots,\phi_k, \varphi_1,\ldots,\varphi_k \in H^2$, and let $\eta \in B_{\delta/(2S_1)}(\gamma)$.
    Then, for any $k \in \mathbb{N}$, there exist polynomials $D_i$ and $D_{i,j}$ (in norms and maximal functions) such that the following estimates hold for all $x,y \in \mathbb{S}^1${\rm :}
    \begin{align*}
        &|\mathcal{D}^k g_1(\gamma)[\phi_1,\ldots,\phi_k] - \mathcal{D}^k g_1(\eta)[\varphi_1,\ldots,\varphi_k]|\\
        &\le \left(D_1 \|\gamma - \eta\|_{H^2} + D_2 \mathcal{M}((\gamma - \eta)'')(x) + \sum_{j=1}^k \left(D_{3,j} \|\phi_j - \varphi_j\|_{H^2} + D_{4,j} \mathcal{M}((\phi_j - \varphi_j)'')(x)\right) \right)d_{\mathbb{S}^1}(x,y)^2,\\
        &\abs{\partial_x\left(\mathcal{D}^k g_1(\gamma)[\phi_1,\ldots,\phi_k] - \mathcal{D}^k g_1(\eta)[\varphi_1,\ldots,\varphi_k]\right)}\\
        &\le \left(\tilde{D}_1 \|\gamma - \eta\|_{H^2} + \tilde{D}_2 |(\gamma - \eta)''(x)| + \tilde{D}_3 \mathcal{M}((\gamma - \eta)''(x)) \right.\\
        &\qquad \left. + \sum_{i=1}^k \left( \tilde{D}_{4,i} \|\phi_i - \varphi_i\|_{H^2} + \tilde{D}_{5,i} |(\phi_i - \varphi_i)''(x)| + \tilde{D}_{6,i} \mathcal{M}((\phi_i - \varphi_i)'')(x) \right) \right)d_{\mathbb{S}^1}(x,y).
    \end{align*}
    Similar estimates also hold for $g_2$, where $``(x)"$ is replaced with $``(y)"$.
\end{lemma}

\begin{proof}
    The orthogonality relation $\nu(\gamma)\cdot\gamma'=0$ implies that $g_1(\gamma) = \nu(\gamma) \cdot \mathcal{N}_1(\gamma)$.
    Since $\mathcal{N}_1$ is linear with respect to $\gamma$, its variation satisfies $\mathcal{D}\mathcal{N}_1(\gamma)[\phi]=\mathcal{N}_1(\phi)$ and $\mathcal{D}^j \mathcal{N}_1 = 0$ for $j \ge 2$.
    Applying the general Leibniz rule, the $k$-th variation is given by:
    \begin{align*}
        \mathcal{D}^k g_1(\gamma)[\phi_1,\ldots,\phi_k]
        =\mathcal{D}^k \nu(\gamma)[\phi_1,\ldots,\phi_k]\cdot\mathcal{N}_1(\gamma)
        +\sum_{i=1}^k \mathcal{D}^{k-1} \nu(\gamma)[\phi_1,\ldots,\hat{\phi}_i,\ldots,\phi_k]\cdot\mathcal{N}_1(\phi_i).
    \end{align*}
    By Lemma~\ref{lem:tau-nu_C-inf}, the variations of $\nu$ are bounded in $H^1$ (and thus in $L^\infty$).
    By Lemma~\ref{lem:N}, the terms $\mathcal{N}_1(\gamma)$ and $\mathcal{N}_1(\phi_i)$ satisfy:
    \begin{align*}
        |\mathcal{N}_1(\gamma)(x,y)| \le \frac{1}{2}\mathcal{M}(\gamma'')(x)d_{\mathbb{S}^1}(x,y)^2, \quad
        |\mathcal{N}_1(\phi_i)(x,y)| \le \frac{1}{2}\mathcal{M}(\phi_i'')(x)d_{\mathbb{S}^1}(x,y)^2.
    \end{align*}
    Therefore, every term in the expansion contains the factor $d_{\mathbb{S}^1}(x,y)^2$.
    The Lipschitz continuity estimate follows from the linearity of $\mathcal{N}_1(\gamma-\eta)$ and the estimates for $\nu$.
    The estimates for the derivative $\partial_x$ follow similarly, using $|\partial_x \mathcal{N}_1(\gamma)| \le |\gamma''| d_{\mathbb{S}^1}$.
\end{proof}

\begin{remark}
    \label{rem:notation_mix}
    In the estimates of Lemma~\ref{lem:g1_g2} and the subsequent Lemma~\ref{lem:F}, two types of terms involving second derivatives appear:
    \begin{itemize}
        \item The maximal function terms (e.g., $\mathcal{M}(\gamma'')(x)$) arise from the integral estimates of difference quotients via Lemma~\ref{lem:N}.
        \item The pointwise terms (e.g., $|\gamma''(x)|$) arise directly from the chain rule applied to $\partial_x$.
    \end{itemize}
    Strictly speaking, for $\gamma \in H^2$, the second derivative $\gamma''$ is defined only almost everywhere.
    However, this mixed notation is justified because our final goal is to derive $L^2$-type estimates for $F$.
    Since the maximal operator is bounded on $L^2$ (Proposition~\ref{prop:HL}), both types of terms are controlled by the $H^2$ norm in the $L^2$ sense.
\end{remark}

\begin{lemma}
    \label{lem:F}
    The operator $F$ satisfies $F \in C^\infty(V; H^1)$.
\end{lemma}

\begin{proof}
    Let $\gamma \in V$. The kernel is given by $K(\gamma) = g_1(\gamma) g_2(\gamma) h(\gamma)$.
    Applying the Leibniz rule, the $k$-th variation $\mathcal{D}^k K(\gamma)[\phi_1, \ldots, \phi_k]$ is expressed as a sum of terms of the form
    \begin{align*}
        J_k(x,y) \coloneqq \mathcal{D}^{k_1} g_1(\gamma)[\Phi_1] \cdot \mathcal{D}^{k_2} g_2(\gamma)[\Phi_2] \cdot \mathcal{D}^{k_3} h(\gamma)[\Phi_3],
    \end{align*}
    where $k_1+k_2+k_3=k$.
    By the estimates established in Lemma~\ref{lem:g1_g2}, the terms involving $g_1$ and $g_2$ satisfy
    \begin{align*}
        |\mathcal{D}^{k_i} g_j(\gamma)[\Phi_i](x,y)| \le C_{k_i} d_{\mathbb{S}^1}(x,y)^2 \qquad (j=1,2),
    \end{align*}
    where the constants $C_{k_i}$ depend on the norms of $\gamma$ and $\phi$.
    Similarly, Lemma~\ref{lem:h} implies that the singularity of $h$ is of order $d_{\mathbb{S}^1}(x,y)^{-3}$.
    Combining these estimates, the product $J_k$ is bounded by
    \begin{align*}
        |J_k(x,y)| \le C d_{\mathbb{S}^1}(x,y)^2 \cdot d_{\mathbb{S}^1}(x,y)^2 \cdot d_{\mathbb{S}^1}(x,y)^{-3} = C d_{\mathbb{S}^1}(x,y).
    \end{align*}
    Since the kernel behaves like $O(d_{\mathbb{S}^1})$ as $x \to y$, the singularity is removable, and the integral defining the $k$-th variation converges pointwise.

    To verify the $H^1$-regularity, we consider the derivative $\partial_x$.
    Differentiation of the kernel increases the singularity order by at most $1$. Specifically, Lemma~\ref{lem:h} gives $|\partial_x \mathcal{D}^{k_3} h| \lesssim d_{\mathbb{S}^1}^{-4}$.
    Consequently, the derivative of the kernel $\partial_x \mathcal{D}^k K$ is controlled by
    \begin{align*}
        |\partial_x J_k(x,y)| \lesssim d_{\mathbb{S}^1}^2 \cdot d_{\mathbb{S}^1}^2 \cdot d_{\mathbb{S}^1}^{-4} \sim O(1).
    \end{align*}
    Alternatively, terms involving $\partial_x \mathcal{D}^{k_i} g_j$ behave like $d_{\mathbb{S}^1} \cdot d_{\mathbb{S}^1}^2 \cdot d_{\mathbb{S}^1}^{-3} \sim O(1)$.
    In all cases, the kernel of the derivative is bounded (or involves at most logarithmic singularities handled by the maximal function estimates in Lemma~\ref{lem:g1_g2}).
    
    Applying the $L^2$-boundedness of the Hardy--Littlewood maximal operator (Proposition~\ref{prop:HL}) to the estimates derived in Lemmas~\ref{lem:g1_g2} and \ref{lem:h}, we conclude that
    \begin{align*}
        \|\partial_x \mathcal{D}^k F(\gamma)[\phi_1, \ldots, \phi_k]\|_{L^2} \le C \|\phi_1\|_{H^2} \cdots \|\phi_k\|_{H^2}.
    \end{align*}
    This implies that the Fréchet derivatives of any order exist and are continuous mappings into $H^1$. Thus $F \in C^\infty(V; H^1)$.
\end{proof}

\begin{proposition}
    \label{prop:f}
    The nonlinearity $f$ satisfies $f \in C^\infty(V; H^1)$.
\end{proposition}

\begin{proof}
    Recall that $f(\gamma) = F(\gamma)\nu(\gamma)$.
    We have established that $F \in C^\infty(V; H^1)$ (Lemma~\ref{lem:F}) and $\nu \in C^\infty(V; H^1)$ (Lemma~\ref{lem:tau-nu_C-inf}).
    Since $H^1(\mathbb{S}^1)$ is a Banach algebra, the product map $(F, \nu) \mapsto F\nu$ is smooth. 
    The result follows immediately from the product rule.
\end{proof}

\subsection{Maximal regularity}

Define the function space $\mathbb{E}(0,T)$ by
\begin{align*}
    \mathbb{E}(0,T) \coloneqq H^1(0,T;H^1) \cap L^2(0,T;H^3).
\end{align*}
This space is equipped with the natural inner product
\begin{align*}
    (\eta,\zeta)_{\mathbb{E}(0,T)} \coloneqq (\eta,\zeta)_{H^1(0,T;H^1)} + (\eta,\zeta)_{L^2(0,T;H^3)},
\end{align*}
under which $(\mathbb{E}(0,T),(\cdot,\cdot)_{\mathbb{E}(0,T)})$ becomes a Hilbert space.
It is well known that $\mathbb{E}(0,T)$ is continuously embedded into $C([0,T];H^2)$;
that is, there exists a positive constant $B$ such that
\begin{align}
    \|\eta\|_{C([0,T];H^2)} \le B\|\eta\|_{\mathbb{E}(0,T)}
    \label{eq:E-C_embed}
\end{align}
holds for all $\eta\in\mathbb{E}(0,T)$ \cite[Corollary 1.14 (p.~23)]{lunardi2009interpolation}.

The aim of this subsection is to prove the following proposition concerning maximal $L^2$-regularity.

\begin{proposition}
    \label{prop:MR}
    For any $\eta\in V$, the operator $A(\eta)$ enjoys the property of maximal $L^2$-regularity on $H^1$.
    More precisely, for any $T>0$ and $g\in L^2(0,T;H^1)$, there exists a unique solution $\gamma = \gamma(g) \in \mathbb{E}(0,T)$ to the initial value problem
    \begin{align}
        \begin{dcases*}
            \partial_t \gamma = A(\eta)\gamma + g & in $\mathbb{S}^1\times(0,T)$,\\
            \gamma(0) = 0.
        \end{dcases*}
        \label{eq:problem_MR}
    \end{align}
    Furthermore, there exists a constant $C$, independent of $g$, such that the estimate
    \begin{align}
        \|\gamma\|_{\mathbb{E}(0,T)} \le C\|g\|_{L^2(0,T;H^1)}
        \label{eq:estimate_MR}
    \end{align}
    holds for all $g\in L^2(0,T;H^1)$.
\end{proposition}

To prove the above proposition, we first introduce several notations.
For $\eta\in V$, set $B(\eta)\coloneqq |\partial_x \eta|^{-2}$; that is,
\begin{align*}
    A(\eta) = B(\eta)\partial_x^2.
\end{align*}
Since $B(\eta)\in C^\alpha$ with $\alpha\in[0,1/2)$ by the Sobolev embedding, there exists a positive constant $B_\ast$, depending only on $\eta$, such that
\begin{align}
    B(\eta) \ge B_\ast
    \label{eq:B_ast}
\end{align}
holds. Next, we introduce a family $(A_s(\eta))_{s\in[0,1]}$ of differential operators by linear interpolation:
\begin{align*}
    A_s(\eta)\coloneqq B_s(\eta) \partial_x^2,\qquad\text{where}\quad
    B_s(\eta)\coloneqq(1-s)B_\ast+sB(\eta)\quad(s\in[0,1]).
\end{align*}
Note that the following estimates hold uniformly in $s \in [0,1]$:
\begin{align}
    B_s(\eta) \ge B_\ast,\quad
    \|B_s(\eta)\|_{L^\infty} \le \|B(\eta)\|_{L^\infty},\quad
    \|B_s(\eta)'\|_{L^2} \le \|B(\eta)'\|_{L^2}.
    \label{eq:estimate_B}
\end{align}

For $\mu\in\mathbb{C}$ and $\delta\in(0,\pi/2)$, set
\begin{align*}
    \Sigma_{\mu,\delta} \coloneqq \{\zeta\in\mathbb{C}\mid|\arg(\zeta-\mu)|<\pi-\delta\}.
\end{align*}

\begin{lemma}
    \label{lem:resolvent_estimate}
    Let $\eta\in V$.
    Then, there exist a constant $\lambda\ge0$ and a positive constant $C_{\delta,\eta}$, depending only on $\delta$ and $\eta$, such that
    \begin{align}
        \|(z+\lambda-A_s(\eta))\gamma\|_{H^1} \ge C_{\delta,\eta}\left(|z|\|\gamma\|_{H^1} + \|\gamma\|_{H^3}\right)
        \label{eq:resolvent_estimate}
    \end{align}
    holds for all $z\in\Sigma_{0,\delta}$, $\gamma\in H^3$, and $s\in[0,1]$.
\end{lemma}

\begin{proof}
We divide the proof into three steps.

\textit{Step 1}.
We first introduce a partition of unity on $\mathbb{S}^1$.
Define a family $\{\phi_k\}_{k=1}^N$ of functions on $\mathbb{S}^1$ by
\begin{align*}
    \phi_k(\cdot)\coloneqq\phi(\cdot-x_k),\qquad
    x_k=\frac{k-1}{N},
\end{align*}
where $\phi\in C^\infty(\mathbb{S}^1)$ satisfies
\begin{align}
    \phi\ge0\quad\text{on }\mathbb{S}^1,\qquad
    \supp\phi\subset\left[-\frac{1}{N},\frac{1}{N}\right],\qquad
    \sum_{k=1}^N\phi_k^2\equiv1\quad\text{on }\mathbb{S}^1.
    \label{eq:partition-of-unity}
\end{align}
It is straightforward to verify the identity
\begin{align}
    \sum_{k=1}^N\|\phi_k u\|_{L^2}^2 = \|u\|_{L^2}^2
    \label{eq:sum_phi_k-u_L2}
\end{align}
for all $u\in L^2$.
Furthermore, there exists a positive constant $C_{\phi,N}$, depending only on $\phi$ and $N$, such that the following estimates hold:
\begin{align}
    \sum_{k=1}^N\|\phi_k'\|_{L^\infty}^2\le C_{\phi,N},\qquad
    \sum_{k=1}^N\|\phi_k''\|_{L^\infty}^2 \le C_{\phi,N}.
    \label{eq:phi'_phi''_estimate}
\end{align}

\textit{Step 2}.
In this step, we show that for any $\varepsilon>0$, there exist positive constants $C_{\varepsilon,2}$ and $C_{\varepsilon,\infty}$, depending only on $\varepsilon$, such that the estimates
\begin{align}
    \|v'\|_{L^2}^2 \le C_{\varepsilon,2}\|v\|_{L^2}^2 + \varepsilon\|v''\|_{L^2}^2,\qquad
    \|v'\|_{L^\infty}^2 \le C_{\varepsilon,\infty}\|v\|_{L^2}^2 + \varepsilon\|v''\|_{L^2}^2
    \label{eq:v'_L2_Linf}
\end{align}
hold for all $v\in H^2$.
In particular, we can choose $C_{\varepsilon,2}=1/(4\varepsilon)$ and $C_{\varepsilon,\infty}=27/(16\varepsilon^3)$.

Integrating by parts and applying the Cauchy--Schwarz and Young inequalities, we obtain
\begin{align}
    \|v'\|_{L^2}^2
    &=-\int_{\mathbb{S}^1}v\cdot v''\,\mathrm{d}x
    \le\|v\|_{L^2}\|v''\|_{L^2}
    \le\frac{1}{4\varepsilon}\|v\|_{L^2}^2 + \varepsilon\|v''\|_{L^2}^2.
    \label{eq:first_ineq}
\end{align}
Thus, setting $C_{\varepsilon,2}\coloneqq1/(4\varepsilon)$, we obtain the first desired inequality.

Next, consider the second inequality.
Since $\int_{\mathbb{S}^1}v'=0$, the one-dimensional Gagliardo--Nirenberg interpolation inequality, together with \eqref{eq:first_ineq} and Young's inequality, yields
\begin{align*}
    \|v'\|_{L^\infty}^2
    \le2\|v'\|_{L^2}\|v''\|_{L^2}
    \le2\left(\|v\|_{L^2}\|v''\|_{L^2}\right)^{1/2}\|v''\|_{L^2}
    =2\|v\|_{L^2}^{1/2}\|v''\|_{L^2}^{3/2}
    \le\frac{27}{16\varepsilon^3}\|v\|_{L^2}^2 + \varepsilon\|v''\|_{L^2}^2,
\end{align*}
which is precisely the second inequality with $C_{\varepsilon,\infty}\coloneqq27/(16\varepsilon^3)$. This completes Step~2.

Note that there exists a positive constant $C_\delta$, depending only on $\delta$, such that
\begin{align}
    |z+a|\ge C_\delta(|z|+a)
    \label{eq:abs_below}
\end{align}
holds for all $z\in\Sigma_{0,\delta}$ and $a\ge0$.
In particular, we can choose $C_\delta=\sqrt{(1-\cos\delta)/2}$.

For $z\in\Sigma_{0,\delta}$, $\lambda\ge0$, $s\in[0,1]$, and $\gamma\in H^3$, set
\begin{align*}
    f\coloneqq(z+\lambda)\gamma-A_s(\eta)\gamma=(z+\lambda)\gamma - B_s(\eta)\gamma''.
\end{align*}
Differentiating this equation, we obtain
\begin{align}
    f'=(z+\lambda)v - B_s(\eta)'v' - B_s(\eta)v'',
    \label{eq:f'}
\end{align}
where $v\coloneqq\gamma'$.

\textit{Step 3}.
In this step, we show that there exists a constant $C_{\delta,\eta}$, depending only on $\delta$ and $\eta$, such that
\begin{align*}
    \|f\|_{H^1}^2 \ge C_{\delta,\eta}^2\left(|z|\|\gamma\|_{H^1} + \|\gamma\|_{H^3}\right)^2
\end{align*}
for sufficiently large $\lambda$.
Let $B_k\coloneqq B(\eta)|_{x=x_k}$.
Multiplying both sides of \eqref{eq:f'} by $\phi_k$ gives
\begin{align*}
    (z+\lambda)\phi_kv - B_k(\phi_kv)''
    =\phi_kf'+B_s(\eta)'\phi_kv'+\phi_k(B_s(\eta)-B_k)v''-2B_k\phi_k'v'-B_k\phi_k''v
    \eqqcolon b_k.
\end{align*}
Applying \eqref{eq:abs_below} and the Cauchy--Schwarz inequality, we obtain
\begin{align*}
    \|b_k\|_{L^2}^2
    &= \|(z+\lambda)\phi_kv - B_s(\eta)(\phi_kv)''\|_{L^2}^2
    =\sum_{n\in\mathbb{Z}}|z+\lambda+B_k(2\pi n)^2|^2|(\phi_kv)_n|^2\\
    &\ge\sum_{n\in\mathbb{Z}}C_\delta^2(|z|+\lambda+B_k(2\pi n)^2)^2|(\phi_kv)_n|^2
    \ge C_\delta^2\left[(|z|+\lambda)^2\|\phi_kv\|_{L^2}^2 + B_\ast^2\|(\phi_kv)''\|_{L^2}^2\right].
\end{align*}
Using the identity $(\phi_kv)'' = \phi_kv''+2\phi_k'v'+\phi_k''v$, Young's inequality implies
\begin{align*}
    \|(\phi_kv)''\|_{L^2}^2
    \ge\frac{1}{2}\|\phi_kv''\|_{L^2}^2 - 13\|\phi_k'v'\|_{L^2}^2 - 7\|\phi_k''v\|_{L^2}^2.
\end{align*}
The estimates \eqref{eq:phi'_phi''_estimate} and \eqref{eq:v'_L2_Linf} imply that
\begin{align*}
    \sum_{k=1}^N\|\phi_k'v'\|_{L^2}^2
    \le C_{\phi,N}\left(C_{\varepsilon,2}\|v\|_{L^2}^2 + \varepsilon\|v''\|_{L^2}^2\right),\qquad
    \sum_{k=1}^N\|\phi_k''v\|_{L^2}^2 \le C_{\phi,N}\|v\|_{L^2}^2.
\end{align*}
Combining these with \eqref{eq:sum_phi_k-u_L2}, we have
\begin{align}
    \sum_{k=1}^N\|b_k\|_{L^2}^2
    &\ge
    C_\delta^2\left[(|z|+\lambda)^2 - 13B_\ast^2C_{\phi,N}C_{\varepsilon,2} - 7B_\ast^2C_{\phi,N}\right]\|v\|_{L^2}^2 + C_\delta^2B_\ast^2\left(\frac{1}{2}-13C_{\phi,N}\varepsilon\right)\|v''\|_{L^2}^2.
    \label{eq:sum_bk_lower}
\end{align}
Next, we consider an upper bound for $\sum_{k=1}^N\|b_k\|_{L^2}^2$.
It follows from \eqref{eq:sum_phi_k-u_L2} and \eqref{eq:v'_L2_Linf} that
\begin{align*}
    \sum_{k=1}^N\|\phi_k(f'+B_s(\eta)'v')\|_{L^2}^2
    &=\|f'+B_s(\eta)'v'\|_{L^2}^2
    \le2\|f'\|_{L^2}^2 + 2\|B_s(\eta)'\|_{L^2}^2\|v'\|_{L^\infty}^2\\
    &\le2\|f'\|_{L^2}^2 + 2\|B(\eta)'\|_{L^2}^2\left(C_{\varepsilon,\infty}\|v\|_{L^2}^2 + \varepsilon\|v''\|_{L^2}^2\right).
\end{align*}
Since $\phi_k$ is supported in $[x_k-N^{-1},x_k+N^{-1}]$, it follows from \eqref{eq:sum_phi_k-u_L2} that
\begin{align*}
    \sum_{k=1}^N\|\phi_k(B_s(\eta)-B_k)v''\|_{L^2}^2
    &\le\sum_{k=1}^N\sup_{|x-x_k|\le N^{-1}}|B_s(\eta)-B_k|^2\|\phi_kv''\|_{L^2}^2
    \le\omega_{B,N}\|v''\|_{L^2}^2,
\end{align*}
where we set $\omega_{B,N}\coloneqq\sup_{|x-y|\le N^{-1}}|B(\eta(x))-B(\eta(y))|^2$.
Therefore, using Young's inequality and \eqref{eq:phi'_phi''_estimate}, we have
\begin{align*}
    \sum_{k=1}^N\|b_k\|_{L^2}^2
    &\le\sum_{k=1}^N\left[4\|\phi_k(f'+B_s(\eta)'v')\|_{L^2}^2 + 4\|\phi_k(B_s(\eta)-B_k)v''\|_{L^2}^2 + 16\|B_k\phi_k'v'\|_{L^2}^2 + 4\|B_k\phi_k''v\|_{L^2}^2\right]\\
    &\le8\|f'\|_{L^2}^2 + \left[8C_{\varepsilon,\infty}\|B(\eta)'\|_{L^2}^2 + 4C_{\phi,N}\|B(\eta)\|_{L^\infty}^2\left(4C_{\varepsilon,2} + 1\right)\right]\|v\|_{L^2}^2\\
    &\quad+\left[4\omega_{B,N} + \left(8\|B(\eta)'\|_{L^2} + 16C_{\phi,N}\|B(\eta)\|_{L^\infty}^2\right)\varepsilon\right]\|v''\|_{L^2}^2.
\end{align*}
Hence, combining with \eqref{eq:sum_bk_lower}, we obtain
\begin{align*}
    8\|f'\|_{L^2}^2
    &\ge(|z|+\lambda)^2\left(C_\delta^2 - \frac{B_\ast^2C_{\phi,N}C_\delta^2(13C_{\varepsilon,2}+7) + 8C_{\varepsilon,\infty}\|B(\eta)'\|_{L^2}^2 + 4C_{\phi,N}\|B(\eta)\|_{L^\infty}^2(4C_{\varepsilon,2}+1)}{\lambda^2}\right)\|v\|_{L^2}^2\\
    &\quad+\left(\frac{1}{2}C_\delta^2B_\ast^2 - 4\omega_{B,N} - \left(13B_\ast^2C_{\phi,N}C_\delta^2 + 8\|B(\eta)'\|_{L^2}^2 + 16C_{\phi,N}\|B(\eta)\|_{L^\infty}^2\right)\varepsilon\right)\|v''\|_{L^2}^2.
\end{align*}
Since $B(\eta)\in H^1\hookrightarrow C^\alpha$ with $\alpha\in[0,1/2)$ by Sobolev embedding, $B(\eta)$ is, in particular, uniformly continuous.
Therefore, taking $N$ sufficiently large and fixing it, we can ensure that $4\omega_{B,N} \le C_\delta^2B_\ast^2 / 8$.
Then, by choosing $\varepsilon$ sufficiently small and fixing it, we can ensure that the coefficient of $\|v''\|_{L^2}^2$ is bounded from below by $\frac{1}{4}C_\delta^2B_\ast^2$.
Accordingly, by choosing $\lambda\ge1$ sufficiently large, we see that the coefficient of $\|v\|_{L^2}^2$ is bounded from below by $C_\delta^2/2$.
Thus, there exists a constant $C_{\delta,\eta}$, depending only on $\delta$ and $\eta$, such that
\begin{align}
    \|f'\|_{L^2}^2 \ge 2C_{\delta,\eta}^2\left((|z|+\lambda)^2\|v\|_{L^2}^2 + \|v''\|_{L^2}^2\right)
    \label{eq:f'_lower}
\end{align}
holds.
In a similar but simpler manner (working directly with $f$ instead of $f'$), we can also show that
\begin{align}
    \|f\|_{L^2}^2 \ge 2C_{\delta,\eta}^2\left((|z|+\lambda)^2\|\gamma\|_{L^2}^2 + \|\gamma''\|_{L^2}^2\right)
    \label{eq:f_lower}
\end{align}
holds by choosing $C_{\delta,\eta}$ smaller if necessary.
Noting that $v=\gamma'$ and thus $\|v''\|_{L^2} = \|\gamma'''\|_{L^2}$, we combine \eqref{eq:f'_lower} with \eqref{eq:f_lower} to obtain
\begin{align*}
    \|f\|_{H^1}^2
    &\ge2C_{\delta,\eta}^2\left[(|z|+\lambda)^2(\|\gamma\|_{L^2}^2 + \|\gamma'\|_{L^2}^2) + \|\gamma''\|_{L^2}^2 + \|\gamma'''\|_{L^2}^2\right]\\
    &\ge2C_{\delta,\eta}^2\left(|z|^2\|\gamma\|_{H^1}^2 + \|\gamma\|_{H^3}^2\right)
    \ge C_{\delta,\eta}^2\left(|z|\|\gamma\|_{H^1} + \|\gamma\|_{H^3}\right)^2,
\end{align*}
where we have used $\lambda\ge1$ in the second inequality.
\end{proof}

\begin{lemma}
    \label{lem:resolvent_set}
    Let $\eta\in V$.
    Then, the set $\Sigma_{\lambda,\delta}=\Sigma_{0,\delta}+\lambda$ is contained in the resolvent set of $A(\eta)$ viewed as an operator from $H^3\subset H^1$ to $H^1$, where $\lambda$ is introduced in Lemma~\ref{lem:resolvent_estimate}.
\end{lemma}

\begin{proof}
    Let $\zeta\in\Sigma_{\lambda,\delta}$.
    A direct calculation shows that $\zeta-A_0(\eta) = \zeta - B_\ast \partial_x^2$ is invertible.
    Suppose that $\zeta-A_s(\eta)$ is invertible for some $s\in[0,1]$.
    Then, Lemma~\ref{lem:resolvent_estimate} implies that
    \begin{align}
        \left\|(\zeta-A_s(\eta))^{-1}u\right\|_{H^3} \le C_{\delta,\eta}^{-1}\|u\|_{H^1}
        \label{eq:bounded_operator_norm}
    \end{align}
    holds for all $u\in H^1$, which shows that the norm of $(\zeta-A_s(\eta))^{-1}$ as an operator from $H^1$ to $H^3$ is bounded by $C_{\delta,\eta}^{-1}$, independently of $s$.
    Let $s'\in[0,1]$ satisfy $|s-s'|\le C_{\delta,\eta}/(4B_\ast)$.
    Since $\zeta-A_s(\eta)$ is invertible, we have
    \begin{align*}
        \zeta-A_{s'}(\eta)
        =(\zeta-A_s(\eta))\left[
            \operatorname{id} + (\zeta-A_s(\eta))^{-1}(A_s(\eta)-A_{s'}(\eta))
        \right].
    \end{align*}
    A direct calculation shows that $A_s(\eta)-A_{s'}(\eta) = -(s-s')(B_\ast-B(\eta))\partial_x^2$, so it follows from \eqref{eq:bounded_operator_norm} that
    \begin{align*}
        \|(\zeta-A_s(\eta))^{-1}(A_s(\eta)-A_{s'}(\eta))\|_{\mathcal{B}(H^3)}
        &\le\|(\zeta-A_s(\eta))^{-1}\|_{\mathcal{B}(H^1,H^3)}\|A_s(\eta)-A_{s'}(\eta)\|_{\mathcal{B}(H^3,H^1)}\\
        &\le C_{\delta,\eta}^{-1} \cdot 2B_\ast\|\partial_x^2\|_{\mathcal{B}(H^3,H^1)}|s-s'|
        \le\frac{1}{2}.
    \end{align*}
    Therefore, by a Neumann series argument, $\zeta-A_{s'}(\eta)$ is invertible.
    Since $\zeta-A_0(\eta)$ is invertible, we conclude, by applying the Neumann series argument a finite number of times, that $\zeta-A_1(\eta)$ is also invertible, with domain $H^3$.
\end{proof}

\begin{corollary}
    \label{cor:analytic_semigroup}
    For any $\eta\in V$, the operator $A(\eta)$ generates an analytic semigroup on $H^1$ with domain $H^3$.
\end{corollary}

\begin{proof}
    Let $\eta\in V$. Then, it follows from Lemma~\ref{lem:resolvent_estimate} that there exists a $\lambda\ge0$ such that the estimate \eqref{eq:resolvent_estimate} holds.
    Furthermore, Lemma~\ref{lem:resolvent_set} yields that the set $\Sigma_{\lambda,\delta}$ is contained in the resolvent set of $A(\eta)\colon H^3\subset H^1\to H^1$.
    The estimate \eqref{eq:resolvent_estimate} yields 
    \begin{align*}
        |z|\|(z+\lambda-A(\eta))^{-1}u\|_{H^1} \le C_{\delta,\eta}^{-1}\|u\|_{H^1}
    \end{align*}
    for any $z\in\Sigma_{0,\delta}$ and $u\in H^1$.
    Namely, we have
    \begin{align*}
        \|(\zeta-A(\eta))^{-1}\|_{\mathcal{B}(H^1)} \le \frac{C_\delta^{-1}}{|\zeta-\lambda|}
    \end{align*}
    for any $\zeta\in\Sigma_{\lambda,\delta}$.
    This shows that $A(\eta)$ is a sectorial operator, and it generates an analytic semigroup on $H^1$ with domain $H^3$.
\end{proof}

\begin{proof}[Proof of Proposition~\ref{prop:MR}]
    The fact that $A(\eta)$ generates an analytic semigroup (Corollary~\ref{cor:analytic_semigroup}) guarantees maximal regularity for Hilbert spaces \cite{deSimon1964Un’applicazione}.
    Namely, for any $T>0$ and $g\in L^2(0,T;H^1)$, there exists a unique solution $\gamma$ to the problem \eqref{eq:problem_MR}.
    We derive the estimate \eqref{eq:estimate_MR} below.

    Consider the equation
    \begin{align}
        \begin{dcases*}
            \partial_t \gamma = (A(\eta)-\lambda)\gamma + g & in $(0,T]$,\\
            \gamma(0) = 0,
        \end{dcases*}
        \label{eq:problem_modified}
    \end{align}
    where $\lambda$ is the constant from Lemma~\ref{lem:resolvent_estimate} and $g\in L^2(0,T;H^1)$.
    We extend $g$ by zero to $\mathbb{R}$ and define the Fourier transform in time by $\hat{\gamma}(\omega) \coloneqq \int_{-\infty}^\infty\gamma(t)\mathrm{e}^{-\mathrm{i}\omega t}\,\mathrm{d}t$.
    Taking the Fourier transform of \eqref{eq:problem_modified} in time, we obtain
    \begin{align*}
        (\mathrm{i}\omega + \lambda - A(\eta))\hat{\gamma}(\omega) = \hat{g}(\omega).
    \end{align*}
    Then, it follows from \eqref{eq:f_lower} and \eqref{eq:f'_lower} (with $z=\mathrm{i}\omega$) that 
    \begin{align*}
        2C_{\delta,\eta}^2\left(|\omega|+\lambda\right)^2\|\hat{\gamma}(\omega)\|_{H^1}^2 \le \|\hat{g}(\omega)\|_{H^1}^2
    \end{align*}
    for $\omega\in\mathbb{R}$; that is,
    \begin{align*}
        \|\hat{\gamma}\|_{L_\omega^2(H^1)} \le C_{\delta,\eta}^{-1}\|\hat{g}\|_{L_\omega^2(H^1)}.
    \end{align*}
    Lemma~\ref{lem:resolvent_estimate} further gives 
    \begin{align*}
        C_{\delta,\eta}\|\mathrm{i}\omega\hat{\gamma}(\omega)\|_{H^1}\le\|\hat{g}(\omega)\|_{H^1}\quad(\omega\in\mathbb{R});\quad\text{that is,}\quad \|\mathrm{i}\omega\hat{\gamma}\|_{L_\omega^2(H^1)} \le C_{\delta,\eta}^{-1}\|\hat{g}\|_{L_\omega^2(H^1)}.
    \end{align*}
    Likewise, using the estimates \eqref{eq:f'_lower} and \eqref{eq:f_lower}, we have
    \begin{align*}
        \|(z+\lambda-A(\eta))\gamma\|_{H^1}^2 \ge 2C_{\delta,\eta}^2\left((|z|+\lambda)^2\|\gamma\|_{H^1}^2 + \|\gamma''\|_{H^1}^2\right)
    \end{align*}
    for all $z\in\Sigma_{0,\delta}$.
    Applying this inequality with $\gamma$ and $z$ replaced by $\hat{\gamma}(\omega)$ and $\mathrm{i}\omega$, respectively, we obtain
    \begin{align*}
        C_{\delta,\eta}^2\|(\mathrm{i}\omega+\lambda)\hat{\gamma}(\omega)\|_{H^1}^2 \le \|(\mathrm{i}\omega+\lambda-A(\eta))\hat{\gamma}(\omega)\|_{H^1}^2 = \|\hat{g}(\omega)\|_{H^1}^2,
    \end{align*}
    which leads to
    \begin{align*}
        \|(\mathrm{i}\omega+\lambda)\hat{\gamma}\|_{L_\omega^2(H^1)} \le C_{\delta,\eta}^{-1}\|\hat{g}\|_{L_\omega^2(H^1)}.
    \end{align*}
    Therefore, we see that
    \begin{align*}
        \|A(\eta)\hat{\gamma}\|_{L_\omega^2(H^1)}
        =\|(\mathrm{i}\omega+\lambda)\hat{\gamma} - \hat{g}\|_{L_\omega^2(H^1)}
        \le\left(1+C_{\delta,\eta}^{-1}\right)\|\hat{g}\|_{L_\omega^2(H^1)}.
    \end{align*}
    Applying Plancherel's theorem, we obtain 
    \begin{gather*}
        \|\gamma\|_{L^2(0,T;H^1)} \le C_{\delta,\eta}^{-1}\|g\|_{L^2(0,T;H^1)},\qquad
        \|\partial_t \gamma\|_{L^2(0,T;H^1)} \le C_{\delta,\eta}^{-1}\|g\|_{L^2(0,T;H^1)},\\
        \|A(\eta)\gamma\|_{L^2(0,T;H^1)} \le \left(1+C_{\delta,\eta}^{-1}\right)\|g\|_{L^2(0,T;H^1)}.
    \end{gather*}
    Since $A(\eta)=B(\eta)\partial_x^2$ and $B(\eta)\ge B_\ast$ holds as in \eqref{eq:B_ast}, the estimate for $A(\eta)\gamma$ implies $\|\gamma\|_{L^2(0,T;H^3)} \le C \|g\|_{L^2(0,T;H^1)}$.
    Consequently, there exists a constant $\hat{C}$, independent of $\gamma$ and $T$, such that
    \begin{align}
        \|\gamma\|_{\mathbb{E}(0,T)} \le \hat{C}\|g\|_{L^2(0,T;H^1)}.
        \label{eq:estimate_lambda}
    \end{align}

    Finally, consider the original problem
    \begin{align*}
        \begin{dcases*}
            \partial_t \gamma = A(\eta)\gamma + g & in $\mathbb{S}^1\times(0,T]$,\\
            \gamma(0) = 0.
        \end{dcases*}
    \end{align*}
    Let $\tilde{\gamma}(t)\coloneqq\mathrm{e}^{-\lambda t}\gamma(t)$ and $\tilde{g}(t) \coloneqq \mathrm{e}^{-\lambda t}g(t)$.
    Then, $\tilde{\gamma}$ satisfies the modified equation \eqref{eq:problem_modified} with source term $\tilde{g}$.
    The estimate~\eqref{eq:estimate_lambda} then yields
    \begin{align*}
        \|\tilde{\gamma}\|_{\mathbb{E}(0,T)} \le \hat{C}\|\tilde{g}\|_{L^2(0,T;H^1)},
    \end{align*}
    which implies
    \begin{align*}
        \|\gamma\|_{\mathbb{E}(0,T)} \le \sqrt{1+2\lambda^2}\mathrm{e}^{\lambda T}\hat{C}\|g\|_{L^2(0,T;H^1)}.
    \end{align*}
    This completes the proof.
\end{proof}

Let us consider the problem \eqref{eq:problem_MR} with general initial data;
that is,
\begin{align}
    \begin{dcases*}
        \partial_t \gamma + A(\eta)\gamma = g&in $(0,T]$,\\
        \gamma(0) = \gamma_0,
    \end{dcases*}
    \label{eq:problem_MR_general}
\end{align}
where $\gamma_0 \in H^2$. In this case, we decompose the above problem into two problems: the inhomogeneous problem~\eqref{eq:problem_MR} with zero initial data and the homogeneous problem with initial data $\gamma_0$.
For the latter problem, using the trace method of interpolation spaces, we see that there exists a positive constant $\tilde{C}$, independent of $\gamma_0$, such that the estimate
\begin{align*}
    \|\tilde{\gamma}\|_{\mathbb{E}(0,T)} \le \tilde{C}\|\gamma_0\|_{H^2}
\end{align*}
holds for all $\gamma_0 \in H^2 = (H^1, H^3)_{1/2,2}$. Combining the above estimate with the maximal $L^2$-regularity (Proposition~\ref{prop:MR}) implies the existence of a constant $C_{\mathrm{MR}}>0$, independent of $\gamma$, such that
\begin{align}
    \|\gamma\|_{\mathbb{E}(0,T)}\le C_{\mathrm{MR}}(\|g\|_{L^2(0,T;H^1)}+\|\gamma_0\|_{H^2})
    \label{eq:MR_ineq}
\end{align}
holds for all $g\in L^2(0,T;H^1)$ and $\gamma_0\in H^2$, where $\gamma$ denotes the unique solution to \eqref{eq:problem_MR_general}.
Moreover, for any fixed initial data $\gamma_0 \in H^2$, the solution $\gamma$ satisfies
\begin{align}
    \|\gamma\|_{\mathbb{E}(0,T)}\to0\qquad\text{as }T\searrow0,
    \label{eq:MR_ineq_limit}
\end{align}
since the norm of $\mathbb{E}(0,T)$ is defined by time integrals over $(0,T)$.

The proof of Corollary~\ref{cor:analytic_semigroup} implies that there exists a positive constant $M$ such that
\begin{align}
    \|\mathrm{e}^{tA(\eta)}\|_{\mathcal{B}(H^2)}\le M\mathrm{e}^{\lambda t}
    \label{eq:semigroup_estimate_2}
\end{align}
holds for all $t>0$.
This follows from the fact that $H^2$ is the interpolation space $(H^1,H^3)_{\frac{1}{2},2}$.

\section{Well-posedness and regularity}
\label{sec:well-posedness}

\subsection{Well-posedness}

In this subsection, we establish the local well-posedness of the modified problem \eqref{eq:Langmuir_rewrite}.
Recall that the problem is formulated as an abstract evolution equation:
\begin{equation}
    \begin{dcases*}
        \partial_t \gamma = A(\gamma)\gamma + f(\gamma) & in $\mathbb{S}^1\times(0,T]$,\\
        \gamma(0) = \gamma_1.
    \end{dcases*}
    \label{eq:abstract_IVP}
\end{equation}

\begin{theorem} \label{thm:existence}
    Let $\gamma_0\in V$.
    Then there exist $T=T(\gamma_0)>0$ and $\varepsilon=\varepsilon(\gamma_0)>0$ such that $B_\varepsilon(\gamma_0)\subset V$, and for every $\gamma_1\in B_\varepsilon(\gamma_0)$, the problem \eqref{eq:abstract_IVP} admits a unique solution
    \begin{align*}
        \gamma(\cdot)=\gamma(\cdot;\gamma_1)\in \mathbb{E}(0,T) \cap C([0,T];V).
    \end{align*}
    Furthermore, the solution map is Lipschitz continuous with respect to the initial data: there exists a constant $C=C(\gamma_0)>0$ such that
    \begin{align*}
        \|\gamma(\cdot;\gamma_1)-\gamma(\cdot;\gamma_2)\|_{\mathbb{E}(0,T)}\le C\|\gamma_1-\gamma_2\|_{H^2}
    \end{align*}
    holds for all $\gamma_1,\gamma_2\in B_\varepsilon(\gamma_0)$.
\end{theorem}

\begin{proof}
    Since $\gamma_0\in V$, there exists a constant $\delta>0$ such that the chord-arc constant satisfies $|\gamma_0|_\ast\ge\delta$.
    Set $\varepsilon_0\coloneqq\delta/(2S_1)$. By the proof of Lemma~\ref{lem:V_open}, we have $B_{\varepsilon_0}(\gamma_0)\subset V$.
    
    From Propositions~\ref{prop:A_C-inf} and \ref{prop:f}, we know that $A \in C^\infty(V; \mathcal{B}(H^3; H^1))$ and $f \in C^\infty(V; H^1)$.
    Since $B_{\varepsilon_0}(\gamma_0)$ is a convex subset of $V$, both $A$ and $f$ are Lipschitz continuous on this ball.
    Specifically, there exist positive constants $L_A$ and $L_f$ such that
    \begin{align}
        \|(A(\eta_1)-A(\eta_2))\xi\|_{H^1}&\le L_A\|\eta_1-\eta_2\|_{H^2}\|\xi\|_{H^3},\label{eq:A_Lipschitz}\\
        \|f(\eta_1)-f(\eta_2)\|_{H^1}&\le L_f\|\eta_1-\eta_2\|_{H^2}, \label{eq:F_Lipschitz}
    \end{align}
    for all $\eta_1,\eta_2\in B_{\varepsilon_0}(\gamma_0)$ and $\xi\in H^3$.

    Fix an arbitrary reference time $T'>0$.
    Let $A_0 \coloneqq A(\gamma_0)$. By Proposition~\ref{prop:MR}, $A_0$ generates an analytic semigroup and enjoys maximal $L^2$-regularity.
    Let $\gamma_0^\ast \in \mathbb{E}(0,T')$ be the unique solution to the linearized problem with initial data $\gamma_0$:
    \begin{equation*}
        \partial_t \eta = A_0 \eta, \quad \eta(0) = \gamma_0.
    \end{equation*}
    
    We construct the solution using the Banach fixed point theorem.
    For $r>0$, $T\in(0,T']$, and $\gamma_1\in B_{\varepsilon_0}(\gamma_0)$, define the closed metric space $\mathbb{B}_{r,T,\gamma_1} \subset \mathbb{E}(0,T)$ by
    \begin{align*}
        \mathbb{B}_{r,T,\gamma_1} \coloneqq \left\{ \eta \in \mathbb{E}(0,T) \relmiddle| \eta(0)=\gamma_1,\ \|\eta-\gamma_0^\ast\|_{\mathbb{E}(0,T)} \le r \right\}.
    \end{align*}
    Also, let $\gamma_1^\ast \in \mathbb{E}(0,T')$ be the solution to $\partial_t \eta = A_0 \eta$ with $\eta(0) = \gamma_1$.

    \textit{Step 1: Invariance of the admissible set.}
    We show that for sufficiently small parameters, any function in the ball remains in $V$.
    Let $\eta \in \mathbb{B}_{r,T,\gamma_1}$.
    Using the triangle inequality and the embedding $\mathbb{E}(0,T) \hookrightarrow C([0,T]; H^2)$ (with constant $B$), we estimate the distance from $\gamma_0$:
    \begin{align*}
        \|\eta(t) - \gamma_0\|_{H^2}
        &\le \|\eta(t) - \gamma_1^\ast(t)\|_{H^2} + \|\gamma_1^\ast(t) - \gamma_0^\ast(t)\|_{H^2} + \|\gamma_0^\ast(t) - \gamma_0\|_{H^2} \\
        &\le B \|\eta - \gamma_1^\ast\|_{\mathbb{E}(0,T)} + \|\mathrm{e}^{tA_0}(\gamma_1 - \gamma_0)\|_{H^2} + \|\gamma_0^\ast(t) - \gamma_0\|_{H^2}.
    \end{align*}
    Note that 
    \begin{align*}
        \|\eta - \gamma_1^\ast\|_{\mathbb{E}(0,T)} \le \|\eta - \gamma_0^\ast\|_{\mathbb{E}(0,T)} + \|\gamma_0^\ast - \gamma_1^\ast\|_{\mathbb{E}(0,T)} \le r + C_{\mathrm{MR}}\|\gamma_0 - \gamma_1\|_{H^2}.
    \end{align*}
    Also, $\|\mathrm{e}^{tA_0}\|_{\mathcal{B}(H^2)} \le M \mathrm{e}^{\lambda T'}$.
    Thus, assuming $\gamma_1 \in B_\varepsilon(\gamma_0)$, we have
    \begin{align}
        \|\eta - \gamma_0\|_{C([0,T];H^2)}
        \le B(r + C_{\mathrm{MR}}\varepsilon) + M \mathrm{e}^{\lambda T'} \varepsilon + \|\gamma_0^\ast - \gamma_0\|_{C([0,T];H^2)}.
        \label{eq:eta-gamma0_ineq}
    \end{align}
    Since $\gamma_0^\ast(t) \to \gamma_0$ in $H^2$ as $t \to 0$, there exists $T_0 \in (0, T']$ such that $\|\gamma_0^\ast - \gamma_0\|_{C([0,T_0];H^2)} \le \varepsilon_0/3$.
    By choosing $r \le \varepsilon_0/(3B)$ and $\varepsilon \le \varepsilon_0 / (3(BC_{\mathrm{MR}} + M\mathrm{e}^{\lambda T'}))$, we ensure
    \begin{align*}
        \|\eta(t) - \gamma_0\|_{H^2} < \varepsilon_0 \quad \text{for all } t \in [0, T],
    \end{align*}
    which implies $\eta(t) \in B_{\varepsilon_0}(\gamma_0) \subset V$.

    \textit{Step 2: Definition of the mapping.}
    For $\gamma_1 \in B_\varepsilon(\gamma_0)$ and $\eta \in \mathbb{B}_{r,T,\gamma_1}$ (with parameters chosen as above), we define $\mathcal{T}_{\gamma_1}\eta$ as the solution $v$ to the linear problem:
    \begin{equation}
        \begin{dcases*}
            \partial_t v = A_0 v + f(\eta) + (A(\eta) - A_0)\eta & in $(0,T]$, \\
            v(0) = \gamma_1.
        \end{dcases*}
        \label{eq:sub-prob}
    \end{equation}
    Since $\eta(t) \in V$, the terms $f(\eta)$ and $(A(\eta)-A_0)\eta$ are well-defined.
    Using the Lipschitz bounds \eqref{eq:A_Lipschitz}--\eqref{eq:F_Lipschitz}, $f(\eta) \in L^2(0,T;H^1)$ and $(A(\eta)-A_0)\eta \in L^2(0,T;H^1)$ (since $\eta \in L^2(0,T;H^3)$).
    Thus, by Proposition~\ref{prop:MR}, there exists a unique solution $v \in \mathbb{E}(0,T)$. Hence $\mathcal{T}_{\gamma_1}$ is well-defined.

    \textit{Step 3: Self-mapping property.}
    We assume $T \in (0, T_0]$. The solution can be written as
    \begin{align*}
        \mathcal{T}_{\gamma_1}\eta = \gamma_1^\ast + \mathrm{e}^{\bullet A_0} \ast \left( f(\eta) + (A(\eta) - A_0)\eta \right).
    \end{align*}
    Subtracting $\gamma_0^\ast$, we estimate the norm in $\mathbb{E}(0,T)$:
    \begin{align*}
        \|\mathcal{T}_{\gamma_1}\eta - \gamma_0^\ast\|_{\mathbb{E}(0,T)}
        \le \|\gamma_1^\ast - \gamma_0^\ast\|_{\mathbb{E}(0,T)} + C_{\mathrm{MR}} \left( \|f(\eta)\|_{L^2(0,T;H^1)} + \|(A(\eta) - A_0)\eta\|_{L^2(0,T;H^1)} \right).
    \end{align*}
    The first term is bounded by $C_{\mathrm{MR}}\varepsilon$.
    For the integral terms, we use \eqref{eq:eta-gamma0_ineq} (which is bounded by $\varepsilon_0$) and the Lipschitz properties.
    \begin{align*}
        \|(A(\eta) - A_0)\eta\|_{L^2(0,T;H^1)} 
        &\le L_A \sup_{t \in [0,T]} \|\eta(t) - \gamma_0\|_{H^2} \|\eta\|_{L^2(0,T;H^3)} \\
        &\le L_A \varepsilon_0 \left( r + \|\gamma_0^\ast\|_{L^2(0,T;H^3)} \right).
    \end{align*}
    Since $\|\gamma_0^\ast\|_{L^2(0,T;H^3)} \to 0$ as $T \to 0$ (maximal regularity with zero forcing), we can make this term arbitrarily small by choosing $T$ small.
    Similarly,
    \begin{align*}
        \|f(\eta)\|_{L^2(0,T;H^1)} 
        &\le \|f(\eta) - f(\gamma_0)\|_{L^2(0,T;H^1)} + \|f(\gamma_0)\|_{L^2(0,T;H^1)} \\
        &\le L_f T^{1/2} \|\eta - \gamma_0\|_{C([0,T];H^2)} + T^{1/2} \|f(\gamma_0)\|_{H^1}.
    \end{align*}
    Both terms contain a factor of $T^{1/2}$ or vanish as $T \to 0$.
    Thus, by taking $T$ and $\varepsilon$ sufficiently small (keeping $r$ fixed), we can ensure $\|\mathcal{T}_{\gamma_1}\eta - \gamma_0^\ast\|_{\mathbb{E}(0,T)} \le r$.
    Hence, $\mathcal{T}_{\gamma_1}$ maps $\mathbb{B}_{r,T,\gamma_1}$ into itself.

    \textit{Step 4: Contraction property.}
    Let $\eta_1, \eta_2 \in \mathbb{B}_{r,T,\gamma_1}$. The triangle inequality leads to
    \begin{align*}
        &\|\mathcal{T}_{\gamma_1}\eta_1 - \mathcal{T}_{\gamma_1}\eta_2\|_{\mathbb{E}(0,T)} \\
        &\le C_{\mathrm{MR}} \left( \|f(\eta_1) - f(\eta_2)\|_{L^2(H^1)} + \|(A(\eta_1) - A_0)\eta_1 - (A(\eta_2) - A_0)\eta_2\|_{L^2(H^1)} \right).
    \end{align*}
    Using the Lipschitz estimates, we have
    \begin{align*}
        \|f(\eta_1) - f(\eta_2)\|_{L^2(H^1)} 
        &\le L_f \|\eta_1 - \eta_2\|_{L^2(0,T;H^2)} 
        \le L_f T^{1/2} B \|\eta_1 - \eta_2\|_{\mathbb{E}(0,T)}.
    \end{align*}
    For the operator term, we obtain
    \begin{align*}
        &\|(A(\eta_1) - A(\eta_2))\eta_1 + (A(\eta_2) - A_0)(\eta_1 - \eta_2)\|_{L^2(H^1)} \\
        &\le L_A \|\eta_1 - \eta_2\|_{C(H^2)} \|\eta_1\|_{L^2(H^3)} + L_A \|\eta_2 - \gamma_0\|_{C(H^2)} \|\eta_1 - \eta_2\|_{L^2(H^3)}.
    \end{align*}
    The first term is bounded by $L_A B \|\eta_1-\eta_2\|_{\mathbb{E}(0,T)} (r + \|\gamma_0^\ast\|_{L^2(H^3)})$.
    The second term is bounded by $L_A \varepsilon_0 \|\eta_1 - \eta_2\|_{\mathbb{E}(0,T)}$.
    By choosing $T$ sufficiently small (to reduce $\|\gamma_0^\ast\|_{L^2(H^3)}$ and the $T^{1/2}$ factor) and $r$ small, we can make the Lipschitz constant of $\mathcal{T}_{\gamma_1}$ less than $1/2$.

    \textit{Step 5: Existence and continuous dependence.}
    By the Banach fixed point theorem, there exists a unique fixed point $\gamma \in \mathbb{B}_{r,T,\gamma_1}$, which is the solution to \eqref{eq:abstract_IVP}.
    Since $\gamma \in \mathbb{E}(0,T)$, it implies $\gamma \in C([0,T]; H^2)$, and by Step 1, $\gamma(t) \in V$.
    The Lipschitz dependence on the initial data follows from standard arguments using the contraction property and the linearity of the equation with respect to initial data differences.
\end{proof}

%%%%%%%%%%%%%%%%%%%%%%%%%%%%%%%%%%%%%%%%%%%%%%%%%%%%%
%%%%%%%%%%%%%%%%%%%%%%%%%%%%%%%%%%%%%%%%%%%%%%%%%%%%%
%%%%%%%%%%%%%%%%%%%%%%%%%%%%%%%%%%%%%%%%%%%%%%%%%%%%%

\subsection{Regularity and instantaneous smoothing}

In this subsection, we prove that the solution $\gamma$ obtained in Theorem~\ref{thm:existence} becomes instantly smooth for $t > 0$.

\begin{theorem}\label{thm:regularity}
    Let $\gamma_0 \in V$. Let $\gamma \in \mathbb{E}(0,T) \cap C([0,T];V)$ be the solution to the problem~\eqref{eq:Langmuir_rewrite} obtained by Theorem~\ref{thm:existence}.
    Then $\gamma$ is smooth in space and time; that is, $\gamma \in C^\infty(\mathbb{S}^1 \times (0, T))$.
\end{theorem}

\begin{proof}
    Let $\gamma_0 \in V$.
    We denote the solution to problem~\eqref{eq:Langmuir_rewrite} by $\gamma_* \in \mathbb{E}(0,T) \cap C([0,T];V)$.
    Since $\gamma_*(\cdot, t) \in V$ for all $t \in [0, T]$, it follows from Proposition~\ref{prop:MR} that $A(\gamma_*(t))$ enjoys the property of maximal $L^2$-regularity on $H^1$ for each $t\in[0,T]$.
    By Theorem~\ref{thm:existence}, there exist $T=T(\gamma_0)>0$ and $\varepsilon=\varepsilon(\gamma_0)>0$ such that for each $\eta \in B_\varepsilon(\gamma_0) \subset V$, problem~\eqref{eq:Langmuir_rewrite} with the initial data $\eta$ has a unique solution $\gamma(\cdot; \eta) \in \mathbb{E}(0,T) \cap C([0,T];V)$.
    
    We employ the ``parameter trick'' (time scaling) to establish time regularity.
    Define the scaled function $\gamma_\lambda : \mathbb{S}^1 \times [0, T/(1+\delta)] \to \mathbb{R}^2$ by 
    \[
        \gamma_\lambda(x,t; \eta) \coloneqq \gamma(x, \lambda t; \eta), \qquad \lambda \in (1-\delta, 1+\delta), \quad \eta \in B_\varepsilon(\gamma_0) \subset V, 
    \]
    where $\delta>0$ is a sufficiently small constant.
    Consider the map 
    \[
        H: (1-\delta, 1+\delta) \times B_\varepsilon(\gamma_0) \times \mathbb{E}(0,T/(1+\delta)) \to L^2(0,T;H^1) \times V
    \]
    defined by
    \[
        H(\lambda, \eta, \psi)(t) \coloneqq \bigl( \partial_t \psi(t) - \lambda A(\psi(t)) \psi(t) - \lambda f(\psi(t)), \psi(0)-\eta \bigr), \qquad t \in [0, T/(1+\delta)].
    \]
    It follows from Propositions~\ref{prop:A_C-inf} and \ref{prop:f} that $A \in C^\infty(V; \mathcal{B}(H^3; H^1))$ and $f \in C^\infty(V; H^1)$.
    Consequently, the map $H$ is of class $C^\infty$. 
    We have $H(1, \gamma_0, \gamma_*)=0$. The Fréchet derivative of $H$ with respect to $\psi$ at $(1, \gamma_0, \gamma_*)$ is given by
    \begin{align*}
        \mathcal{D}_\psi H(1, \gamma_0, \gamma_*) [\varphi] = ( \partial_t \varphi - A(\gamma_*) \varphi - \mathcal{D}A(\gamma_*)[\varphi]\gamma_\ast - \mathcal{D}f(\gamma_*)[\varphi], \varphi(0)),
    \end{align*}
    for $\varphi \in \mathbb{E}(0,T/(1+\delta))$.
    Here, the Fréchet derivative of $A$ is explicitly given by
    \begin{align*}
        \mathcal{D}A(\psi)[\varphi]\psi 
        \coloneqq \left. \frac{\mathrm{d}}{\mathrm{d}\varepsilon} A(\psi + \varepsilon \varphi) \psi \right|_{\varepsilon = 0}
        = -\frac{2}{\pi} \frac{\partial_x \psi \cdot \partial_x \varphi}{|\partial_x \psi|^4}\partial_x^2 \psi.
    \end{align*}

    We claim that the linearized operator $\mathcal{L} \coloneqq \mathcal{D}_\psi H(1, \gamma_0, \gamma_*)$ is an isomorphism from $\mathbb{E}(0,T/(1+\delta))$ to $L^2(0,T/(1+\delta);H^1) \times V$.
    Consider the equation $\mathcal{L}\varphi = (g, \eta)$. This is equivalent to finding $\varphi$ such that
    \begin{equation} \label{eq:reg-eq}
        \begin{dcases*}
            \partial_t \varphi - A(\gamma_\ast)\varphi = \mathcal{D}A(\gamma_\ast)[\varphi]\gamma_\ast + \mathcal{D}f(\gamma_\ast)[\varphi] + g & in $(0,T/(1+\delta)]$,\\
            \varphi(0) = \eta.
        \end{dcases*}
    \end{equation}
    To prove the unique solvability, we treat the right-hand side terms involving $\varphi$ as perturbations.
    Define the mapping $\mathcal{K}$ on $\mathbb{E}(0,T/(1+\delta))$ by letting $\mathcal{K}(\varphi)$ be the solution $v$ to
    \begin{equation*}
        \partial_t v - A(\gamma_\ast)v = \mathcal{D}A(\gamma_\ast)[\varphi]\gamma_\ast + \mathcal{D}f(\gamma_\ast)[\varphi] + g, \quad v(0) = \eta.
    \end{equation*}
    Since $A(\gamma_\ast)$ generates maximal $L^2$-regularity (Proposition~\ref{prop:MR}), this map is well-defined.
    We show that $\mathcal{K}$ is a contraction mapping if the time interval is sufficiently small.
    Let $\varphi_1, \varphi_2 \in \mathbb{E}(0,T/(1+\delta))$. Then $\psi \coloneqq \mathcal{K}(\varphi_1) - \mathcal{K}(\varphi_2)$ satisfies
    \begin{align*}
        \partial_t \psi - A(\gamma_\ast)\psi = \mathcal{D}A(\gamma_\ast)[\varphi_1 - \varphi_2]\gamma_\ast + \mathcal{D}f(\gamma_\ast)[\varphi_1 - \varphi_2], \quad \psi(0) = 0.
    \end{align*}
    Applying the maximal regularity estimate \eqref{eq:estimate_MR}, we have
    \begin{align*}
        \|\psi\|_{\mathbb{E}} \le C_{\mathrm{MR}} \left( \|\mathcal{D}A(\gamma_\ast)[\varphi_1 - \varphi_2]\gamma_\ast\|_{L^2(H^1)} + \|\mathcal{D}f(\gamma_\ast)[\varphi_1 - \varphi_2]\|_{L^2(H^1)} \right).
    \end{align*}
    We estimate the perturbation terms.
    From the explicit form of $\mathcal{D}A$, using the algebra property of $H^1$ and the embedding $H^2 \hookrightarrow C^1$, we have
    \begin{align*}
        \|\mathcal{D}A(\gamma_\ast)[\varphi]\gamma_\ast\|_{H^1}
        &= \left\| -\frac{2}{\pi} \frac{\partial_x \gamma_\ast \cdot \partial_x \varphi}{|\partial_x \gamma_\ast|^4} \partial_x^2 \gamma_\ast \right\|_{H^1} \\
        &\le C \left( \|\gamma_\ast\|_{H^3} \|\varphi\|_{C^1} + \|\gamma_\ast\|_{C^1} \|\varphi\|_{H^2} \|\gamma_\ast\|_{H^2} \right) \\
        &\le C \|\gamma_\ast\|_{H^3} \|\varphi\|_{H^2},
    \end{align*}
    where the constant $C$ depends on $\|\gamma_*\|_{C([0,T]; H^2)}$.
    Therefore, integrating over time $(0, \tau)$ with $\tau \coloneqq T/(1+\delta)$, we obtain
    \begin{align*}
        \|\mathcal{D}A(\gamma_\ast)[\varphi_1 - \varphi_2]\gamma_\ast\|_{L^2(0, \tau; H^1)}
        &\le C \left( \int_0^\tau \|\gamma_\ast(t)\|_{H^3}^2 \|\varphi_1(t) - \varphi_2(t)\|_{H^2}^2 \, \mathrm{d}t \right)^{1/2} \\
        &\le C \sup_{t \in [0,\tau]} \|\varphi_1(t) - \varphi_2(t)\|_{H^2} \left( \int_0^\tau \|\gamma_\ast(t)\|_{H^3}^2 \, \mathrm{d}t \right)^{1/2} \\
        &\le C B \|\varphi_1 - \varphi_2\|_{\mathbb{E}(0,\tau)} \|\gamma_\ast\|_{L^2(0, \tau; H^3)}.
    \end{align*}
    For the non-local term, the smoothness of $f$ implies $\|\mathcal{D}f(\gamma_\ast)[\varphi]\|_{H^1} \le C \|\varphi\|_{H^2}$. Thus,
    \begin{align*}
        \|\mathcal{D}f(\gamma_\ast)[\varphi_1 - \varphi_2]\|_{L^2(0, \tau; H^1)}
        &\le C \|\varphi_1 - \varphi_2\|_{L^2(0, \tau; H^2)} \\
        &\le C \tau^{1/2} \sup_{t \in [0,\tau]} \|\varphi_1(t) - \varphi_2(t)\|_{H^2} \\
        &\le C B \tau^{1/2} \|\varphi_1 - \varphi_2\|_{\mathbb{E}(0,\tau)}.
    \end{align*}
    Combining these estimates, we obtain
    \begin{align*}
        \|\mathcal{K}(\varphi_1) - \mathcal{K}(\varphi_2)\|_{\mathbb{E}}
        \le C_{\mathrm{MR}} C B \left( \|\gamma_\ast\|_{L^2(0, \tau; H^3)} + \tau^{1/2} \right) \|\varphi_1 - \varphi_2\|_{\mathbb{E}}.
    \end{align*}
    Since $\gamma_\ast \in \mathbb{E}(0,T) \subset L^2(0,T;H^3)$, the norm $\|\gamma_\ast\|_{L^2(0, \tau; H^3)}$ tends to $0$ as $\tau \to 0$.
    Thus, by choosing $\tau = T/(1+\delta)$ sufficiently small, the Lipschitz constant becomes less than $1/2$, proving that $\mathcal{L}$ is invertible.
    
    By the implicit function theorem on Banach spaces, there exist $\tilde{\delta} > 0$, $\tilde{\varepsilon} >0$, and a $C^\infty$-map 
    \[
        \Phi : (1- \tilde{\delta}, 1+ \tilde{\delta}) \times B_{\tilde{\varepsilon}}(\gamma_0) \to \mathbb{E}(0,T/(1+\tilde{\delta}))
    \]
    such that $H(\lambda, \eta, \Phi(\lambda, \eta)) = 0$ near $(1, \gamma_0, \gamma_*)$.
    Uniqueness implies $\Phi(\lambda, \eta) = \gamma_\lambda(\cdot; \eta)$.
    Since the map $(\lambda, \eta) \mapsto \gamma_\lambda(\cdot; \eta)$ is $C^\infty$ into $\mathbb{E}$, and $\mathbb{E} \hookrightarrow C([0,T]; H^2)$, the solution $\gamma(t)$ is smooth with respect to the time variable.
    Specifically, time derivatives of all orders exist and belong to $H^2(\mathbb{S}^1)$ for each $t$.
    
    To recover spatial regularity, we utilize a bootstrap argument based on elliptic regularity.
    Rewrite the equation as an elliptic problem for each fixed $t \in (0, T]$:
    \begin{equation} \label{eq:elliptic}
        \partial_x^2 \gamma = \pi |\partial_x \gamma|^2 \partial_t \gamma - \pi |\partial_x \gamma|^2 f(\gamma).
    \end{equation}
    Since $\gamma(t) \in H^2$ and $\partial_t \gamma(t) \in H^2$, the right-hand side belongs to $H^1$ (recall that $H^1$ is a Banach algebra and $f$ maps $H^2$ to $H^1$).
    Standard elliptic regularity theory then implies $\gamma(t) \in H^3$.
    Proceeding inductively, suppose $\gamma(t) \in H^k$ for some $k \ge 3$. Then the right-hand side is in $H^{k-1}$ (using the smoothness of $f$ and algebra properties), which implies $\gamma(t) \in H^{k+1}$.
    Since we have time derivatives of all orders, we can differentiate \eqref{eq:elliptic} with respect to $t$ to handle the time-derivative terms in higher-order spatial estimates (mixed derivatives).
    By induction, $\gamma(t) \in C^\infty(\mathbb{S}^1)$ for all $t > 0$.
    Thus, $\gamma \in C^\infty(\mathbb{S}^1 \times (0, T))$.
\end{proof}

\subsection{Equivalence with the ILLSS model}
\label{subsec:equivalence}

Having established the existence of a smooth solution to the boundary integral equation, we are now in a position to rigorously prove the equivalence between the boundary integral formulation and the original ILLSS model.
The logic proceeds as follows: the smooth solution of the boundary integral equation allows us to construct a set of bulk and surface fields.
We verify that these fields satisfy the regularity and decay requirements of Definition~\ref{def:classical}.
Consequently, by the uniqueness result established in Theorem~\ref{thm:uniqueness}, these fields constitute the unique classical solution to the ILLSS model.

\begin{theorem}
    \label{thm:existence_ILLSS}
    Let $\gamma \in C^\infty(\mathbb{S}^1 \times (0, T))$ be the solution to the boundary integral equation \eqref{eq:BIE} obtained in Theorems~\ref{thm:existence} and \ref{thm:regularity}.
    Define the surface velocity $u$ on $\partial B$ by the integral representation \eqref{eq:velocity_field_conv} derived in Section~\ref{sec:ILLSS-to-BIE}.
    Define the surface pressure $p$ via the relation $\nabla p = -\Lambda_{\mathrm{DN}}u$ in $\partial B \setminus \Gamma$.
    Construct the bulk velocity $v$ and pressure $q$ in $B$ via the Poisson integral as in Proposition~\ref{prop:Laplace_solves_Stokes} (with $q \equiv 0$).
    Then, the quadruplet $(v, q, u, p)$ constitutes the unique classical solution to the ILLSS model \eqref{eq:Stokes_subfluid}--\eqref{eq:jump_condition} in the sense of Definition~\ref{def:classical}.
\end{theorem}

\begin{proof}
    \textit{Step 1: Construction and regularity.}
    From Theorem~\ref{thm:regularity}, we have established that the interface evolution $\gamma(t)$ is $C^\infty$ for $t > 0$.
    Consequently, the curve $\Gamma(t)$ is smooth, and the geometric quantities such as the curvature $\kappa$ and the unit normal $\nu$ are all smooth ($C^\infty$) functions on $\Gamma$.
    
    The surface velocity $u$ is defined by the convolution of the fundamental solution $E$ with the force density $-\kappa\nu$ supported on $\Gamma$:
    \[
        u(x) = -\int_\Gamma E(x-y) \kappa(y)\nu(y) \,\mathrm{d}s(y), \quad x \in \partial B.
    \]
    This integral representation corresponds to a hydrodynamic \textit{single-layer potential}.
    According to the classical theory of hydrodynamic potentials (see, e.g., Ladyzhenskaya~\cite[Chapter~3, Section~2]{ladyzhenskaya1969mathematical}), a single-layer potential with a continuous density is continuous throughout the entire space.
    This continuity property is fully consistent with the jump condition $\llbracket u \rrbracket = 0$ derived in Proposition~\ref{prop:velocity-limit}.
    Thus, the condition $\llbracket u \cdot \nu \rrbracket = 0$ in Definition~\ref{def:classical}~(ii) is satisfied.

    Regarding higher regularity, since the density $-\kappa\nu$ and the curve $\Gamma$ are of class $C^\infty$, standard results in potential theory (regularity of potentials up to the boundary) ensure that $u$ admits smooth extensions to the boundary from each side.
    Specifically, all derivatives of $u$ possess well-defined limiting values as one approaches $\Gamma$ from $\Omega_\pm$, although they may exhibit jumps across the interface.
    Consequently, $u$ belongs to the class
    \begin{align*}
        u \in C^\infty(\overline{\Omega_-}; \mathbb{R}^2) \cap C^\infty(\overline{\Omega_+}; \mathbb{R}^2).
    \end{align*}
    This implies $u \in C^1(\overline{\Omega_\pm})$, satisfying Definition~\ref{def:classical}~(i).

    The surface pressure $p$ is determined by the balance of tangential stress derived in the model derivation: $\nabla p = -\partial_{x_3} v_\parallel|_{\partial B}$ in the domains $\Omega_\pm$, where $v_\parallel$ denotes the tangential component of the bulk velocity (constructed below).
    As shown below, since the bulk velocity $v$ is smooth up to the boundary, its tangential derivative defining $p$ is also smooth.
    Thus, the pressure $p$ satisfies:
    \begin{align*}
        p \in C^\infty(\overline{\Omega_-}) \cap C^\infty(\overline{\Omega_+}).
    \end{align*}
    The jump condition $\llbracket p \rrbracket = \kappa$ holds by construction (as shown in Section~\ref{sec:ILLSS-to-BIE}).

    The bulk velocity $v$ is constructed as the harmonic extension of $u$ into the lower half-space $B$, and we set the bulk pressure $q \equiv 0$.
    According to Proposition~\ref{prop:Laplace_solves_Stokes}, since the boundary data $u$ is continuous and divergence-free on the surface (as $\llbracket u \cdot \nu \rrbracket = 0$ is satisfied and $\nabla \cdot u = 0$ holds in $\Omega_\pm$), the extension $v$ satisfies $\nabla \cdot v = 0$ in $B$.
    Therefore, the pair $(v, q) = (v, 0)$ satisfies the Stokes equations $-\triangle v + \nabla q = 0$ exactly.
    
    We now verify the regularity of the bulk velocity $v$.
    Inside the domain $B$, $v$ is real analytic.
    Moreover, since the boundary data $u$ is smooth ($C^\infty$) on $\partial B \setminus \Gamma$, standard elliptic regularity theory for the Laplace equation ensures that $v$ is smooth up to the boundary away from the interface:
    \[
        v \in C^\infty(\overline{B} \setminus \Gamma; \mathbb{R}^3) \cap C^0(\overline{B}; \mathbb{R}^3).
    \]
    This satisfies the regularity requirements for the bulk fields.

    \textit{Step 2: Decay and integrability conditions.}
    We verify the decay and finite energy conditions.
    As discussed in Section~\ref{sec:ILLSS-to-BIE}, the closed nature of the interface $\Gamma$ implies the zero net force condition: $\int_\Gamma \kappa \nu \,\mathrm{d}s = 0$.
    This cancellation of the monopole moment implies that the surface velocity $u(x)$ decays as a dipole field at infinity:
    \[
        |u(x)| = \mathcal{O}(|x|^{-2}) \quad \text{as } |x| \to \infty \text{ on } \partial B.
    \]
    Since $v$ is the harmonic extension of $u$, it inherits this decay rate in the bulk:
    \[
        |v(x)| = \mathcal{O}(|x|^{-2}) \quad \text{as } |x| \to \infty \text{ in } B.
    \]
    The bulk pressure $q \equiv 0$ trivially satisfies the decay condition.
    
    Furthermore, the decay of $v$ implies $|\nabla v(x)| = \mathcal{O}(|x|^{-3})$ at infinity, which ensures square integrability at infinity.
    Regarding the local behavior near the interface, since $u \in H^{1/2}(\partial B)$ (which follows from its smoothness and decay), its harmonic extension $v$ belongs to the energy space with finite Dirichlet integral.
    More specifically, since $v$ is smooth up to $\Gamma$ from both sides, $\nabla v$ is bounded in the neighborhood of $\Gamma$ except possibly for jump discontinuities across the interface extension. Thus:
    \[
        \int_{B} |\nabla v|^2 \,\mathrm{d}x < \infty.
    \]
    Thus, the finite energy condition (Definition~\ref{def:classical}~(iii)) is satisfied.

    \textit{Step 3: Uniqueness.}
    The constructed quadruplet $(v, q, u, p) = (v, 0, u, p)$ satisfies all the conditions of Definition~\ref{def:classical}.
    By Theorem~\ref{thm:uniqueness}, the classical solution to the problem is unique.
    Therefore, the fields constructed from the solution of the boundary integral equation are the unique solution to the ILLSS model.
\end{proof}

%%%%%%%%%%
%%%%%%%%%%
%%%%%%%%%%

%%%%%%%%%%
%%%%%%%%%%
%%%%%%%%%%

\section{Numerical experiments}
\label{sec:numerics}

In this section, we present the results of several numerical experiments.
Throughout, we use the standard (unscaled) curvature notation.
Specifically, we consider a numerical scheme for the following system:
\begin{subnumcases}{}
    \partial_t \gamma = \left(-\frac{1}{\pi} \kappa(\gamma) + F(\gamma,\kappa(\gamma))\right) \nu(\gamma)
    & in $\mathbb{S}^1 \times (0,T)$,\label{eq:Langmuir_original}\\
    \partial_s^2 \gamma = -\kappa(\gamma) \nu(\gamma)
    & in $\mathbb{S}^1 \times (0,T)$,\label{eq:Frenet}\\
    \gamma(\cdot,0) = \gamma_0(\cdot)
    & on $\mathbb{S}^1$.\notag
\end{subnumcases}
Here, the nonlocal term $F(\gamma,\kappa(\gamma))$ is defined by
\begin{equation*}
    F(\gamma,\kappa(\gamma))(x) = -\frac{1}{2\pi} \int_0^1 K(\gamma)(x, y) \kappa(\gamma)(y) |\partial_x \gamma(y)|\, \mathrm{d}y.
\end{equation*}

\subsection{Numerical scheme}

In this subsection, we develop a numerical scheme for the boundary integral equation~\eqref{eq:BIE} using the parametric finite element method~\cite[Chapter~4]{bonito2020geometric}.
For completeness, we provide a detailed derivation of the scheme.

\subsubsection{Weak formulation}

Multiplying both sides of \eqref{eq:Langmuir_original} by $\nu(\gamma)\chi$ with $\chi\in L^2(\mathrm{d}s^\gamma)$ and integrating with respect to $\mathrm{d}s^\gamma$, we obtain
\begin{equation}
    \int_0^1\partial_t \gamma \cdot \nu(\gamma)\chi\,\mathrm{d}s^\gamma = -\frac{1}{\pi} \int_0^1\kappa(\gamma)\chi\,\mathrm{d}s^\gamma + \int_0^1F(\gamma,\kappa(\gamma))\chi\,\mathrm{d}s^\gamma.
    \label{eq:Langmuir_normal-velocity}
\end{equation}
Multiplying both sides of the Frenet formula~\eqref{eq:Frenet} by $\eta \in H^1(\mathrm{d}s^\gamma)^2$ and integrating by parts, we obtain
\begin{equation}
    \int_0^1 \partial_s \gamma \cdot \partial_s \eta\, \mathrm{d}s^\gamma
    = \int_0^1 \kappa(\gamma) \nu(\gamma) \cdot \eta\, \mathrm{d}s^\gamma,
    \label{eq:Frenet_weak}
\end{equation}
where $\mathrm{d}s^\gamma$ denotes the line element along the curve $\gamma$.

Combining \eqref{eq:Langmuir_normal-velocity} with \eqref{eq:Frenet_weak}, we arrive at the following weak formulation:
\begin{equation}
    \begin{dcases*}
        (\partial_t \gamma, \nu(\gamma) \chi)_{\gamma} + \frac{1}{\pi}(\kappa(\gamma), \chi)_{\gamma} - (F(\gamma,\kappa(\gamma)), \chi)_{\gamma} = 0 & for all $\chi \in L^2(\mathrm{d}s^\gamma)$, \\
        (\kappa(\gamma) \nu(\gamma), \eta)_{\gamma} - (\partial_s \gamma, \partial_s \eta)_{\gamma} = 0 & for all $\eta \in H^1(\mathrm{d}s^\gamma)^2$,
    \end{dcases*}
    \label{eq:Langmuir-weak}
\end{equation}
where $(\cdot, \cdot)_{\gamma} \colon L^2(\mathrm{d}s^\gamma) \times L^2(\mathrm{d}s^\gamma) \to \mathbb{R}$ denotes the $L^2$ inner product along the curve $\gamma$ with respect to the arclength measure:
\begin{equation*}
    (\mathsf{f}, \mathsf{g})_{\gamma}
    \coloneqq \int_0^1 \mathsf{f} \cdot \mathsf{g}\, \mathrm{d}s^\gamma
    = \int_0^1 \mathsf{f}(x) \cdot \mathsf{g}(x)\, |\partial_x \gamma(x)|\, \mathrm{d}x,
    \qquad \mathsf{f}, \mathsf{g} \in L^2(\mathrm{d}s^\gamma).
\end{equation*}

\subsubsection{Temporal discretization}

Let $0 = t_0 < t_1 < \cdots < t_{M-1} < t_M = T$ be a time mesh, and define $\tau_m \coloneqq t_{m+1} - t_m$ for $m = 0, 1, \ldots, M-1$.
We denote the curve at time $t = t_m$ by $\gamma^m$.
Furthermore, we abbreviate $\mathsf{g}(\gamma^m)$ as $\mathsf{g}^m$.

We discretize the weak formulation \eqref{eq:Langmuir-weak} in time using a semi-implicit scheme, resulting in the following system:
\begin{equation}
    \begin{dcases}
        \left( \frac{\gamma^{m+1} - \gamma^m}{\tau_m}, \nu^m \chi \right)_{\gamma^m}
        + \frac{1}{\pi} \left( \kappa^{m+1}, \chi \right)_{\gamma^m}
        - \left( F(\gamma^m,\kappa^{m+1}), \chi \right)_{\gamma^m} = 0
        & \text{for all } \chi \in L^2(\mathrm{d}s^{\gamma^m}), \\
        \left( \kappa^{m+1} \nu^m, \eta \right)_{\gamma^m}
        - \left( \partial_{s,\gamma^m} \gamma^{m+1}, \partial_{s,\gamma^m} \eta \right)_{\gamma^m}
        = 0
        & \text{for all } \eta \in H^1(\mathrm{d}s^{\gamma^m})^2,
    \end{dcases}
    \label{eq:Langmuir-weak_time-discrete}
\end{equation}
where $\partial_{s,\gamma^m}$ denotes the arclength derivative with respect to the curve $\gamma^m$. Note that this scheme leads to a linear system for the unknowns $(\gamma^{m+1}, \kappa^{m+1})$.

\subsubsection{Spatial discretization}

Let $0 < x_1 < x_2 < \cdots < x_N \le 1$ be a spatial mesh.
Define the space of piecewise linear functions by
\begin{equation*}
    \mathbb{P}_1^h \coloneqq \left\{ \chi \in C(\mathbb{S}^1) \relmiddle| \text{$\chi|_{[x_{j-1}, x_j]}$ is affine for each } j = 1, 2, \ldots, N \right\},
\end{equation*}
where we set $x_0\coloneqq x_N-1$.
We approximate $\gamma^m$ and $\kappa^{m+1}$ by $\gamma_h^m \in (\mathbb{P}_1^h)^2$ and $\kappa_h^{m+1} \in \mathbb{P}_1^h$, respectively.
Then the unit outward normal vector $\nu_h^m = \nu(\gamma_h^m)$ can be explicitly computed as
\begin{equation*}
    \nu_h^m = \nu_j^m = \mathcal{R} \frac{\gamma_h^m(x_j) - \gamma_h^m(x_{j-1})}{|\gamma_h^m(x_j) - \gamma_h^m(x_{j-1})|} \qquad \text{on } (x_{j-1}, x_j)
\end{equation*}
for each $j$.

For any functions $\mathsf{f}, \mathsf{g} \colon \mathbb{S}^1 \to \mathbb{R}^2$ such that the one-sided limits $\lim_{x \searrow x_{j-1}} \mathsf{f}(x)$ and $\lim_{x \nearrow x_j} \mathsf{f}(x)$ exist for all $j$, we define a mass-lumped inner product
\begin{equation*}
    (\mathsf{f}, \mathsf{g})_m^h \coloneqq \frac{1}{2} \sum_{j=1}^N \left( (\mathsf{f} \cdot \mathsf{g})(x_{j-1}^+) + (\mathsf{f} \cdot \mathsf{g})(x_j^-) \right) r_j^m,
\end{equation*}
where $r_j^m \coloneqq |\gamma_h^m(x_j) - \gamma_h^m(x_{j-1})|$ denotes the length of the $j$-th edge $\gamma_h^m([x_{j-1}, x_j])$, and
\begin{equation*}
    \mathsf{f}(x_{j-1}^+) \coloneqq \lim_{h \searrow 0} \mathsf{f}(x_{j-1} + h), \qquad
    \mathsf{f}(x_j^-) \coloneqq \lim_{h \searrow 0} \mathsf{f}(x_j - h).
\end{equation*}
Then, we approximate $F(\gamma_h^m,\kappa_h^{m+1})$ using the mass-lumped inner product as
\begin{align*}
    F(\gamma_h^m,\kappa_h^{m+1})(x_j^\pm)
    &=-\frac{1}{2\pi}\left(K(\gamma_h^m)(x_j^\pm,\cdot),\kappa_h^{m+1}\right)_{\gamma_h^m} \\
    &\approx-\frac{1}{2\pi}\left(K(\gamma_h^m)(x_j^\pm,\cdot),\kappa_h^{m+1}\right)_m^h
    \eqqcolon F_h^m(\kappa_h^{m+1})(x_j^\pm).
\end{align*}
Consequently, we can approximate the term $(F(\gamma_h^m,\kappa_h^{m+1}),\chi_h)_{\gamma_h^m}$ by $(F_h^m(\kappa_h^{m+1}),\chi_h)_m^h$ for $\chi_h\in \mathbb{P}_h^1$.

Under the above settings, we obtain the following spatial discretization of \eqref{eq:Langmuir-weak_time-discrete}:
\begin{subnumcases}{}
    \left( \frac{\gamma^{m+1} - \gamma^m}{\tau_m}, \nu_h^m \chi_h \right)_m^h + \frac{1}{\pi} \left( \kappa_h^{m+1}, \chi_h \right)_m^h - \left( F_h^m(\kappa_h^{m+1}), \chi_h \right)_m^h = 0 & for all $\chi_h \in \mathbb{P}_1^h$,\label{eq:Langmuir-weak_space-time-discrete_1st}\\
    \left( \kappa_h^{m+1} \nu_h^m, \eta_h \right)_m^h - \left( \partial_{s, \gamma_h^m} \gamma_h^{m+1}, \partial_{s, \gamma_h^m} \eta_h \right)_{L^2(\mathrm{d}s^{\gamma_h^m})} = 0 & for all $\eta_h \in (\mathbb{P}_1^h)^2$.\label{eq:Langmuir-weak_space-time-discrete_2nd}
\end{subnumcases}

\subsubsection{Derivation of the linear system}

Let $\{ \phi_k \}_{k=1}^N$ be the canonical basis of $\mathbb{P}_1^h$, that is, $\phi_k \in \mathbb{P}_1^h$ satisfies $\phi_k(x_l) = \delta_{kl}$, where $\delta_{kl}$ denotes the Kronecker delta.
Substituting $\phi_j$ into $\chi_h$ in \eqref{eq:Langmuir-weak_space-time-discrete_1st} and $\eta_j\phi_j$ with $\eta_j\in\mathbb{R}^2$ into $\eta_h$ in \eqref{eq:Langmuir-weak_space-time-discrete_2nd}, we obtain the following linear system:
\begin{align*}
    \begin{pmatrix}
        \dfrac{\tau_m}{\pi}\mathbf{M} + \dfrac{\tau_m}{2\pi}\mathbf{L} & \mathbf{N}^\top\\[1ex]
        \mathbf{N}&-\mathbf{A}
    \end{pmatrix}\begin{pmatrix}
        \bm{\kappa}^{m+1}\\
        \bm{\gamma}^{m+1}
    \end{pmatrix} = \begin{pmatrix}
        \mathbf{N}^\top\bm\gamma^m\\
        \bm{0}
    \end{pmatrix},
\end{align*}
where the matrices $\mathbf{L} \in \mathbb{R}^{N\times N}$, $\mathbf{M} \in \mathbb{R}^{N\times N}$, $\mathbf{N} \in (\mathbb{R}^2)^{N\times N}$, and $\mathbf{A} \in(\mathbb{R}^{2\times 2})^{N\times N}$ are defined as follows:
\begin{align*}
    L_{jk} &\coloneqq \begin{dcases*}
        \frac{\left(\mathcal{R}\left(\frac{\gamma_{j+1}^m-\gamma_{j-1}^m}{2}\right)\cdot(\gamma_k^m-\gamma_j^m)\right)\left(\mathcal{R}\left(\frac{\gamma_{k+1}^m-\gamma_{k-1}^m}{2}\right)\cdot(\gamma_k^m-\gamma_j^m)\right)}{|\gamma_k^m-\gamma_j^m|^3} & if $j\neq k$,\\
        0 & if $j=k$,
    \end{dcases*}\\
    M_{jk} &\coloneqq \begin{dcases*}
        \frac{r_{j+1}^m+r_j^m}{2} & if $j=k$,\\
        0 & if $j\neq k$,
    \end{dcases*}\\
    N_{jk} &\coloneqq \begin{dcases*}
        \mathcal{R}\left(\frac{\gamma_{j+1}^m - \gamma_{j-1}^m}{2}\right) & if $j=k$,\\
        0 & if $j\neq k$,
    \end{dcases*}\\
    A_{jk} &\coloneqq \begin{dcases*}
        -\frac{1}{r_j^m}I_2 & if $j-k\equiv1\pmod N$,\\
        \left(\frac{1}{r_j^m}+\frac{1}{r_{j+1}^m}\right)I_2 & if $j=k$,\\
        -\frac{1}{r_{j+1}^m}I_2 & if $j-k \equiv -1 \pmod N$,
    \end{dcases*}
\end{align*}
with $\bm{\gamma}^m = (\gamma_j^m)_{j=1}^N \in (\mathbb{R}^2)^N$ and $\bm{\kappa}^m = (\kappa_j^m)_{j=1}^N\in\mathbb{R}^N$.
Here, we set $r_{N+1}^m \coloneqq r_1^m$, and $I_2$ denotes the identity matrix of size $2$.

\subsection{Results}

In this subsection, we present several numerical experiments to investigate the dynamics of the Langmuir model.
Throughout all computations, the parameters are set as follows:
\begin{equation*}
    N=200,\qquad
    h=\frac{1}{N},\qquad
    \tau_m\equiv\tau=10^{-2}.
\end{equation*}
The spatial mesh is defined by $\{x_j\}_{j=1}^N$ with $x_j=(j-1/2)h$ for $j=1,\ldots,N$.

We consider an initial curve defined as
\begin{equation}
    \gamma(u)
    =\begin{pmatrix}
        \gamma_1(u)\\
        \gamma_2(u)
    \end{pmatrix}
    =\begin{dcases*}
        \begin{pmatrix}
            \overline{\gamma}_1(u)\\
            \overline{\gamma}_2(u)
        \end{pmatrix} & for $u\in[0,1/2)$,\\
        \begin{pmatrix}
            -\overline{\gamma}_1(u-1/2)\\
            -\overline{\gamma}_2(u-1/2)
        \end{pmatrix} & for $u\in[1/2,1]$,
    \end{dcases*}
    \label{eq:bola}
\end{equation}
where
\begin{align*}
    \begin{pmatrix}
        \overline{\gamma}_1(u)\\
        \overline{\gamma}_2(u)
    \end{pmatrix}
    &=\begin{dcases*}
        \begin{pmatrix}
            \hat{\gamma}_1(u)\\
            \hat{\gamma}_2(u)
        \end{pmatrix} & for $u\in[0,1/4)$,\\
        \begin{pmatrix}
            -\hat{\gamma}_1(1/2-u)\\
            \hat{\gamma}_2(1/2-u)
        \end{pmatrix} & for $u\in[1/4,1/2]$,
    \end{dcases*}\\
    \hat{\gamma}_1(u)&=\begin{dcases*}
        5+\rho(1+\cos(8(\pi-\theta_\varepsilon)u)) & for $u\in[0,1/8)$,\\
        2\left(5+\rho(1-\cos\theta_\varepsilon)\right)(1-4u) & for $u\in[1/8,1/4]$,
    \end{dcases*}\\
    \hat{\gamma}_2(u)&=\begin{dcases*}
        \rho\sin(8(\pi-\theta_\varepsilon)u) & for $u\in[0,1/8)$,\\
        \varepsilon & for $u\in[1/8,1/4]$,
    \end{dcases*}
\end{align*}
with $\varepsilon=1/5$, $\rho=2/5$, and $\theta_\varepsilon\coloneqq\arcsin(\varepsilon/\rho)$.
This curve has a ``bola'' shape, as studied in the original work~\cite{alexander2007domain} (see Figure~\ref{fig:bola}).

\begin{figure}[htbp]
    \begin{minipage}{.33\hsize}
        \centering
        \includegraphics[width=\hsize]{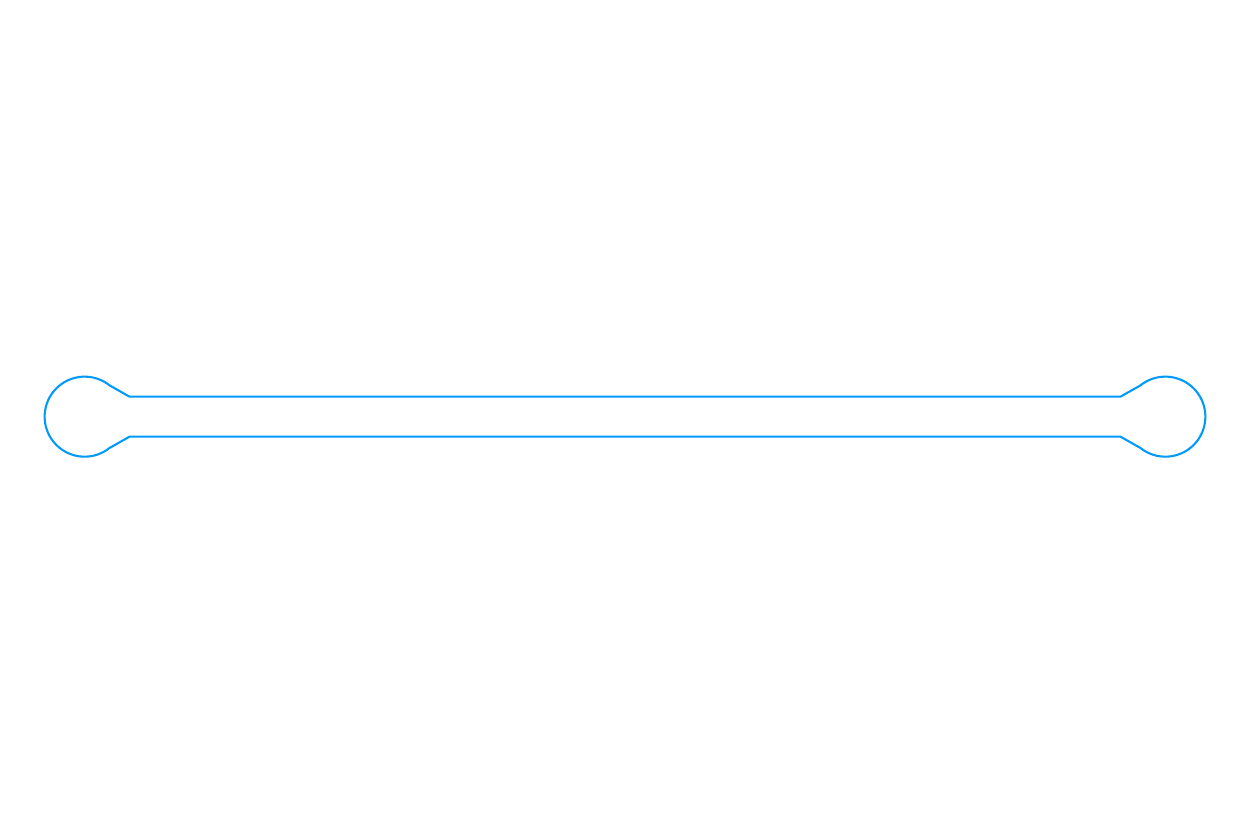}
        $t = 0$
    \end{minipage}%
    \begin{minipage}{.33\hsize}
        \centering
        \includegraphics[width=\hsize]{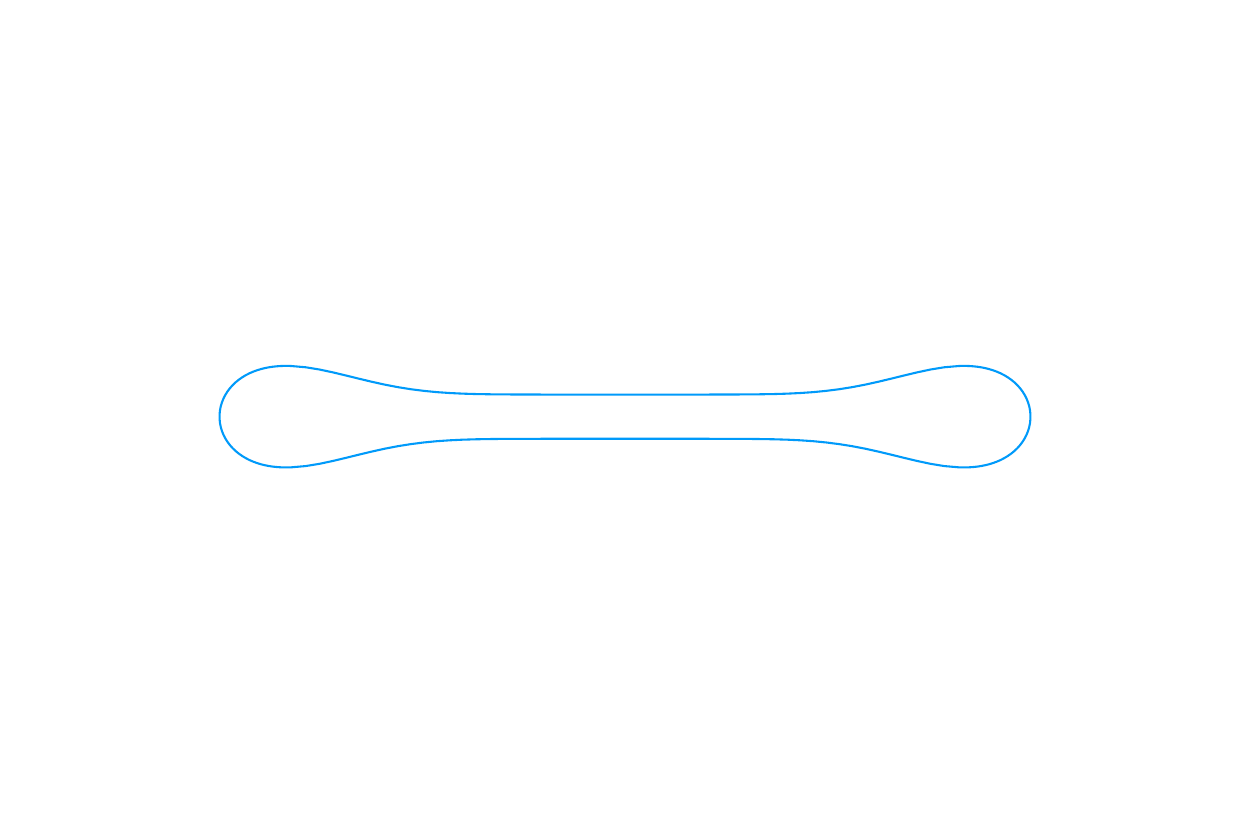}
        $t = 300\tau$
    \end{minipage}%
    \begin{minipage}{.33\hsize}
        \centering
        \includegraphics[width=\hsize]{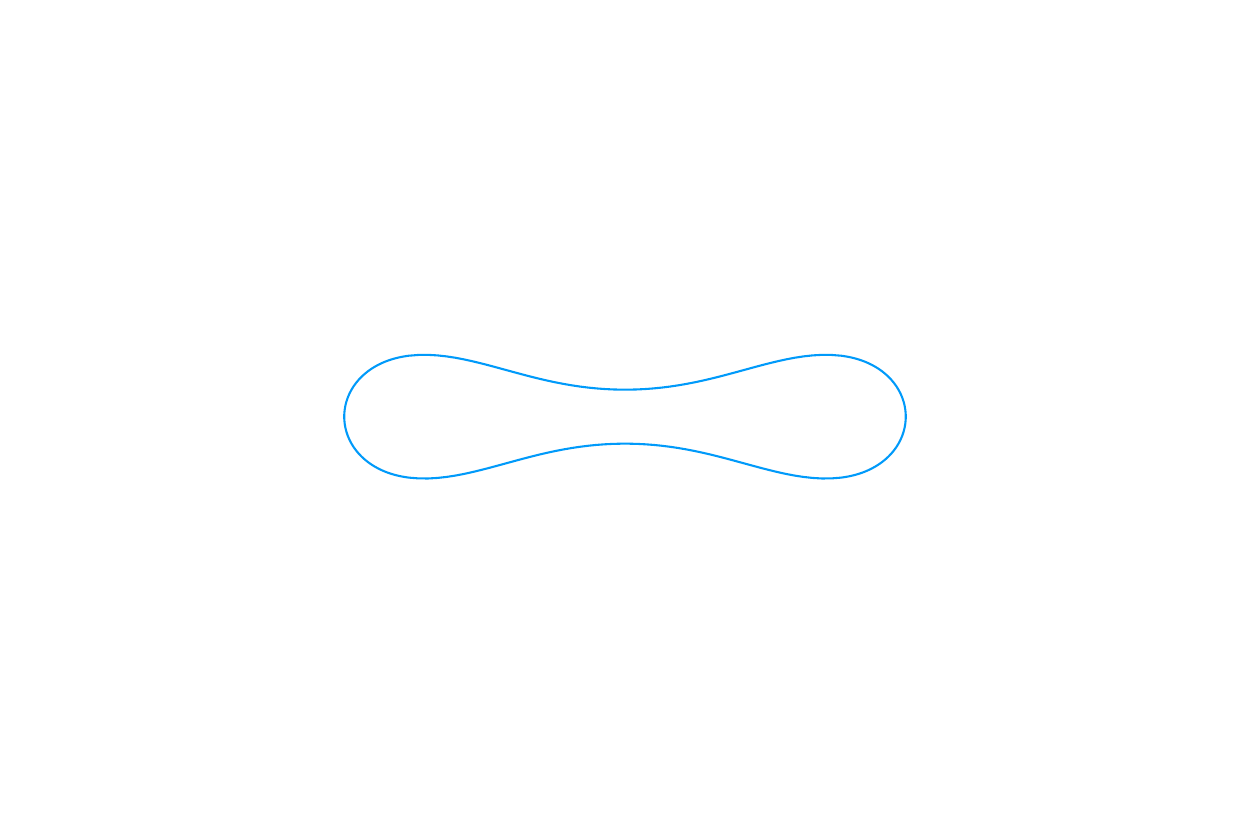}
        $t = 600\tau$
    \end{minipage}%

    \begin{minipage}{.33\hsize}
        \centering
        \includegraphics[width=\hsize]{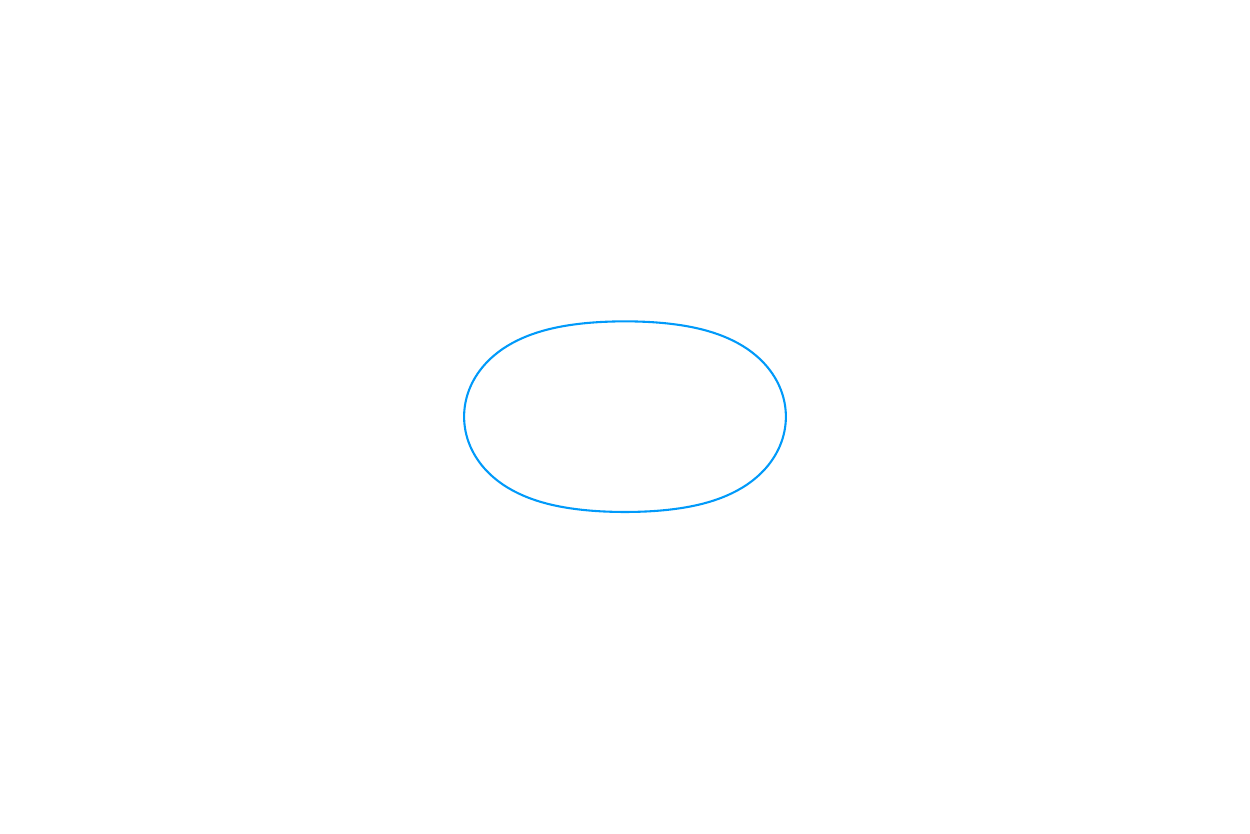}
        $t = 1000\tau$
    \end{minipage}%
    \begin{minipage}{.33\hsize}
        \centering
        \includegraphics[width=\hsize]{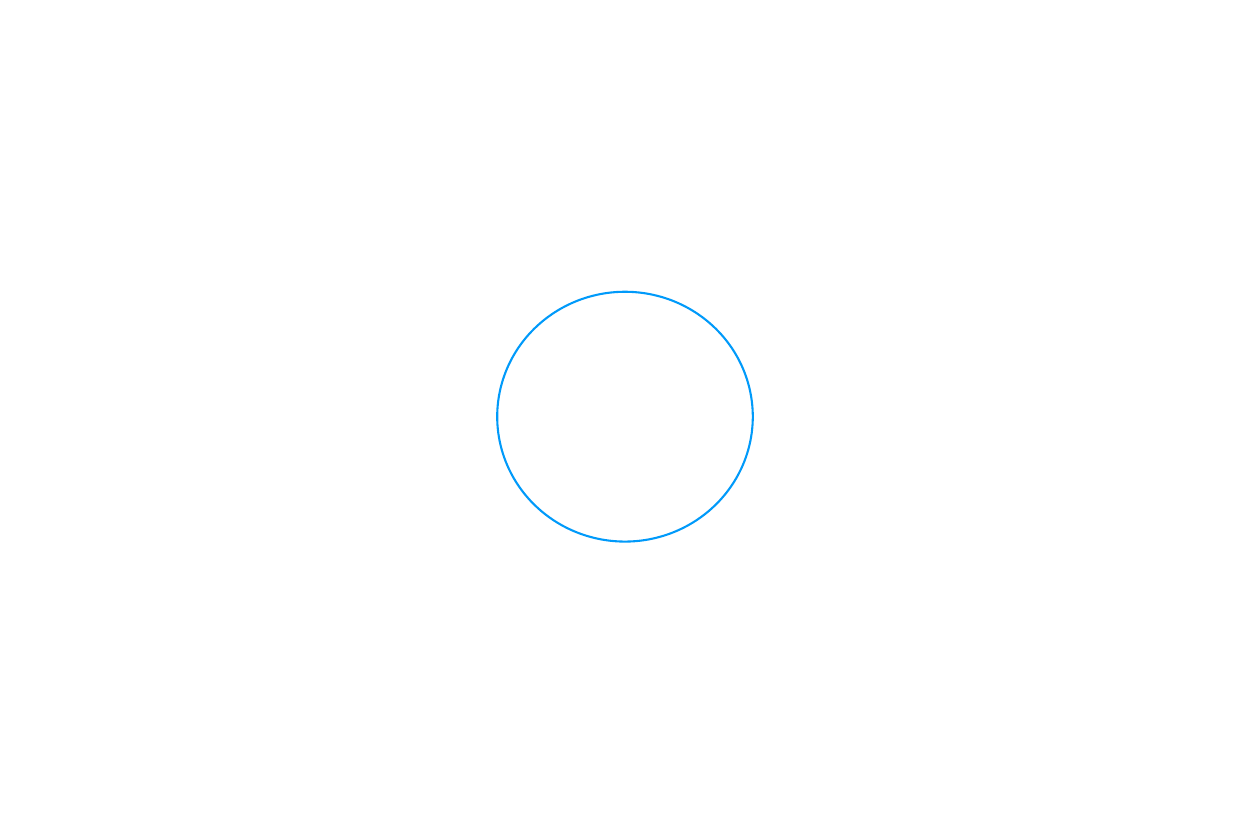}
        $t = 1500\tau$
    \end{minipage}%
    \begin{minipage}{.33\hsize}
        \centering
        \includegraphics[width=\hsize]{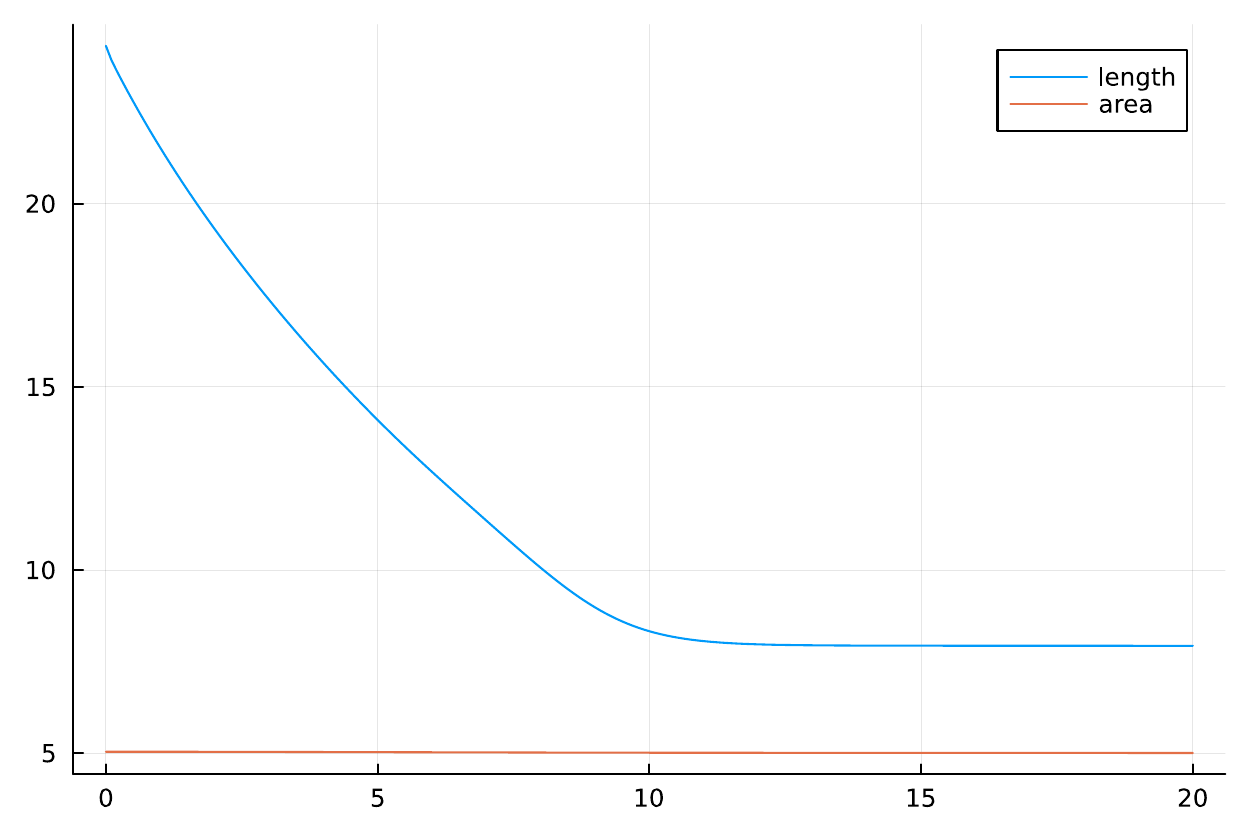}
        Time evolution of length and area
    \end{minipage}
    \caption{Numerical result with initial curve~\eqref{eq:bola}.}
    \label{fig:bola}
\end{figure}

The numerical results are displayed in Figure~\ref{fig:bola}.
As time progresses, the curve contracts gradually without rupture, eventually becomes convex, and converges to a circle with the same area as the initial shape.
The monotonically decreasing length is consistent with Proposition~\ref{prop:CS}, which states that the Langmuir model possesses the curve-shortening property.
Furthermore, since the underlying fluid is incompressible, the area enclosed by the curve is preserved throughout the evolution.
Therefore, the observed results are fully consistent with the theoretical properties of the model.
Similar behavior is also reported in~\cite{alexander2007domain}, suggesting that the proposed numerical scheme is working appropriately.

%%%%%%%%%%
%%%%%%%%%%
%%%%%%%%%%

\bibliographystyle{siam}
\bibliography{Langmuir}

\begin{thebibliography}{10}

\bibitem{alexander2007domain}
{\sc J.~C. Alexander, A.~J. Bernoff, E.~K. Mann, J.~A. Mann, Jr., J.~R. Wintersmith, and L.~Zou}, {\em Domain relaxation in {L}angmuir films}, J. Fluid Mech., 571 (2007), pp.~191--219.

\bibitem{blodgett1935films}
{\sc K.~B. Blodgett}, {\em Films built by depositing successive monomolecular layers on a solid surface}, J. Am. Chem. Soc., 57 (1935), pp.~1007--1022.

\bibitem{bonito2020geometric}
{\sc A.~Bonito and R.~H. Nochetto}, eds., {\em Geometric partial differential equations. {P}art {I}}, vol.~21 of Handbook of Numerical Analysis, Elsevier/North-Holland, Amsterdam, [2020] \copyright2020.

\bibitem{caffarelli2007extension}
{\sc L.~Caffarelli and L.~Silvestre}, {\em An extension problem related to the fractional laplacian}, Commun. Partial Differ. Equ., 32 (2007), pp.~1245--1260.

\bibitem{chen1993hele}
{\sc X.~Chen}, {\em The {H}ele-{S}haw problem and area-preserving curve-shortening motions}, Arch. Rational Mech. Anal., 123 (1993), pp.~117--151.

\bibitem{constantin1993global}
{\sc P.~Constantin and M.~Pugh}, {\em Global solutions for small data to the {H}ele-{S}haw problem}, Nonlinearity, 6 (1993), pp.~393--415.

\bibitem{deSimon1964Un’applicazione}
{\sc L.~de~Simon}, {\em Un'applicazione della teoria degli integrali singolari allo studio delle equazioni differenziali lineari astratte del primo ordine}, Rend. Sem. Mat. Univ. Padova, 34 (1964), pp.~205--223.

\bibitem{dynarowicz2024advantages}
{\sc P.~Dynarowicz-Latka, A.~Wnetrzak, and A.~Chachaj-Brekiesz}, {\em Advantages of the classical thermodynamic analysis of single—and multi-component langmuir monolayers from molecules of biomedical importance—theory and applications}, J. R. Soc. Interface, 21 (2024), p.~20230559.

\bibitem{escher1997classical}
{\sc J.~Escher and G.~Simonett}, {\em Classical solutions for {H}ele-{S}haw models with surface tension}, Adv. Differential Equations, 2 (1997), pp.~619--642.

\bibitem{escher1998volume}
\leavevmode\vrule height 2pt depth -1.6pt width 23pt, {\em The volume preserving mean curvature flow near spheres}, Proc. Amer. Math. Soc., 126 (1998), pp.~2789--2796.

\bibitem{evans2010partial}
{\sc L.~C. Evans}, {\em Partial differential equations}, vol.~19 of Graduate Studies in Mathematics, American Mathematical Society, Providence, RI, second~ed., 2010.

\bibitem{folland1995introduction}
{\sc G.~B. Folland}, {\em Introduction to partial differential equations}, Princeton University Press, 1995.

\bibitem{friedman1964partial}
{\sc A.~Friedman}, {\em Partial differential equations of parabolic type}, Prentice-Hall, Inc., Englewood Cliffs, NJ, 1964.

\bibitem{gage1986area}
{\sc M.~Gage}, {\em On an area-preserving evolution equation for plane curves}, in Nonlinear problems in geometry ({M}obile, {A}la., 1985), vol.~51 of Contemp. Math., Amer. Math. Soc., Providence, RI, 1986, pp.~51--62.

\bibitem{henon1991microscope}
{\sc S.~H{\'e}non and J.~Meunier}, {\em Microscope at the brewster angle: Direct observation of first-order phase transitions in monolayers}, Rev. Sci. Instrum., 62 (1991), pp.~936--939.

\bibitem{hopper1990plane}
{\sc R.~W. Hopper}, {\em Plane {S}tokes flow driven by capillarity on a free surface}, J. Fluid Mech., 213 (1990), pp.~349--375.

\bibitem{huisken1987volume}
{\sc G.~Huisken}, {\em The volume preserving mean curvature flow}, J. Reine Angew. Math., 382 (1987), pp.~35--48.

\bibitem{ladyzhenskaya1969mathematical}
{\sc O.~A. Ladyzhenskaya}, {\em The Mathematical Theory of Viscous Incompressible Flow}, Gordon and Breach, New York, 2nd~ed., 1969.

\bibitem{langmuir1917constitution}
{\sc I.~Langmuir}, {\em The constitution and fundamental properties of solids and liquids. ii. liquids.}, J. Am. Chem. Soc., 39 (1917), pp.~1848--1906.

\bibitem{lin2024controllable}
{\sc H.~Lin, Y.~Zheng, C.~Zhong, L.~Lin, K.~Yang, Y.~Liu, H.~Hu, and F.~Li}, {\em Controllable-assembled functional monolayers by the langmuir--blodgett technique for optoelectronic applications}, J. Mater. Chem. C, 12 (2024), pp.~1177--1210.

\bibitem{lions1972non}
{\sc J.~L. Lions and E.~Magenes}, {\em Non-Homogeneous Boundary Value Problems and Applications}, vol.~1 of Die Grundlehren der mathematischen Wissenschaften, Springer-Verlag, Berlin, 1972.
\newblock Translated from the French by P. Kenneth.

\bibitem{lischke2020what}
{\sc A.~Lischke, G.~Pang, M.~Gulian, F.~Song, C.~Glusa, X.~Zheng, Z.~Mao, W.~Cai, M.~M. Meerschaert, M.~Ainsworth, and G.~E. Karniadakis}, {\em What is the fractional laplacian? a comparative review with new results}, Journal of Computational Physics, 404 (2020), p.~109009.

\bibitem{lunardi2009interpolation}
{\sc A.~Lunardi}, {\em Interpolation theory}, Appunti. Scuola Normale Superiore di Pisa (Nuova Serie). [Lecture Notes. Scuola Normale Superiore di Pisa (New Series)], Edizioni della Normale, Pisa, second~ed., 2009.

\bibitem{olafsen2010experimental}
{\sc J.~Olafsen}, {\em Experimental and computational techniques in soft condensed matter physics}, Cambridge University Press, 2010.

\bibitem{pereira2021recent}
{\sc A.~R. Pereira and O.~N.~d. Oliveira~Junior}, {\em Recent advances in the use of langmuir monolayers as cell membrane models}, Ecl{\'e}tica Qu{\'\i}mica, 46 (2021), pp.~18--29.

\bibitem{prokert1995existence}
{\sc G.~Prokert}, {\em On the existence of solutions in plane quasistationary {S}tokes flow driven by surface tension}, vol.~6, 1995, pp.~539--558.
\newblock Complex analysis and free boundary problems (St.\ Petersburg, 1994).

\bibitem{rojewska2021langmuir}
{\sc M.~Rojewska, W.~Smu{\l}ek, E.~Kaczorek, and K.~Prochaska}, {\em Langmuir monolayer techniques for the investigation of model bacterial membranes and antibiotic biodegradation mechanisms}, Membranes, 11 (2021), p.~707.

\bibitem{stein1970singular}
{\sc E.~M. Stein}, {\em Singular integrals and differentiability properties of functions}, vol.~No. 30 of Princeton Mathematical Series, Princeton University Press, Princeton, NJ, 1970.

\end{thebibliography}
\end{document}